\documentclass{amsart}

\usepackage{bbm, amssymb, amsthm, verbatim}
\usepackage{xcolor}
\usepackage[margin=3cm]{geometry}

\allowdisplaybreaks[2]
 
\newcommand{\pD}[1]{\partial_#1}
\newcommand{\npD}[1]{\partial^\perp_#1}

\newcommand{\rD}[2]{\frac{d #1}{d #2}}

\newcommand{\vn}[1]{\lVert#1\rVert}

\newcommand{\sfrac}[2]{\text{\fontsize{5}{5}\selectfont$\frac{#1}{#2}$}}
\newcommand{\IP}[2]{\left< #1 , #2 \right>}

\newcommand{\R}{\ensuremath{\mathbb{R}}}

\newcommand{\N}{\ensuremath{\mathbb{N}}}

\newcommand{\BF}{\ensuremath{\mathbf{F}}}
\newcommand{\BQ}{\ensuremath{\mathbf{Q}}}

\newcommand{\SL}{\ensuremath{\mathcal{L{}}}}

\newcommand{\cg}{\ensuremath{c_{\gamma}}}

\newtheorem{thm}{Theorem}
\newtheorem{cor}[thm]{Corollary}
\newtheorem{prop}[thm]{Proposition}
\newtheorem{lem}[thm]{Lemma}

\theoremstyle{remark}
\newtheorem{rmk}{Remark}

\begin{document}

\title{Concentration-compactness and finite-time singularities for Chen's flow}
\author{Yann Bernard
   \and Glen Wheeler
   \and Valentina-Mira Wheeler$^*$
   }
\thanks{* Corresponding author, \texttt{vwheeler@uow.edu.au}}
\address{Yann Bernard \\
         School of Mathematical Sciences\\
         Level 4 Rainforest Walk Clayton Campus\\
         Monash University\\
         Victoria 3800\\
         Australia\\~\\
Glen Wheeler \& Valentina-Mira Wheeler\\Institute for Mathematics and its Applications\\
         Faculty of Informatics and Engineering\\
         University of Wollongong\\
         Northfields Avenue, 2522\\
         Wollongong, NSW, Australia}

\maketitle

\begin{abstract}
Chen's flow is a fourth-order curvature flow motivated by the spectral
decomposition of immersions, a program classically pushed by B.-Y. Chen since
the 1970s. In curvature flow terms the flow sits at the critical level of scaling
together with the most popular extrinsic fourth-order curvature flow, the Willmore and
surface diffusion flows. Unlike them however the famous Chen conjecture
indicates that there should be no stationary nonminimal data, and so in
particular the flow should drive all closed submanifolds to singularities. We
investigate this idea, proving that (1) closed data becomes extinct in finite
time in all dimensions and for any codimension; (2) singularities are
characterised by concentration of curvature in $L^n$ for intrinsic dimension $n
\in \{2,4\}$ and any codimension (a Lifespan Theorem); and (3) for $n=2$ and in
any codimension, there exists an explicit $\varepsilon_2$ such that if the
$L^2$ norm of the tracefree curvature is initially smaller than
$\varepsilon_2$, the flow remains smooth until it shrinks to a point, and that
the blowup of that point is an embedded smooth round sphere.

\keywords{global differential geometry\and fourth order\and geometric analysis\and biharmonic\and Chen conjecture}
\subjclass{53C44\and 58J35}
\end{abstract}

\section{Introduction}%{{{

Suppose $f:M^n\rightarrow \R^{N}$, $N>n$ is a smooth isometric immersion.
We assume that $M^n$ is closed and complete.
Denote by $\vec{H}$ the mean curvature vector of $f$.
Then
\[
(\Delta f)(p) = \vec{H}(p)
\]
for all $p\in M^n$, where $\Delta$ here refers to the rough Laplacian.
The rough Laplacian is that induced by the connection on the pullback bundle $f^*(T\R^{n+1})$.
Applying the operator again yields
\[
(\Delta^2 f)(p) = (\Delta \vec{H})(p)
\,.
\]
If $\Delta^2f\equiv0$, we call $f$ \emph{biharmonic}.
Chen's conjecture is the statement that $\Delta \vec{H} \equiv 0$ implies
$\vec{H} \equiv 0$.
This conjecture is motivated by Chen's work in the spectral decomposition of
immersed submanifolds.
There has been much activity on the conjecture (see as a sample the recent
papers \cite{b08,d06,luo14,m14,mon06,n14,ou10,ou16,ou12,w14,whe13}
and Chen's recent survey \cite{chen13}), but still it remains open.

In this paper we study the heat flow for $\Delta^2$:
this is a one-parameter family of smooth isometric immersions
$f:M^n\times[0,T)\rightarrow\R^{N}$ satisfying $f(p,0) = f_0(p)$ for a given
smooth isometric immersion $f_0:M^n\rightarrow\R^{N}$ and
\begin{equation}
\label{CF}
\tag{CF}
(\partial_tf)(p,t) = -(\Delta^2 f)(p,t)
\,,
\end{equation}
for all $(p,t) \in M^n\times(0,T)$.
We call \eqref{CF} \emph{Chen's flow} and $f_0$ the initial data.
Since $\Delta^2$ is a fourth-order quasilinear elliptic operator, local
existence and uniqueness for \eqref{CF} is standard.  Details can be
found in \cite[Chapter 3]{bakerthesis}.  See also \cite[Chapter
5]{eidelman1998pbv}, \cite{solonnikov1965bvp} and \cite{shuanhuli}.

\begin{thm}
\label{TMste}
Let $f_0:M^n\rightarrow\R^{N}$ be a smooth closed isometrically immersed submanifold.
There exists a $T\in(0,\infty]$ and unique one-parameter family of smooth
closed isometric immersions $f:M^n\times[0,T)\rightarrow\R^{N}$ such that
\eqref{CF} is satisfied and $T$ is maximal.
\end{thm}
Note that \emph{maximal} above means that there does not exist another family
$\hat{f}:M^n\times[0,\hat{T})\rightarrow\R^{N}$ of smooth closed
isometrically immersed hypersurfaces satisfying \eqref{CF}, $\hat{f}(p,0) =
f_0(p)$ with $\hat{T} > T$.

A simple consequence of the argument used by Jiang \cite{jiang} is that there
are no closed biharmonic submanifolds of Euclidean space.
Therefore it is natural to expect that the flow may only exist for at most finite time, that is, $T<\infty$.
The following result gives a precise estimate, sharp for $n\in\{2,3,4\}$.

\begin{thm}
\label{TMfinite}
Chen's flow $f:M^n\times[0,T)\rightarrow\R^{N}$ with smooth, closed initial
data $f_0:M^n\rightarrow\R^{N}$ has finite maximal time of existence, with
the explicit estimate
\begin{equation}
\label{EQmaxTest}
T \le \frac{\mu(f_0)^\frac{4}{n}}{C_n}\,,
\end{equation}
where for $n\in\{2,3,4\}$ we have $C_n = 4\omega_n^{\frac4n} n^2$, and for $n
>4$ we have $C_n =  \frac{\omega_n^{\frac4n}}{n^24^{4n+3}}$. Here
$\omega_n$ denotes the area of the unit $n$-sphere.
Furthermore, if equality is achieved in \eqref{EQmaxTest}, then
$\mu(f_t)\searrow0$ as $t\nearrow T$.
\end{thm}
\begin{rmk}
Round spheres are driven to points under Chen's flow with
$T=\frac{r_0^4}{4n^2}$, where $r_0$ is the initial radius.
This shows that the estimate \eqref{EQmaxTest} is sharp in dimensions $2, 3$
and $4$.
We expect that the same estimate holds in higher dimensions.
\end{rmk}
Given Theorem \ref{TMfinite}, it is natural to ask for a classification of
finite-time singularities.
For higher-order curvature flow such as Chen's flow, such classifications are
very difficult.
For example, a classification of singular geometries remains well open for the
two most popular extrinsic fourth-order curvature flow, that is, the Willmore
flow and the surface diffusion flow (see for example
\cite{kuwert2001wfs,kuwert2002gfw,kuwert2004rps,MWW10,metzger2013willmore,W11,mySDLTE,mythesis,wheeler2015gap}).

For both the surface diffusion and Willmore flows, the general principle of concentration or
compactness from the classical theory of harmonic map heat flow remains valid.
We are able to obtain a similar result here: We present the following
characterisation of finite-time singularities, also called a
concentration-compactness alternative or lifespan theorem.

\begin{thm}
\label{TMlifespan}
Let $n\in\{2,4\}$.
There exist constants $\varepsilon_1>0$ and $c<\infty$ depending only on $n$ and $N$
with the following property.
Let $f:M^n \times[0,T)\rightarrow \R^{N}$ be a Chen flow with smooth initial
data.

({\bf Case 1: $\mathbf{n=2}$.})
Let $\rho$ be chosen such that
\begin{equation}
\label{EQsmallconcentrationcondition}
\int_{f^{-1}(B_\rho(x))} |A|^2 d\mu\Big|_{t=0} = \varepsilon(x) \le \varepsilon_1
\qquad
\text{ for all $x\in \R^{N}$}\,.
\end{equation}
Then the maximal time $T$ of smooth existence satisfies
\begin{equation*}
T \ge \frac1{c}\rho^4\,,
\end{equation*}
and we have the estimate
\begin{equation*}
\int_{f^{-1}(B_\rho(x))} |A|^2 d\mu \le c\varepsilon_1
\qquad\qquad\qquad\hskip-1mm\text{ for all }
t \in \Big[0, \frac1{c}\rho^4\Big].
\end{equation*}

({\bf Case 2: $\mathbf{n=4}$.})
Let $\rho$ be chosen such that
\begin{equation}
\label{EQsmallconcentrationcondition4}
\int_{f^{-1}(B_\rho(x))} |A|^4+|\nabla A|^2 d\mu\Big|_{t=0} = \varepsilon(x) \le \varepsilon_1
\qquad
\text{ for all $x\in \R^{N}$}\,.
\end{equation}
Then the maximal time $T$ of smooth existence satisfies
\begin{equation*}
T \ge \frac1{c}\rho^4\,,
\end{equation*}
and we have the estimate
\begin{equation}
\int_{f^{-1}(B_\rho(x))} |A|^4+|\nabla A|^2 d\mu \le c\varepsilon_1
\qquad\qquad\qquad\hskip-1mm\text{ for all }
t \in \Big[0, \frac1{c}\rho^4\Big].
\label{eqgoalestn4}
\end{equation}
\end{thm}

\begin{rmk}
Our proof applies to a general class of flows, including the Willmore flow and
the surface diffusion flow.  This is new for the Willmore flow and the surface diffusion flow in
four dimensions (the two dimensional case for the Willmore flow is the main result of
\cite{kuwert2002gfw}, and a corresponding theorem surface diffusion flow is contained in \cite{mySDLTE}).
In three dimensions a lifespan theorem for the surface diffusion flow is known
\cite{mySDLTE}, however the constants $(\varepsilon_1,c)$ there for $n=3$
depend on the measure of the initial data.  Here, new estimates enable our constants to be universal.

The main result of \cite{kuwert2002gfw} and the lifespan theorems from
\cite{MWW10,MW16,W11,mySDLTE,mythesis} (assuming the external force vanishes
identically) are generalised by our work here.
See Theorem \ref{TMgenerallifespan} for a precise statement.
\end{rmk}

The concentration phenomenon that Theorem \ref{TMlifespan} guarantees can be seen as follows.
If $\rho(t)$ denotes the largest radius such that
either of the concentration conditions (\eqref{EQsmallconcentrationcondition}
or \eqref{EQsmallconcentrationcondition4}) holds at time $t$, then $\rho(t) \le
\sqrt[4]{c(T-t)}$ and so at least $\varepsilon_1$ of the curvature (or its
derivative if $n=4$) concentrates in a ball $f^{-1}(B_{\rho(T)}(x))$.  That is,
\[
	\begin{cases}\displaystyle
		(n=2)\qquad\lim_{t\rightarrow T} \int_{f^{-1}(B_{\rho(t)}(x))} |A|^2 d\mu \ge \varepsilon_1\,,\\\displaystyle
		(n=4)\qquad\lim_{t\rightarrow T} \int_{f^{-1}(B_{\rho(t)}(x))} |A|^4+|\nabla A|^2 d\mu \ge \varepsilon_1\,,
	\end{cases}
\]
where $x = x(t)$ is understood to be the centre of a ball where the integral
above is maximised.
In either case, this implies that a blowup of such a singularity will be nontrivial.

Although Theorem \ref{TMlifespan} yields a characterisation of finite time
singularities as space-time concentrations of curvature, it
does not give any information at all about the asymptotic geometry of such a
singularity.
One of the simplest observations in this direction is that for spherical
initial data with radius $r_0$, the flow shrinks homothetically 
to a point with maximal time
\[
T = \frac{r_0^4}{4n^2}\,.
\]
As the evolution is homothetic, parabolic rescaling about the space-time
singularity reveals a standard round sphere.
This asymptotic behaviour is called \emph{shrinking to a round point}.

One may therefore hope that this behaviour holds in a neighbourhood of a sphere.
This is our final result of the paper, proved using blowup analysis.

\begin{thm}
\label{TMglobal}
There exists an absolute constant $\varepsilon_2>0$ depending only on $N$ such
that if $f:M^2\times[0,T)\rightarrow\R^N$ is Chen's flow satisfying
\begin{equation}
\int_M|A^o|^2d\mu\bigg|_{t=0} \le \varepsilon_2 < 8\pi
\label{EQsmallnessTMglobal}
\end{equation}
then $T<\infty$, and $f(M^2,t)$ shrinks to a round point as $t\rightarrow T$.
\end{thm}

This paper is organised as follows.
In Section 2 we describe our notation, some fundamental identities, and the
Chen flow in the normal bundle.
Section 3 gives evolution equations and the proof of Theorem \ref{TMfinite}.
Our analysis throughout the paper relies on control obtained via localised integral estimates.
The key tools that facilitate this are the Michael-Simon Sobolev inequality
\cite{michael1973sam} and the divergence theorem.
Section 4 contains the consequences of these that we need here, and proofs of all new statements.
One especially long proof is delayed to the Appendix.
Integral estimates valid along the flow are also proved in Section 4, including
control on the local growth of the $L^n$ norm of $A$.
The section is concluded with a proof of the lifespan theorem.
Section 5 is concerned with global analysis for the flow, and contains a proof
of the monotonicity result for the $L^2$ norm of $A$, and blowup analysis,
ending in the proof of Theorem \ref{TMglobal}.

%}}}

\section{Notation and the normal flow}%{{{

Let us first collect various general formulae from the differential geometry of submanifolds which we need for later analysis.
We use notation similar to that of Kuwert-Sch\"atzle
\cite{kuwert2001wfs,kuwert2002gfw,kuwert2004rps}, Hamilton \cite{RH} and
Huisken \cite{huisken1984fmc,huisken86riemannian}.
We have as our principal object of study a smooth isometric immersion
$f:M^n\rightarrow\R^{N}$ of a Riemannian manifold $(M^n,g) =
(M^n,f^*\delta^{\R^{N}})$ into $\R^{N}$.

The induced metric has components
\begin{equation}
g_{ij} = \IP{\partial_if}{\partial_jf},
\label{EQmetric}
\end{equation}
where $\partial$ denotes the regular partial derivative and $\IP{\cdot}{\cdot}$ is the standard Euclidean inner product.
Integration on $\Sigma$ is performed with respect to the induced area element
\begin{equation}
d\mu = \sqrt{\text{det }g}\ d\SL^n,
\label{EQdmu}
\end{equation}
where $d\SL^n$ is the standard Hausdorff measure on $\R^n$.

The \emph{second fundamental form} $A$ is a symmetric $(0,2)$ tensor field in
the normal bundle of $f$ with components
\begin{equation}
A_{ij} = (\partial^2_{ij}f)^\perp
\label{EQsff}
\end{equation}
There are two invariants of $A$ relevant to our work here: the first is the
trace with respect to the metric
\[
\vec{H} = \text{trace}_g\ A = g^{ij}A_{ij}
\]
called the \emph{mean curvature vector}, and the second the \emph{tracefree second fundamental form} defined by
\[
A^o_{ij} = A_{ij} - \frac1n\vec{H}g_{ij}\,.
\]
We define the \emph{Gauss curvature} to be
\[
K = \frac12(|\vec{H}|^2 - |A|^2)
\,.
\]
From \eqref{EQsff} and the smoothness of $f$ we see that the second fundamental form is
symmetric; less obvious but equally important is the symmetry of the first covariant derivatives of
$A$: 
\[ \nabla_iA_{jk} = \nabla_jA_{ik} = \nabla_kA_{ij}; \]
these are the Codazzi equations.
In the case here of high codimension they follow with $\nabla$ the connection
induced in the normal bundle along $f$ from the fact that the ambient space has
constant curvature.

One basic consequence of the Codazzi equations which we shall make use of is
that the gradient of the mean curvature is completely controlled by a
contraction of the $(0,3)$ tensor $\nabla A^o$.  To see this, first note that
\[
\nabla_i A^i_j = \nabla_i \vec{H} = \nabla_i \Big( (A^o)^i_j + \frac1n g_j^i \vec{H}\Big),
\]
then factorise to find
\begin{equation}
\label{EQbasicgradHgradAo}
\nabla_j\vec{H} = 2\nabla_i (A^o)^i_j =: \frac{n}{n-1}(\nabla^* A^o)_j\,.
\end{equation}
This in fact shows that all derivatives of $A$ are controlled by derivatives of $A^o$.
For a $(p,q)$ tensor field $T$, let us denote by $\nabla_{(n)}T$ the tensor field with components
 $\nabla_{i_1\ldots i_n}T_{j_1\ldots j_q}^{k_1\ldots k_p} =
  \nabla_{i_1}\cdots\nabla_{i_n}T_{j_1\ldots j_q}^{k_1\ldots k_p}$.
In our notation, the $i_n$-th covariant derivative is applied first.
Since
\[
\nabla_{(k)} A = \Big(\nabla_{(k)} A^o + \frac1n g \nabla_{(k)} \vec{H}\Big)
               = \Big(\nabla_{(k)} A^o + \frac1{n-1} g \nabla_{(k-1)}\nabla^*A^o\Big),
\]
we have
\begin{equation}
\label{EQbasicgradHgradAo2}
|\nabla_{(k)}A|^2 \le \frac{2n-1}{n-1}|\nabla_{(k)}A^o|^2\,.
\end{equation}
The fundamental relations between components of the Riemann curvature tensor $R_{ijkl}$, the Ricci
tensor $R_{ij}$ and scalar curvature $R$ are given by Gauss' equation
\begin{align*}
R_{ijkl} &= \IP{A_{ik}}{A_{jl}} - \IP{A_{il}}{A_{jk}},\intertext{with contractions}
g^{jl}R_{ijkl}
   = R_{ik} 
  &= \IP{\vec{H}}{A_{ik}} - \IP{A_i^j}{A_{jk}}\text{, and}\\
g^{ik}R_{ik}
   = R
  &= |H|^2 - |A|^2.
\end{align*}
We will need to interchange covariant derivatives.
For a $(0,m)$-tensor $T$ normal along $f$ we have
\begin{equation}
\label{EQinterchangespecific}
\nabla_{ij}T = \nabla_{ji}T + R^\perp_{ij}T
\end{equation}
where
\[
R^\perp_{ij}T = A_{ki}\IP{A_{kj}}{T} - A_{kj}\IP{A_{ki}}{T}
 = A^o_{ki}\IP{A^o_{kj}}{T} - A^o_{kj}\IP{A^o_{ki}}{T}
\,.
\]
Note that $\IP{R^\perp_{ij}T}{T} = 0$

We also use for normal tensor fields $T$ and $S$ the notation $T*S$ to denote a
linear combination of new tensors, each formed by contracting pairs of indices
from $T$ and $S$ by the metric $g$ with multiplication by a universal constant.
The resultant tensor will have the same type as the other quantities in the
expression it appears.
We denote polynomials in the iterated normal derivatives of $T$ by
\[
P_j^i(T) = \sum_{k_1+\ldots+k_j = i} c_{ij}\nabla_{(k_1)}T*\cdots*\nabla_{(k_j)}T,
\]
where the constants $c_{ij}\in\R$ are absolute.
As is common for the $*$-notation, we slightly abuse these constants when certain
subterms do not appear in our $P$-style
terms. For example
\begin{align*}
|\nabla A|^2
    = \IP{\nabla A}{\nabla A}
    = 1\cdot\left(\nabla_{(1)}A*\nabla_{(1)}A\right) + 0\cdot\left(A*\nabla_{(2)}A\right)
    = P_2^2(A).
\end{align*}
Using the Codazzi equation with the interchange of covariant derivative formula given above, we
obtain Simons' identity \cite{JS}:
\begin{equation}
\label{EQsi}
\Delta A = \nabla_{(2)}H + A*A*A.             
\end{equation}
The interchange of covariant derivatives formula for mixed tensor fields $T$ is simple to state in $*$-notation:
\begin{equation}
\label{EQinterchangegeneral}
\nabla_{ij}T = \nabla_{ji}T + T*A*A\,.
\end{equation}
Let $\{\partial_1f,\ldots,\partial_nf\}$ be an orthonormal basis for $T_pM$ and
$\{\nu_1,\ldots,\nu_{N-n}\}$ be an orthonormal basis for $N_pM$ with
Christoffel symbols in the normal bundle vanishing at $p$, that is,
$\overline{\Gamma}(p) = 0$.
We call such a frame for $T_pM \otimes N_pM$ a \emph{normal frame}.
Then
\begin{align}
\partial_i\nu_\alpha
 &=  \IP{\partial_i\nu_\alpha}{\partial_kf}\partial_kf 
  + \IP{\partial_i\nu_\alpha}{\nu_\beta}\nu_\beta
\notag\\
 &= - \IP{A_i^k}{\nu_\alpha}\partial_kf + \overline{\Gamma}_{i\alpha}^\beta\nu_\beta
  = - \IP{A_i^k}{\nu_\alpha}\partial_kf
\label{EQfirstW}
\intertext{so that}
\Delta\nu_\alpha
 &= -g^{ij}\big(
      \IP{\partial_jA_i^k}{\nu_\alpha}\partial_kf
      + 
      \IP{A_i^k}{\partial_j\nu_\alpha}\partial_kf
      + 
      \IP{A_i^k}{\nu_\alpha}\partial_j\partial_kf
           \big)
\notag\\
 &= -\IP{\nabla \vec{H}}{\nu_\alpha}
    -\IP{A^{ij}}{\nu_\alpha}A_{ij}\,.
\label{EQjacobi}
\end{align}
In most of our integral estimates, we include a function $\gamma:M^n\rightarrow\R$ in the integrand.
Eventually, this will be specialised to a smooth cutoff function on the
preimage of balls on $\R^{n+1}$ via the immersion $f$.
For now however, let us only assume that $\gamma = \tilde{\gamma}\circ f$, where
\begin{equation}
\tag{${\gamma}$}
0\le\tilde{\gamma}\le 1,\qquad\text{ and }\qquad
\vn{\tilde{\gamma}}_{C^2(\R^{n+1})} \le c_{\tilde{\gamma}} < \infty.
\end{equation}
Using the chain rule, this implies $D\gamma = (D\tilde{\gamma}\circ f)Df$ and then
$D^2\gamma = (D^2\tilde{\gamma}\circ f)(Df,Df) + (D\tilde{\gamma}\circ f)D^2f(\cdot,\cdot)$.
A routine calculation shows that there exists a constant $\cg=\cg(c_{\tilde{\gamma}}) \in \R$ such that
\begin{equation}
\label{EQgamma}
\tag{$\gamma$}
|\nabla\gamma| \le c\,\cg\,,\quad
|\nabla_{(2)}\gamma| \le c\,\cg(c_\gamma + |A|)\,,
\quad\text{ and }\quad
|\nabla_{(3)}\gamma| \le c\,\cg(\cg^2 + \cg|A| + |A|^2 +|\nabla A|).
\end{equation}
When we write ``for a function $\gamma:M^n\rightarrow\R$ as in \eqref{EQgamma}'' we mean a function
$\gamma:M^n\rightarrow\R$ as above, satisfying all conditions labeled \eqref{EQgamma}, which additionally
achieves the values zero and one in at least two points on $M^n$.

We note that if $\tilde{\gamma}$ is a cutoff function on a ball in $\R^{n+1}$
of radius $\rho$, then we may choose $\cg = \frac{c}\rho$ where $c$ is a
universal constant and we have used that
$c_{\tilde{\gamma}}=c_{\tilde{\gamma}}(\rho)$.

\subsection{The normal flow}

Chen's flow has tangential and normal components.
We calculate in a normal frame using \eqref{EQjacobi}:
\begin{align*}
\Delta^2 f &= \Delta \vec{H}
 = \Delta\Big(\IP{\vec{H}}{\nu_\alpha}\nu_\alpha\Big)
 = \Delta\big(H_\alpha\nu_\alpha\big)
\\
 &= (\Delta H_\alpha)\nu_\alpha + H_\alpha\Delta\nu_\alpha + 2\IP{\nabla H_\alpha}{\nabla\nu_\alpha}
\\
 &= \Delta H_\alpha\,\nu_\alpha - H_\alpha A_{ij}\IP{A^{ij}}{\nu_\alpha}
    + \Big(
       2\IP{\nabla H_\alpha}{\nabla\nu_\alpha} - H_\alpha\IP{\nabla \vec{H}}{\nu_\alpha}
      \Big)
\,.
\end{align*}
To see that the bracketed term is tangential, we compute using \eqref{EQfirstW}:
\begin{align*}
&\IP{2\IP{\nabla H_\alpha}{\nabla\nu_\alpha} - H_\alpha\IP{\nabla \vec{H}}{\nu_\alpha}}{\nu_\beta}
\\
 &\hskip+2cm= \IP{- 2g^{ij}\partial_iH_\alpha \IP{A^k_j}{\nu_\alpha}\partial_kf
                  - H_\alpha g^{ik}\IP{\partial_i\vec{H}}{\nu_\alpha} \partial_k f}{\nu_\beta}
\\
 &\hskip+2cm= \Big(
                  - 2g^{ij}\partial_iH_\alpha \IP{A^k_j}{\nu_\alpha}
                  - H_\alpha g^{ik}\IP{\partial_i\vec{H}}{\nu_\alpha}\Big)\IP{\partial_k f}{\nu_\beta}
\\
 &\hskip+2cm= 0
\,.
\end{align*}
It is a standard result that for closed curvature flow tangential motion acts
in the diffeomorphism group of $M^n$, which is tantamount to a
reparametrisation at each time (see for example \cite[Chapter 3]{bakerthesis}).
Therefore Chen's flow is equivalent to the purely normal flow:
\begin{equation}
\label{EQwf}
\tag{NCF}
(\partial_t f)(p,t) = -(\Delta^2 f)^{\perp}
 = -\big(\Delta H_\alpha\,\nu_\alpha - H_\alpha A_{ij}\IP{A^{ij}}{\nu_\alpha}\big)
 = -\BF
\,,
\end{equation}
with initial conditions $f(\cdot,0)=f_0$.
For simplicity we conduct our analysis with this formulation.

Note that we may express the velocity $\BF$ in a coordinate invariant manner as
\[
\BF = \Delta^\perp\vec{H} - Q(A)\vec{H}
\]
where $Q$ is a normal endomorphism of $NM$ acting on a section $\phi$ by
\[
Q(A)\phi = A_{ij}\IP{A^{ij}}{\phi}\,.
\]
The same endomorphism arises in the study of the Willmore flow in high codimension, see
for example (2.4) in \cite{kuwert2002gfw}.
%}}}

\section{Finite-time singularities and evolution equations}%{{{

The following evolution equations hold (see Lemma 2.2 in \cite{kuwert2002gfw}):

\begin{lem}
\label{LMevolutionequations}
For $f:M^n\times[0,T)\rightarrow \R^{N}$ evolving by $\partial_tf = -\BF$ the following equations hold:
\begin{align*}
  \pD{t}g_{ij} &= 2\IP{\BF}{A_{ij}}\,,\qquad
  \pD{t}d\mu  = \IP{\vec{H}}{\BF} \,d\mu\,,\qquad
  \pD{t}g^{ij}  = -2\IP{\BF}{A^{ij}}\,,\\
  \npD{}{t}A_{ij} &= -\nabla_{ij}\BF  + A_{ik}\IP{A^k_j}{\BF} \,,
\end{align*}
where $\npD{t}\phi = (\partial_t\phi)^\perp$.
\end{lem}
Using the $P$-notation introduced in the previous section we write the
evolution of the second fundamental form as
\begin{align*}
  \npD{t}A_{ij} = &- \nabla_{ij}\Delta^\perp \vec{H} + \big(P_3^2 + P_5^0\big)(A) \,.
\end{align*}
Interchanging covariant derivatives and applying \eqref{EQsi} then gives the following lemma:
\begin{lem}
\label{LMevoforA}
For $f:M^n\times[0,T)\rightarrow \R^{N}$ evolving by \eqref{EQwf} the following equation holds:
\begin{align*}
  \npD{t}A_{ij} = &- (\Delta^\perp)^2 A_{ij} + \big(P_3^2 + P_5^0 \big)(A)\,.
\end{align*}
\end{lem}
\begin{lem}
\label{LMevolutionequationhigher}
For $f:M^n\times[0,T)\rightarrow \R^{N}$ evolving by \eqref{EQwf} the following equation holds:
\begin{align*}
 \pD{}{t}\nabla_{(k)}A_{ij}
&= -\Delta^2\nabla_{(k)}A + \big(P_3^{k+2} + P_5^k\big)(A) \,.
\end{align*}
\end{lem}
Note that this is exactly the same structure that arises in the Willmore flow.
Therefore the $n=2$ case of the lifespan theorem can be proved using the
methods of \cite{kuwert2002gfw}.
For $n=3$, the work in \cite{mySDLTE} can be adapted along the lines of
\cite{mythesis,MWW10}.
For $n=4$, different arguments are required.

We now state the evolution of curvature quantities along the flow.
The proof is standard, and can be adapted from \cite{kuwert2002gfw}.

\begin{lem}
\label{LMenergyhighervanillaconstantricci}
Let $f:M^n\times[0,T)\rightarrow \R^{N}$ be a solution of \eqref{EQwf} and
$\gamma$ be as in \eqref{EQgamma}.
Suppose $s\ge2k+4$.
For each $\delta>0$ there exists a constant $c\in(0,\infty)$ depending only on $s$, $n$, $N$ and
$\delta$ such that the following estimate holds:
\begin{align*}
\rD{}{t}&\int_M|\nabla_{(k)}A|^2\gamma^sd\mu
    + (2-\delta)\int_M |\nabla_{(k+2)}A|^2 \gamma^sd\mu
\\*
&\le
    c(\cg)^{2k+4}\int_M |A|^2\gamma^{s-2k-4}d\mu
  + c\int_M \nabla_{(k)}A*\Big(P_3^{k+2}(A)+P_5^k(A)
  \Big)\,\gamma^sd\mu
\,.
\end{align*}
\end{lem}
Area is monotone under the flow:
\begin{lem}
\label{lemmapreservationarea}
Let $n\in\{2,3,4\}$.
For $f:M^n\times[0,T)\rightarrow \R^{N}$ evolving by \eqref{EQwf} we have
\begin{align*}
\mu(f_t)^\frac4n \le \mu(f_0)^\frac4n - C_nt
\end{align*}
where for $n\in\{2,3,4\}$ we have $C_n = 4\omega_n^{\frac4n} n^2$, and for $n
>4$ we have $C_n =  \frac{\omega_n^{\frac4n}}{n^24^{4n+3}}$. Here
$\omega_n$ denotes the area of the unit $n$-sphere.
\end{lem}
\begin{proof}
Differentiating,
\begin{align*}
\frac{d}{dt}\mu(f_t) &= \frac{d}{dt} \int_M d\mu
\\
&= \int_M \IP{\vec{H}}{\Delta^\perp \vec{H}} - \IP{Q(A)\vec{H}}{\vec{H}}\,d\mu\,.
\intertext{Using the estimate $\IP{Q(A)\vec{H}}{\vec{H}} \ge \frac1n|\vec{H}|^4$ and
the divergence theorem we estimate}
\frac{d}{dt}\mu(f_t)
&\le -\int_M |\nabla \vec H|^2\,d\mu  - \frac1n\int_M |\vec H|^4\,d\mu
 \le - \frac1n\int_M |\vec H|^4\,d\mu\,.
\end{align*}
Now we use the inequality
\begin{equation}
\label{EQchen}
\int_M |\vec H|^4\,d\mu
\ge \hat C_n\mu(f_t)^\frac{n-4}{n}\,,
\end{equation}
to estimate 
\[
\frac{n}{4}\Big(\mu(f_t)^\frac{4}{n}\Big)' \le -\frac{\hat C_n}{n}
\,.
\]
This implies
\[
\mu(f_t)^\frac4n \le \mu(f_0)^\frac4n - 4\frac{\hat C_n}{n^2}t
\]
as required.
The inequality \eqref{EQchen} follows for $n=2,3,4$ from the fundamental sharp estimate
\[
\int_M |\vec H|^n\,d\mu \ge \omega_nn^n
\]
of Chen \cite{chen1971total}.
For $n=4$ \eqref{EQchen} is immediate, whereas for $n=2,3$ we first use H\"older's inequality
\[
\int_M |\vec H|^n\,d\mu \le \bigg(\int_M |\vec H|^4\,d\mu\bigg)^\frac{n}{4}\mu(f_t)^\frac{4-n}{4}
\]
and then rearrange, to obtain
\[
\int_M |\vec H|^4\,d\mu
\ge \bigg(\int_M |\vec H|^n\,d\mu\bigg)^\frac{4}{n} \mu(f_t)^\frac{n-4}{n}
\ge \omega_n^\frac{4}{n}n^4 \mu(f_t)^\frac{n-4}{n}\,,
\]
that is, the estimate \eqref{EQchen} with $\hat C_n = \omega_n^\frac{4}{n}n^4$.
For $n>4$, this argument does not work and we must lose some sharpness in the constant.
In this case, we use Theorem 28.4.1 of \cite{burago1988gi} to estimate $\vn{H}_1$ from
below in terms of the area scaled appropriately.
Such an estimate follows directly from the Michael-Simon Sobolev
inequality (see Theorem 2.1 in \cite{michael1973sam}, stated in Theorem
\ref{TMmss} below) by an approximation argument:
\[
\mu(f_t)^{\frac{n-1}{n}} \le \frac{4^{n+1}}{\omega_n^{1/n}}\int_M |\vec{H}|\,d\mu
\,.
\]
Using H\"older's inequality we find
\[
\int_M |\vec{H}|\,d\mu \le \bigg(\int_M |\vec{H}|^4\,d\mu\bigg)^\frac14
                           \mu(f_t)^\frac34
\]
and so
\[
\int_M|\vec{H}|^4\,d\mu
 \ge \frac{\omega_n^{\frac4n}}{4^{4n+4}} \mu(f_t)^{\frac{n-4}n}
\,.
\]
This establishes \eqref{EQchen} with $\hat C_n =  \frac{\omega_n^{\frac4n}}{4^{4n+4}}$.
\end{proof}

\begin{proof}[Proof of Theorem \ref{TMfinite}]
Assume that the flow remains smooth with
$T>\frac{\mu(f_0)^\frac4n}{C_n}$. Then, applying Lemma
\ref{lemmapreservationarea} with
\[
t = \frac{\mu(f_0)^\frac4n}{C_n}
\]
shows that $\mu(f_t) = 0$, contradicting the assumption that the flow remains
smooth for $t\in[0,T)$.
Therefore either the family $f$ shrink to a point, in which case $T \le
\frac{\mu(f_0)}{C_n}$, or there is a loss of regularity beforehand.
In either case we have the estimate \eqref{EQmaxTest} as required.
\end{proof}
%}}}

\section{Integral estimates with small concentration of curvature}%{{{

The argument for $n=3$ and $n=4$ is by necessity different to that for $n=2$.
This is due to the important role played by the Michael-Simon Sobolev
inequality.

\begin{thm}[Theorem 2.1 in \cite{michael1973sam}] Let $f:M^n\rightarrow\R^{N}$
be a smooth immersed submanifold.  Then for any $u\in C_c^1(M)$ we have
\[  \left(\int_M |u|^{n/(n-1)}d\mu\right)^{(n-1)/n}
    \le \frac{4^{n+1}}{\omega_n^{1/n}} \int_M |\nabla u| + |u|\,|\vec H|\, d\mu\,.
\]
\label{TMmss}
\end{thm}

Notice the exponent on the left.  Our eventual goal for this section is to prove local $L^\infty$
estimates for all derivatives of curvature under a hypothesis that the local
concentration of curvature is small.  Our main tool to convert $L^q$ bounds to
$L^\infty$ bounds is the following theorem, which is an $n$-dimensional
analogue of Theorem 5.6 from \cite{kuwert2002gfw}.  The proof is contained in
Appendix A of \cite{mythesis}.

\begin{thm}
\label{myLZthm}
Let $f:M^n\rightarrow\R^{N}$ be a smooth immersed submanifold. For $u\in C_c^1(M)$,
$n<p\le\infty$, $0\le \beta\le \infty$ and $0<\alpha<1$ where $\frac{1}{\alpha} =
\big(\frac{1}{n}-\frac{1}{p}\big)\beta + 1$ we have
\begin{equation}
  \vn{u}_\infty \le c\vn{u}_\beta^{1-\alpha}(\vn{\nabla u}_p + \vn{\vec Hu}_p)^\alpha,
\label{myLZthmeqn}
\end{equation}
where $c = c(p,n,N,\beta)$.
\end{thm}

The proof follows ideas from \cite{ladyzhenskaya1968laq} and
\cite{kuwert2002gfw}.  Due to the exponent in the Michael-Simon Sobolev
inequality (which is itself an isoperimetric obstruction), it is not possible
to decrease the lower bound on $p$, even at the expense of other parameters in
the inequality.

For $n=3$, it is possible to use $p=4$ in Theorem \ref{myLZthm}.
This means that estimates on the same quantities as in the $n=2$
case may be used.
For $n=4$ we are not able to use $p=4$ in Theorem \ref{myLZthm}.
We thus need to estimate new quantities in this case.

\begin{lem}
\label{MS1lem}
Let $\gamma$ be as in \eqref{EQgamma}.  Then:
\begin{enumerate}
\item[(i)] For an immersed surface $f:M^2\rightarrow\R^{N}$, $s\ge4$, we
have
  \begin{align*}
  \int_M \big(|A|^2|\nabla A|^2 + |A|^6\big)\,\gamma^sd\mu
   &\le  c\int_{[\gamma>0]}|A|^2d\mu\int_M(|\nabla_{(2)}A|^2 + |A|^6)\gamma^sd\mu
   \\
   &\hskip+2cm + c(\cg)^4\Big( \int_{[\gamma>0]}|A|^2d\mu \Big)^2,
  \end{align*}
where $c = c(s,N)$ is an absolute constant.
\item[(ii)] For an immersion $f:M^4\rightarrow\R^N$, $s\ge2$, we have
\begin{align*}
	\int_M \big(|\nabla A|^2|A|^2 + |A|^6\big)\,\gamma^s\,d\mu
	\le
	&\ \theta\int_M |\nabla_{(2)}A|^2\,\gamma^s\,d\mu
	     + c(\vn{A}_{4,[\gamma>0]}^\frac43
	     + \vn{A}_{4,[\gamma>0]}^4)
	      \int_M |A|^6\,\gamma^s\,d\mu
	     \\&
	   + (\cg)^2\vn{A}_{4,[\gamma>0]}^4\,,
\end{align*}
and for $s\ge4$ we have
\begin{align*}
	\int_M \big(|\nabla A|^2|A|^3 + |A|^7\big)\,\gamma^s\,d\mu
	&\le
	 \big(c\vn{A}_{3,[\gamma>0]}^2 + \theta\big)\int_M \big(|\nabla_{(2)}A|^2|A| + |\nabla A|^2|A|^3 + |A|^7\big)\,\gamma^s\,d\mu
	     \\&\qquad
	   + (\cg)^4\vn{A}_{3,[\gamma>0]}^3\,,
\end{align*}
where $\theta\in(0,1)$ and $c = c(s,\theta,N)$ is an absolute constant.
\item[(iii)] For an immersion $f:M^4\rightarrow\R^N$, $s\ge8$, we have
\begin{align*}
\int_M &\big(|A|^2|\nabla_{(2)}A|^2 + |A|^4|\nabla A|^2 + |\nabla A|^4 + |A|^8\big)\,\gamma^s\,d\mu
\notag\\&\hskip-3mm
 + \bigg(\int_M |\nabla_{(2)}A|^{\frac{12}5}\,\gamma^{\frac{4s}{5}}\,d\mu\bigg)^\frac54
 + \bigg(\int_M |\nabla A|^3\,\gamma^{\frac{s}{2}}\,d\mu\bigg)^2
\notag\\
&\le
     \big(\theta + c\vn{A}_{4,[\gamma>0]}^\frac43\big)
     \int_M \big(|\nabla_{(3)}A|^2 + |A|^2\,|\nabla_{(2)}A|^2 + |\nabla A|^4 + |A|^8\big)\,\gamma^s\,d\mu
\notag\\&\quad
 + c\vn{A}_{4,[\gamma>0]}^\frac{20}{3}
    \bigg(\int_M |\nabla A|^3\,\gamma^\frac{s}{2}\,d\mu\bigg)^2
 + c\vn{A}_{4,[\gamma>0]}^\frac23
         \bigg(
           \int_M |\nabla_{(2)} A|^\frac{12}{5}\,\gamma^\frac{4s}{5}\,d\mu
         \bigg)^\frac54
\notag\\&\quad
 + c(c_\gamma)^4\Big(
           1
           + \vn{A}_{4,[\gamma>0]}^4
           + [(c_\gamma)^4\mu_\gamma(f)]^6
           \Big)\vn{A}_{4,[\gamma>0]}^3
      \,,
\end{align*}
where $\theta\in(0,1)$ and $c = c(s,\theta,N)$ is an absolute constant.
\item[(iv)] for an immersion $f:M^4\rightarrow\R^N$, $s\ge16$, we have
\begin{align*}
\int_M &\big(
          |\nabla_{(3)}A|^2|A|^2 + |\nabla_{(2)}A|^2|A|^4 + |\nabla_{(2)}A|^2|\nabla A|^2
        + |\nabla A|^2|A|^6 + |A|^{10}\big)\gamma^s\,d\mu
\\&
  + \bigg(
    \int_M |\nabla A|^4\,\gamma^{\frac{2s}{3}}\,d\mu
 \bigg)^\frac32
	+ \vn{A}_{4,[\gamma>0]}^2\int_M |\nabla_{(3)}A|^4\,\gamma^{2s}\,d\mu
\\&\le
           (\theta + c\vn{A}_{4,[\gamma>0]}^4)\int_M
                 \big(
             |\nabla_{(4)}A|^2
           + |\nabla_{(2)}A|^2|\nabla A|^2
           + |\nabla_{(2)}A|^2|A|^4
	   + |\nabla A|^2|A|^6 + |A|^{10}
                 \big)
                 \gamma^s\,d\mu
\notag\\&\quad
  + \vn{A}_{4,[\gamma>0]}^6
     \bigg(
       \int_M |\nabla A|^4\,\gamma^\frac{2s}{3}\,d\mu
     \bigg)^\frac32
	+ c(\theta + \vn{A}_{4,[\gamma>0]}^\frac23)\vn{A}_{4,[\gamma>0]}^2\int_M |\nabla_{(3)}A|^4\,\gamma^{2s}\,d\mu
\notag\\&\quad
  + c(\cg)^6\vn{A}_{4,[\gamma>0]}^2
            \big(1+[(\cg)^4\mu_\gamma(f)]^\frac12
            \big)
            \big(1 + \vn{A}_{4,[\gamma>0]}^2
	    \big)
\,.
\end{align*}
where $\theta\in(0,1)$ and $c = c(s,\theta,N)$ is an absolute constant.
\item[(v)] for an immersion $f:M^4\rightarrow\R^N$, $s\ge4$, we have
\begin{align}
\int_M |A|^8\,\gamma^s\,d\mu
\le c\vn{A}_{4,[\gamma>0]}^\frac43\int_M\big(|\nabla A|^4 + |A|^8\big)\,\gamma^s\,d\mu
 + c(\cg)^4\vn{A}_{4,[\gamma>0]}^\frac{16}3
 \label{EQnew1}
\end{align}
and
\begin{align}
\int_M |\nabla A|^4\,\gamma^s\,d\mu
\le c\int_M |\nabla_{(2)}A|^2|A|^2\,\gamma^s\,d\mu
 + c(\cg)^4\vn{A}_{4,[\gamma>0]}^4
 \label{EQnew2}
\end{align}
where $c = c(s,N)$ is an absolute constant.
\end{enumerate}
\end{lem}

We postpone the proof of Lemma \ref{MS1lem} to the Appendix.
Under an appropriate smallness condition, many terms can be absorbed, yielding
the following Corollary.

\begin{cor}
\label{MS1cor}
Let $\gamma$ be as in \eqref{EQgamma}.
There exists an $\varepsilon>0$ depending only on $n$, $s$, and $N$ such that if
\[
\int_{[\gamma>0]}|A|^n\,d\mu \le \varepsilon \le 1
\]
we have
\begin{enumerate}
\item[(i)] for an immersed surface $f:M^2\rightarrow\R^{N}$, $s\ge4$:
  \begin{align*}
  \int_M \big(|A|^2|\nabla A|^2+|A|^6\big)\,\gamma^sd\mu
   &\le  c\varepsilon\bigg(\int_M |\nabla_{(2)}A|^2\,\gamma^sd\mu + (\cg)^4\bigg)
  \end{align*}
where $c = c(s,N)$ is an absolute constant.
\item[(ii)] for an immersion $f:M^4\rightarrow\R^N$, $s\ge2$, we have
\begin{align*}
	\int_M \big(|\nabla A|^2|A|^2 + |A|^6\big)\,\gamma^s\,d\mu
	\le
	&\ \theta\int_M |\nabla_{(2)}A|^2\,\gamma^s\,d\mu
	   + (\cg)^2\vn{A}_{4,[\gamma>0]}^4\,,
\end{align*}
and for $s\ge4$ we have
\begin{align*}
	\int_M \big(|\nabla A|^2|A|^3 + |A|^7\big)\,\gamma^s\,d\mu
	&\le
	 \big(c\vn{A}_{3,[\gamma>0]}^2 + \theta\big)\int_M |\nabla_{(2)}A|^2|A|\,\gamma^s\,d\mu
	   + (\cg)^4\vn{A}_{3,[\gamma>0]}^3\,,
\end{align*}
where $\theta\in(0,1)$ and $c = c(s,\theta,N)$ is an absolute constant.
\item[(iii)] for an immersion $f:M^4\rightarrow\R^N$, $s\ge8$, we have
\begin{align*}
\int_M &\big(|A|^2|\nabla_{(2)}A|^2 + |A|^4|\nabla A|^2 + |\nabla A|^4 + |A|^8\big)\,\gamma^s\,d\mu
\notag\\
&\le
     \big(\theta + c\varepsilon^\frac13\big)
     \int_M |\nabla_{(3)}A|^2\,\gamma^s\,d\mu
 + c(c_\gamma)^4\Big(
           1
           + \varepsilon
           + [(c_\gamma)^4\mu_\gamma(f)]^6
           \Big)\vn{A}_{4,[\gamma>0]}^3
      \,,
\end{align*}
where $\theta\in(0,1)$ and $c = c(s,\theta,N)$ is an absolute constant.
\item[(iv)] for an immersion $f:M^4\rightarrow\R^N$, $s\ge16$, we have
\begin{align*}
\int_M &\big(
          |\nabla_{(3)}A|^2|A|^2 + |\nabla_{(2)}A|^2|A|^4 + |\nabla_{(2)}A|^2|\nabla A|^2
        + |\nabla A|^2|A|^6 + |A|^{10}\big)\gamma^s\,d\mu
\\&\le
           (\theta + c\vn{A}_{4,[\gamma>0]}^4)\int_M
             |\nabla_{(4)}A|^2
                 \gamma^s\,d\mu
  + c(\cg)^6\vn{A}_{4,[\gamma>0]}^2
            \big(1+[(\cg)^4\mu_\gamma(f)]^\frac12
            \big)
            \big(1 + \vn{A}_{4,[\gamma>0]}^2
	    \big)
\,.
\end{align*}
where $\theta\in(0,1)$ and $c = c(s,\theta,N)$ is an absolute constant.
\item[(v)] for an immersion $f:M^4\rightarrow\R^N$, $s\ge4$, we have
\begin{align*}
	\int_M |A|^8\,\gamma^s\,d\mu
\le c\vn{A}_{4,[\gamma>0]}^\frac43\int_M|\nabla A|^4\,\gamma^s\,d\mu
 + c(\cg)^4\vn{A}_{4,[\gamma>0]}^\frac{16}3
\end{align*}
and
\begin{align*}
\int_M |\nabla A|^4\,\gamma^s\,d\mu
\le c\int_M |\nabla_{(2)}A|^2|A|^2\,\gamma^s\,d\mu
 + c(\cg)^4\vn{A}_{4,[\gamma>0]}^4
\end{align*}
where $c = c(s,N)$ is an absolute constant.
\end{enumerate}
\end{cor}

Next we give a local refinement of Theorem \ref{myLZthm}.

\begin{prop}
\label{MS2prop}
Suppose $\gamma$ is as in \eqref{EQgamma}.
For any tensor $T$ normal along $f:M^n\rightarrow\R^{N}$, if $n=2$, we have
\begin{equation}
  \vn{T}^4_{\infty,[\gamma=1]}
\le c\vn{T}^{2}_{2,[\gamma>0]}\big( \vn{\nabla_{(2)}T}^2_{2,[\gamma>0]} +
(c_\gamma)^{4}\vn{T}^2_{2,[\gamma>0]} + \vn{T A^2}^2_{2,[\gamma>0]} \big)
\,,
\label{MS2}
\end{equation}
and if $n=4$, then we have
\begin{align}
	\vn{T}_{\infty,[\gamma=1]}^3
	\le c\vn{T}_{2,[\gamma>0]}\Big(
	\vn{\nabla_{(3)}T}_{2,[\gamma>0]}^2
	&+ \vn{T\,A^3}_{2,[\gamma>0]}^2
	+ (c_\gamma)^2\vn{A\nabla T}_{2,[\gamma>0]}^2
\notag\\&
\label{MS2n4}
	+ (c_\gamma)^2\vn{T\nabla A}_{2,[\gamma>0]}^2
	+ (c_\gamma)^4\vn{\nabla T}_{2,[\gamma>0]}^2
	+ (c_\gamma)^6\vn{T}_{2,[\gamma>0]}^2
             \Big)
\,,
\end{align}
where $c=c(n,N)$.

Assume $T=A$.
There exists an $\varepsilon_0 = \varepsilon_0(n,N)$ such that if
\[
  \vn{A}^n_{n,[\gamma>0]} \le \varepsilon_0
\]
we have for $n=2$:
\begin{equation*}
  \vn{A}^{4}_{\infty,[\gamma=1]}
    \le c\varepsilon_0
         \big(\vn{\nabla_{(2)}A}^{2}_{2,[\gamma>0]}
           + \varepsilon_0(c_\gamma)^4\big)\,,
\end{equation*}
and for $n=4$:
\begin{align}
\label{MS2secondstatement}
	\vn{A}_{\infty,[\gamma=1]}^3
	&\le c\vn{A}_{2,[\gamma>0]}\bigg(
             \vn{\nabla_{(3)}A}^2_{2,[\gamma>0]}
+ (c_\gamma)^4\big(
                    1 + \vn{A}_{4,[\gamma>0]}^4
                      + (c_\gamma)^4\mu_\gamma(f)
              \big)
             \bigg)
\,,
\end{align}
with $c=c(n,N,\varepsilon_0)$.
\end{prop}
\begin{proof}%{{{
The proof proceeds in two parts: first we deal with the case where $n=2$.
Then we prove the statements for $n=4$.
In each part we will estimate an arbitrary tensor field
$S$, and then we will localise the estimate for $S$ by using a $\gamma$
function.  Precisely, we specialise the estimate for $S$ to $S = T\gamma^2$ in
the first part and $S = T\gamma^4$ in the second, taking care to factor out the
quantity $\vn{T}^2_{2,[\gamma>0]}$ to conclude our desired inequality.  Note
that for $n=2$ the result is in \cite{kuwert2002gfw} (except here we keep track
of $c_\gamma$, and in the relevant result from \cite{kuwert2002gfw} the
constant $c$ depends on $\gamma$).

Here, and until we deal with the case $n=4$, we leave $n$ as a free parameter.
This is because the proof below works for both $n=2$ and $n=3$.
Therefore, let us take $p=4$, $\beta = 2$ in Theorem \ref{myLZthm} to obtain
\begin{equation}
\vn{S}_\infty \le c\vn{S}^\frac{4-n}{n+4}_2
                   \big(
                        \vn{\nabla S}_4 + \vn{S\ \vec{H}}_4
                   \big)^\frac{2n}{n+4}.
\label{MS2e1}
\end{equation}
We now use integration by parts and the H\"older inequality to derive
\begin{align}
\vn{\nabla S}^4_4
 &\le \int_M S*(\nabla_{(2)}S\,|\nabla S|^2 + 2\nabla S*\nabla S*\nabla_{(2)}S)d\mu
\notag\\
 &\le c\vn{S}_\infty\vn{\nabla S}^2_4\vn{\nabla_{(2)}S}_2,\text{ so}
\notag\\
\vn{\nabla S}_4
 &\le c\vn{S}_\infty^\frac{1}{2} \vn{\nabla_{(2)}S}_2^\frac{1}{2}.
\label{MS2eBabyfirstinterpolation}
\end{align}
Combine inequality \eqref{MS2eBabyfirstinterpolation} with \eqref{MS2e1} and use Jensen's
inequality to obtain
\begin{equation}
\vn{S}_\infty \le c\vn{S}_2^\frac{4-n}{n+4}\big[
                   (\vn{S}^\frac{1}{2}_\infty\vn{\nabla_{(2)}S}_2^\frac{1}{2})^\frac{2n}{n+4}
                 + \vn{S\ \vec H}^\frac{2n}{n+4}_4\big].
\label{MS2e2}
\end{equation}
Using H\"older's inequality we estimate
\[
\vn{S\ \vec H}^\frac{2n}{n+4}_4
  \le \left(\vn{S^2}_\infty^\frac{1}{4}\vn{S^\frac{1}{2}\vec H}_4\right)^\frac{2n}{n+4}
  \le \vn{S}^\frac{n}{n+4}_\infty\vn{S^\frac{1}{2}\vec H}^\frac{2n}{n+4}_4,
\]
and combining this with \eqref{MS2e2} above we conclude
\begin{align}
\vn{S}^4_\infty
  &= \Big(\vn{S}^{1-\frac{n}{n+4}}_\infty\Big)^{n+4}
\notag\\
  &\le \Big(c\vn{S}^\frac{4-n}{n+4}_2\big(\vn{\nabla_{(2)}S}_2^\frac{n}{n+4}
                              + \vn{S^\frac{1}{2}\vec H}_4^\frac{2n}{n+4}\big)\Big)^{n+4}
\notag\\
  &\le c\vn{S}_2^{4-n}\big(\vn{\nabla_{(2)}S}^n_2 + \vn{S\ |\vec H|^2}^n_2\big).
\label{MS2e3}
\end{align}
We now turn our attention to localising the estimate for $S$.  As mentioned earlier, for this
purpose we set $S = T\gamma^2$. We first evaluate and estimate the second derivative term
$\vn{\nabla_{(2)}S}^2_2$:
\begin{align}
\vn{\nabla_{(2)}&S}^2_2 = \int_M |\nabla_{(2)}(T\gamma^2)|^2d\mu
\notag\\
&\le c\bigg(
	\int_M |\nabla_{(2)}T|^2\gamma^4d\mu + \int_M|\nabla T|^2|\nabla\gamma^2|^2d\mu
 + \int_M |T|^2|\nabla_{(2)}\gamma^2|^2d\mu
	\bigg)
\notag\\
&\le c\bigg(
	\int_M |\nabla_{(2)}T|^2\gamma^4d\mu + \int_M|\nabla T|^2|\nabla\gamma|^2\gamma^2d\mu
\notag\\
&\qquad
 + \int_M |T|^2\left[|\nabla_{(2)}\gamma|\gamma+|\nabla\gamma|^2\right]^2d\mu
	\bigg)
\notag\\
&\le c\bigg(
	\int_M |\nabla_{(2)}T|^2\gamma^4d\mu + (c_\gamma)^2\int_M|\nabla T|^2\gamma^2d\mu
\notag\\
&\qquad
 + (c_\gamma)^2\int_M |T|^2|A|^2\gamma^2d\mu
 + (c_\gamma)^4\int_{[\gamma>0]} |T|^2d\mu
	\bigg)\,.
\label{MS2e4}
\end{align}
We interpolate the first derivative term:
\begin{align*}
(c_\gamma)^2\int_M&|\nabla T|^2\gamma^2d\mu
\le (c_\gamma)^2\int_M|T|\ |\nabla_{(2)}T|\gamma^2d\mu
 + c(c_\gamma)^3 \int_M |T|\ |\nabla T|\gamma d\mu
\notag\\
&\le 
     \frac{1}{2}(c_\gamma)^2\int_M|\nabla T|^2\gamma^2d\mu
   + c\int_M |\nabla_{(2)}T|^2 \gamma^4 d\mu
   + c(c_\gamma)^4\int_{[\gamma>0]}|T|^2d\mu
\notag\intertext{and thus}
(c_\gamma)^2\int_M&|\nabla T|^2\gamma^2d\mu
\le c\int_M |\nabla_{(2)}T|^2 \gamma^4 d\mu
   + c(c_\gamma)^4\int_{[\gamma>0]}|T|^2d\mu
\,.
\end{align*}
Inserting this result into \eqref{MS2e4}, and estimating
\[
   c(c_\gamma)^2\int_M |T|^2|A|^2\gamma^2d\mu
\le
   c\int_M |T|^2|A|^4\gamma^4d\mu
   + c(c_\gamma)^4\int_{[\gamma>0]}|T|^2d\mu
\]
we obtain
\begin{equation}
\vn{\nabla_{(2)}S}^2_2 \le c\int_{[\gamma>0]}|\nabla_{(2)}T|^2\,d\mu
   + c\int_M |T|^2|A|^4\gamma^4d\mu
   + c(c_\gamma)^4\int_{[\gamma>0]}|T|^2d\mu
\,.
\label{MS2e5}
\end{equation}
Combining this with our estimate for $\vn{S}_\infty$ earlier, inequality
\eqref{MS2e3}, gives
\begin{align}
\vn{S}^4_\infty &\le c\vn{S}^{4-n}_2\big(\vn{\nabla_{(2)}T}^n_{2,[\gamma>0]}
           + (c_\gamma)^{2n}\vn{T}^n_{2,[\gamma>0]} + \vn{S\ H^2}_2^n + \vn{T A^2\gamma^2}_2^n\big)
\notag\\
&\le c\vn{T}^{4-n}_{2,[\gamma>0]}\big( \vn{\nabla_{(2)}T}^n_{2,[\gamma>0]} +
(c_\gamma)^{2n}\vn{T}^n_{2,[\gamma>0]} + \vn{T A^2}^n_{2,[\gamma>0]} \big)
\,.
\label{MS2e6}
\end{align}
Estimating $\vn{T}^4_{\infty,[\gamma=1]} \le \vn{S}^4_\infty$ proves \eqref{MS2}.

Now set $T=A$ in \eqref{MS2}.

For $n=2$, Lemma \ref{MS1lem} (i) implies
\begin{align*}
\int_M|A|^6\gamma^4d\mu
 &\le c\vn{A}^2_{2,[\gamma>0]}\big(\vn{\nabla_{(2)}A}^2_{2,[\gamma>0]}
 + \vn{A\gamma^\frac{2}{3}}^6_6\big)
 + c(c_\gamma)^4\vn{A}^4_{2,[\gamma>0]},
\intertext{and absorbing on the left we obtain}
\int_M |A|^6\gamma^4d\mu
 &\le c \vn{A}^2_{2,[\gamma>0]}\big(\vn{\nabla_{(2)}A}^2_{2,[\gamma>0]}
                                + (c_\gamma)^4\vn{A}^2_{2,[\gamma>0]}\big)
\,.
\end{align*}
Inserting this into \eqref{MS2e6} gives the second statement for $n=2$.

For the $n=4$ inequalities, we proceed similarly. We first claim
\begin{equation}
\vn{S}_\infty
	\le c\vn{S}_2^\frac13(\vn{\nabla_{(3)}S}_2^\frac23
+ \vn{S^\frac13\vec{H}}_{6}^2)
\,.
\label{EQgoal}
\end{equation}
In order to prove \eqref{EQgoal}, we need some auxilliary estimates.
First, we calculate 
\begin{align}
	\int_M |\nabla S|^6 
		      &\le c\int_M |S|\,|\nabla S|^4\,|\nabla_{(2)}S|\,d\mu
		\notag      \\
		      &\le c\vn{S}_\infty \vn{\nabla S}_6^4\bigg(\int_M |\nabla_{(2)}S|^3\,d\mu\bigg)^\frac13
\notag\intertext{so}
\vn{\nabla S}_6^6 &\le c\vn{S}_\infty^3\vn{\nabla_{(2)}S}_3^3
\,.
\label{EQhelper1}
\end{align}
We also need
\begin{align*}
	\int_M |\nabla_{(2)}S|^3\,d\mu
	&\le c\int_M |\nabla S|\,|\nabla_{(2)}S|\,|\nabla_{(3)}S|\,d\mu
	\\
	&\le \frac12\vn{\nabla_{(2)}S}_3^3
	+ c\int_M |\nabla S|^\frac32\,|\nabla_{(3)}S|^\frac32\,d\mu
\intertext{so}
\int_M |\nabla_{(2)}S|^3\,d\mu
&\le c\int_M |\nabla S|^\frac32\,|\nabla_{(3)}S|^\frac32\,d\mu
\\
&\le \frac1{2c\vn{S}_\infty^3}\int_M |\nabla S|^6\,d\mu
   + c\vn{S}_\infty\int_M |\nabla_{(3)}S|^2\,d\mu\,.
\end{align*}
(Note that if $\vn{S}_\infty = 0$ then the estimate is trivially true, and so we assume this is not the case.)
Combining with \eqref{EQhelper1} and absorbing we find
\[
\vn{\nabla S}_6^6 \le c\vn{S}_\infty^4\vn{\nabla_{(3)}S}_2^2
\,.
\]
Now applying Theorem \ref{myLZthm} yields
\begin{align*}
	\vn{S}_\infty
	&\le c\vn{S}_\beta^{1-\alpha}(\vn{\nabla S}_{6} + \vn{\vec{H}\,S}_{6})^\alpha
	\\
	&\le c\vn{S}_\beta^{1-\alpha}(\vn{S}_\infty^\frac23\vn{\nabla_{(3)}S}_2^\frac13
	  + \vn{S}_\infty^\frac23\vn{S^\frac13\vec{H}}_{6})^\alpha
	  \intertext{so}
\vn{S}_\infty
	&\le c\vn{S}_\beta^{\frac{3-3\alpha}{3-2\alpha}}(\vn{\nabla_{(3)}S}_2^\frac13
+ \vn{S^\frac13\vec{H}}_{6})^\frac{3\alpha}{3-2\alpha}
\end{align*}
where
$\alpha^{-1} = 1 + (\frac{1}{n} - \frac{1}{6})\beta$.
Since $n=4$, $\alpha^{-1} = \frac{12 + \beta}{12}$, and
\[
	\frac{3\alpha}{3-2\alpha}
	= \frac{12}{12 + \beta}\frac{3}{3 - 2\frac{12}{12+\beta}}
	= \frac{36}{36 + 3\beta - 24}
	= \frac{12}{4 + \beta}
\]
so in particular if $\beta = 2$ then $3\alpha/(3-2\alpha) = 2$ or $\alpha = 6/7$.
We also note that $\frac{3-3\alpha}{3-2\alpha} = \frac13$.
This proves the estimate \eqref{EQgoal}.

Now we set $S = T\gamma^4$ and calculate
\begin{align*}
\int_M |\nabla_{(3)}(T\gamma^4)|^2\,d\mu
 &\le c\int_M |\nabla_{(3)}T|^2\,\gamma^8\,d\mu
+ c(c_\gamma)^2\int_M |\nabla_{(2)}T|^2\,\gamma^6\,d\mu
\\&\quad
+ c(c_\gamma)^2\int_M |\nabla T|^2\,((c_\gamma^2 + |A|^2)\gamma^2 + c_\gamma^2)\gamma^4\,d\mu
\\&\quad
+ c(c_\gamma)^2\int_M |T|^2\,(
                   (c_\gamma^4 + c_\gamma^2|A|^2 + |A|^4 + |\nabla A|^2)\gamma^4
		   + (c_\gamma^2(c_\gamma^2 + |A|^2))\gamma^2
		 + c_\gamma^4)\gamma^2\,d\mu
\\
 &\le c\int_M |\nabla_{(3)}T|^2\,\gamma^8\,d\mu
+ c(c_\gamma)^2\int_M |\nabla_{(2)}T|^2\,\gamma^6\,d\mu
+ c(c_\gamma)^2\int_M |\nabla T|^2\,|A|^2\,\gamma^6\,d\mu
\\&\quad
+ c(c_\gamma)^2\int_M |T|^2\,
                   |\nabla A|^2\,\gamma^6\,d\mu
+ c(c_\gamma)^2\int_M |T|^2\,
                   |A|^4\,\gamma^6\,d\mu
\\&\quad
+ c(c_\gamma)^4\int_M |\nabla T|^2\,\gamma^4\,d\mu
+ c(c_\gamma)^4\int_M |T|^2\,
                   |A|^2\,\gamma^4\,d\mu
+ c(c_\gamma)^6\int_M |T|^2\,
		 \gamma^2\,d\mu
\\
 &\le c\int_M |\nabla_{(3)}T|^2\,\gamma^8\,d\mu
+ c(c_\gamma)^2\int_M |\nabla_{(2)}T|^2\,\gamma^6\,d\mu
+ c(c_\gamma)^2\int_M |\nabla T|^2\,|A|^2\,\gamma^6\,d\mu
\\&\quad
+ c(c_\gamma)^2\int_M |T|^2\,
                   |\nabla A|^2\,\gamma^6\,d\mu
+ c(c_\gamma)^2\int_M |T|^2\,
                   |A|^4\,\gamma^6\,d\mu
\\&\quad
+ c(c_\gamma)^4\int_M |\nabla T|^2\,\gamma^4\,d\mu
+ c(c_\gamma)^6\int_M |T|^2\,
		 \gamma^2\,d\mu
\,.
\end{align*}
Note that
\begin{align*}
	(c_\gamma)^2\int_M |\nabla_{(2)} T|^2\,\gamma^6\,d\mu
&\le c(c_\gamma)^2\int_M |\nabla T|\,|\nabla_{(3)}T|\,\gamma^6\,d\mu
	+ c(c_\gamma)^3\int_M |\nabla T|\,|\nabla_{(2)} T|\,\gamma^5\,d\mu
  \\
&\le
\frac12(c_\gamma)^2\int_M |\nabla_{(2)} T|^2\,\gamma^6\,d\mu
   + c(c_\gamma)^2\int_M |\nabla T|\,|\nabla_{(3)}T|\,\gamma^6\,d\mu
  + c(c_\gamma)^4\int_M |\nabla T|^2\,\gamma^4\,d\mu
\end{align*}
so that
\[
	(c_\gamma)^2\int_M |\nabla_{(2)} T|^2\,\gamma^6\,d\mu
\le
     \int_M |\nabla_{(3)}T|^2\,\gamma^8\,d\mu
  + (c_\gamma)^4\int_M |\nabla T|^2\,\gamma^4\,d\mu
  \,.
\]
This refines the above to
\begin{align*}
\int_M |\nabla_{(3)}(T\gamma^4)|^2\,d\mu
 &\le c\int_M |\nabla_{(3)}T|^2\,\gamma^8\,d\mu
+ c(c_\gamma)^2\int_M |\nabla T|^2\,|A|^2\,\gamma^6\,d\mu
\\&\quad
+ c(c_\gamma)^2\int_M |T|^2\,
                   |\nabla A|^2\,\gamma^6\,d\mu
+ c(c_\gamma)^2\int_M |T|^2\,
                   |A|^4\,\gamma^6\,d\mu
\\&\quad
+ c(c_\gamma)^4\int_M |\nabla T|^2\,\gamma^4\,d\mu
+ c(c_\gamma)^6\int_M |T|^2\,
		 \gamma^2\,d\mu
\,.
\end{align*}
Combining this with \eqref{EQgoal} and then cubing everything yields
\begin{align*}
\vn{T\gamma^4}_\infty^3
&\le c\vn{T\gamma^4}_2\Big(
  \int_M |\nabla_{(3)}T|^2\,\gamma^8\,d\mu
     + \int_M |T|^2|A|^6\,\gamma^8\,d\mu
\\&\qquad
+ (c_\gamma)^2\int_M |\nabla T|^2\,|A|^2\,\gamma^6\,d\mu
+ (c_\gamma)^2\int_M |T|^2\,
                   |\nabla A|^2\,\gamma^6\,d\mu
\\&\qquad
+ (c_\gamma)^2\int_M |T|^2\,
                   |A|^4\,\gamma^6\,d\mu
+ (c_\gamma)^4\int_M |\nabla T|^2\,\gamma^4\,d\mu
+ (c_\gamma)^6\int_M |T|^2\,
		 \gamma^2\,d\mu
             \Big)
\,.
\end{align*}
Using the definition of $\gamma$ we have
\begin{align*}
	\vn{T}_{\infty,[\gamma=1]}^3
	\le c\vn{T}_{2,[\gamma>0]}\Big(
	\vn{\nabla_{(3)}T}_{2,[\gamma>0]}^2
	&+ \vn{T\,A^3}_{2,[\gamma>0]}^2
	+ (c_\gamma)^2\vn{A\nabla T}_{2,[\gamma>0]}^2
\\&
	+ (c_\gamma)^2\vn{T\nabla A}_{2,[\gamma>0]}^2
	+ (c_\gamma)^4\vn{\nabla T}_{2,[\gamma>0]}^2
	+ (c_\gamma)^6\vn{T}_{2,[\gamma>0]}^2
             \Big)
\,.
\end{align*}
Note that we interpolated one term.

In the particular case where $T=A$ we find
\begin{align*}
	\vn{A}_{\infty,[\gamma=1]}^3
	&\le c\vn{A}_{2,[\gamma>0]}\bigg(
  \int_M |\nabla_{(3)}A|^2\,\gamma^8\,d\mu
+ \int_M (|\nabla A|^4 + |A|^8)\,\gamma^8\,d\mu
\\&\hskip+4cm
+ (c_\gamma)^4\big(
                    1 + \vn{A}_{4,[\gamma>0]}^4
                      + (c_\gamma)^4\mu_\gamma(f) + (c_\gamma)^2\vn{A}_{2,[\gamma>0]}^2
              \big)
             \bigg)
\,.
\end{align*}
When $\vn{A}^4_{4,[\gamma>0]}$ is small, we may use Corollary \ref{MS1cor} (iii) to
absorb the second integral on the right, and conclude
\begin{align*}
	\vn{A}_{\infty,[\gamma=1]}^3
	&\le c\vn{A}_{2,[\gamma>0]}\bigg(
             \vn{\nabla_{(3)}A}^2_{2,[\gamma>0]}
+ (c_\gamma)^4\big(
                    1 + \vn{A}_{4,[\gamma>0]}^4
                      + (c_\gamma)^4\mu_\gamma(f) + (c_\gamma)^2\vn{A}_{2,[\gamma>0]}^2
              \big)
             \bigg)
\,.
\end{align*}
\end{proof}%}}}

The Lifespan Theorem is proved using an alternative that relies on being able
to, in a weak sense, preserve the assumption
\[
	\int_{[\gamma>0]} |A|^n\,d\mu < \varepsilon_0
\]
at later times.
A key difficulty is that the flow lives naturally in the $L^2$ heirarchy, and
so does not directly control the $L^n$ norm of curvature.
This in turn introduces difficulties in obtaining pointwise control of curvature.
For $n=2$ this does not cause any issue.
For $n=4$ the same Sobolev inequalities can not apply.
Nevertheless we are able to use those proved above to obtain pointwise control
in this case as well.

We begin with the $L^2$-control.

\begin{prop}
\label{p:prop44}
Let $n\in\{2,4\}$.
Suppose $f:M^n\times[0,T^*]\rightarrow\R^{N}$ evolves by
\eqref{EQwf} and $\gamma$ is a cutoff function as in \eqref{EQgamma}.
Then there is a universal $\varepsilon_0 = \varepsilon_0(N)$ such that if
\begin{equation}
\label{eq9}
\varepsilon = \sup_{[0,T^*]}\int_{[\gamma>0]}|A|^nd\mu\le\varepsilon_0
\end{equation}
then for any $t\in[0,T^*]$ we have
\begin{align*}
&\int_{[\gamma=1]} |A|^2 \,d\mu\ +
 \int_0^t\int_{[\gamma=1]} (|\nabla_{(2)}A|^2 + |A|^2|\nabla A|^2 + |A|^6)\,
d\mu d\tau\\
&\le \int_{[\gamma>0]}|A|^2d\mu\bigg|_{t=0}
	       + ct(\cg)^{6-n}
                      \Big(
                            1 + (n-2)(4-n)[(\cg)^3\mu(f_0)]^\frac13
                              + (n-2)(n-3)[(\cg)^4\mu(f_0)]^\frac12
                      \Big)\varepsilon^\frac{2}{n}
	       \,,
\end{align*}
where $c=c(n,N)$.
\end{prop}
\begin{proof}%{{{
The idea of the proof is to integrate Lemma
\ref{LMenergyhighervanillaconstantricci}, and then use the multiplicative
Sobolev inequality Lemma \ref{MS1lem}.  This will introduce a multiplicative factor of
$\vn{A}_{n,[\gamma>0]}$ in front of several integrals, which we can then absorb on the left.
The proof for $n=2$ is the same as that in \cite{kuwert2002gfw}.
Therefore here we give only the proof for $n=4$.

Setting $k=0$ and $s=4$ in Lemma \ref{LMenergyhighervanillaconstantricci} we have
\begin{align*}
&\rD{}{t}\int_M |A|^2\gamma^4d\mu
 + (2-\theta)\int_M |\nabla_{(2)}A|^2\gamma^4d\mu
\\
&\qquad
\le c(\cg)^4\int_{[\gamma>0]}|A|^2\,d\mu
            + c\int_M \left([P_3^{2}(A)+P_5^0(A)]*A\right)\gamma^{4} d\mu.
\end{align*}
We estimate the $P$-style terms:
\begin{align}
\int_M\left([P_3^{2}(A)+P_5^0(A)]*A\right)\gamma^{4} d\mu
&\le c\int_M\Big(\big[|A|^2\cdot|\nabla_{(2)}A|
                + |\nabla A|^2|A|+|A|^5\big]|A|\Big)\gamma^{4} d\mu
\notag\\&\le
     c\int_M\big[|A|^3|\nabla_{(2)}A|
                + |\nabla A|^2|A|^2+|A|^6\big]\gamma^{4} d\mu
\notag\\&\le
     \theta_0\int_M |\nabla_{(2)}A|^2\gamma^4d\mu
   +  c\int_M (|A|^6 + |\nabla A|^2|A|^2)\gamma^{4} d\mu.
\notag
\intertext{
We use Corollary \ref{MS1cor} (ii) to estimate the second integral and obtain}
\int_M\left([P_3^{2}(A)+P_5^0(A)]*A\right)\gamma^{4} d\mu
	&\le
	 \theta\int_M |\nabla_{(2)}A|^2\,\gamma^s\,d\mu
	   + (\cg)^2\vn{A}_{4,[\gamma>0]}^4\,,
\label{prop44eq1}
\end{align}
We add the integrals $\int_M |A|^6 \gamma^4d\mu$ and $\int_M |\nabla A|^2|A|^2\gamma^4d\mu$
to the estimate of Lemma \ref{LMenergyhighervanillaconstantricci} (with $k=0$, $s=4$) and find
\begin{align*}
&\rD{}{t}\int_M |A|^2\gamma^4d\mu
 + (2-\theta)\int_M \big(|\nabla_{(2)}A|^2+|A|^2|\nabla A|^2+|A|^6\big)\gamma^4 d\mu
\\
&\qquad \le c(\cg)^4\int_{[\gamma>0]} |A|^2d\mu
     + c\int_M \big(|A|^2|\nabla A|^2+|A|^6\big)\gamma^4 d\mu
     + c\int_M \left([P_3^{2}(A)+P_5^0(A)]*A\right)\gamma^{4} d\mu
\\
&\qquad \le c(\cg)^4\int_{[\gamma>0]} |A|^2d\mu
+ \theta\int_M |\nabla_{(2)}A|^2\,\gamma^s\,d\mu
	   + (\cg)^2\vn{A}_{4,[\gamma>0]}^4\,,
\intertext{
which upon absorbing and choosing $\theta$ small yields
}
&\rD{}{t}\int_M |A|^2\gamma^4d\mu
 + \int_M \big(|\nabla_{(2)}A|^2+|A|^2|\nabla A|^2+|A|^6\big)\gamma^4 d\mu
\le c(\cg)^2\vn{A}_{4,[\gamma>0]}^2
     \Big(
      (\cg)^2\mu_\gamma(f_t)^\frac12 + \vn{A}_{4,[\gamma>0]}^2
     \Big)
\,.
\end{align*}
Integrating, we have
\begin{align*}
&\int_{[\gamma=1]} |A|^2 \,d\mu\ +
 \int_0^t\int_{[\gamma=1]} (|\nabla_{(2)}A|^2 + |A|^2|\nabla A|^2 + |A|^6)\,
d\mu d\tau\\
&\le \int_{[\gamma>0]}|A|^2d\mu\bigg|_{t=0}
	       + ct(\cg)^{6-n}
                      \Big(
                            1 + (n-2)(4-n)[(\cg)^3\mu(f_0)]^\frac13
                              + (n-2)(n-3)[(\cg)^4\mu(f_0)]^\frac12
                      \Big)\varepsilon^\frac{2}{n}
	       \,,
\end{align*}
where we have incorporated the three cases into one statement, and used
$\varepsilon\le1$, $\mu_\gamma(f_t) \le \mu(f_t) \le \mu(f_0)$,
$[\gamma=1]\subset[\gamma>0]$ and $0\le\gamma\le1$.
\end{proof}%}}}

\begin{rmk}
It is possible to proceed as in \cite{metzger2013willmore} and prove a bound
directly for $\mu_\gamma(f_t)$ in terms of $\mu_\gamma(f_0)$, under the
smallness hypothesis \eqref{eq9}. However this yields a bound exponential in
time, which is quickly worse than the simple but uniform in time bound used
above. It is an interesting open question on how to control the area locally
uniformly in time without resorting to this crude estimate.
In order to overcome this issue we prove the following estimate for
the scale-invariant $\vn{\nabla A}_{2,\gamma^s}^2 + \vn{A}_{4,\gamma^s}^4$ directly.
\end{rmk}

\begin{prop}
\label{newl4}
Suppose $f:M^4\times[0,T^*]\rightarrow\R^{N}$ evolves by
\eqref{EQwf} and $\gamma$ is a cutoff function as in \eqref{EQgamma}.
Then there is a universal $\varepsilon_0 = \varepsilon_0(N)$ such that if
\begin{equation*}
\varepsilon = \sup_{[0,T^*]}\int_{[\gamma>0]}|A|^4+|\nabla A|^2d\mu\le\varepsilon_0
\end{equation*}
then for any $t\in[0,T^*]$ we have
\begin{align*}
&\int_{[\gamma=1]} |A|^4+|\nabla A|^2 \,d\mu\ +
 \int_0^t\int_{[\gamma=1]} (|\nabla_{(3)}A|^2 + |\nabla_{(2)}A|^2|A|^2 + |\nabla A|^2|A|^4 + |\nabla A|^4 + |A|^8)\,
d\mu d\tau\\
&\le \int_{[\gamma>0]}|A|^4+|\nabla A|^2d\mu\bigg|_{t=0}
	       + ct(\cg)^4
                      \varepsilon_0
	       \,,
\end{align*}
where $c$ depends only on $\cg$ and $N$.
\end{prop}
\begin{proof}%{{{
Let us first calculate
\begin{align*}
\frac{d}{dt}\int_M |A|^4\gamma^s\,d\mu
 &= 4\int_M |A|^2\IP{A}{A_t}\gamma^s\,d\mu
+ \int_M |A|^4\IP{\vec{H}}{\BF}\gamma^s\,d\mu
  + s\int_M |A|^4\gamma_t\gamma^{s-1}\,d\mu
\\&=
    4\int_M |A|^2\IP{A}{
  - (\Delta^\perp)^2 A + \big(P_3^2 + P_5^0 \big)(A)
			}\gamma^s\,d\mu
\\&\qquad
+ \int_M |A|^4\IP{\vec{H}}{\big(P_1^2 + P_3^0 \big)(A)}\gamma^s\,d\mu
  + s\int_M |A|^4\gamma_t\gamma^{s-1}\,d\mu
\,.
\end{align*}
Observe that
\begin{align*}
    4&\int_M |A|^2\IP{A}{
  - (\Delta^\perp)^2 A + \big(P_3^2 + P_5^0 \big)(A)
			}\gamma^s\,d\mu
+ \int_M |A|^4\IP{\vec{H}}{\big(P_1^2 + P_3^0 \big)(A)}\gamma^s\,d\mu
\\
&=
  - 4\int_M |A|^2\IP{A}{
    \nabla^p \nabla^q \nabla_q \nabla_p A
			}\gamma^s\,d\mu
+ \int_M P_3^0(A)*\big(P_3^2 + P_5^0 \big)(A)\gamma^s\,d\mu
\\
&=
  - 4\int_M \IP{\nabla_{(2)}(A|A|^2)}{
    \nabla_{(2)} A
			}\gamma^s\,d\mu
  + 4s\int_M |A|^2\IP{A\nabla^p\gamma}{
    \nabla^q \nabla_q \nabla_p A
			}\gamma^{s-1}\,d\mu
\\&\qquad
  - 4s\int_M \IP{\nabla^p(A|A|^2)\nabla^q\gamma}{
    \nabla_q \nabla_p A
			}\gamma^{s-1}\,d\mu
+ \int_M P_3^0(A)*\big(P_3^2 + P_5^0 \big)(A)\gamma^s\,d\mu
\\
&=
  - 4\int_M \IP{\nabla^q(\nabla^p A|A|^2 + 2A\IP{A}{\nabla^p A})}{
    \nabla_{qp} A
			}\gamma^s\,d\mu
\\&\qquad
  + 4s\int_M |A|^2\IP{A\nabla^p\gamma}{
    \nabla^q \nabla_q \nabla_p A
			}\gamma^{s-1}\,d\mu
  - 4s\int_M \IP{\nabla^p(A|A|^2)\nabla^q\gamma}{
    \nabla_q \nabla_p A
			}\gamma^{s-1}\,d\mu
\\&\qquad
+ \int_M P_3^0(A)*\big(P_3^2 + P_5^0 \big)(A)\gamma^s\,d\mu
\\
&=
  - 4\int_M \IP{\nabla^{qp} A|A|^2 + 2\nabla^pA\IP{A}{\nabla^qA} + 2\nabla^qA\IP{A}{\nabla^p A}
                                   + 2A\IP{\nabla^pA}{\nabla^qA} + 2A\IP{A}{\nabla^{qp} A}}{
    \nabla_{qp} A
			}\gamma^s\,d\mu
\\&\qquad
  + 4s\int_M |A|^2\IP{A\nabla^p\gamma}{
    \nabla^q \nabla_q \nabla_p A
			}\gamma^{s-1}\,d\mu
  - 4s\int_M \IP{\nabla^p(A|A|^2)\nabla^q\gamma}{
    \nabla_q \nabla_p A
			}\gamma^{s-1}\,d\mu
\\&\qquad
+ \int_M P_3^0(A)*\big(P_3^2 + P_5^0 \big)(A)\gamma^s\,d\mu
\\
&=
  - 4\int_M |\nabla_{(2)}A|^2|A|^2\,\gamma^s\,d\mu
  - 8\int_M |\IP{A}{\nabla_{(2)}A}|^2\,\gamma^s\,d\mu
\\&\qquad
  - 4\int_M \IP{2\nabla^pA\IP{A}{\nabla^qA} + 2\nabla^qA\IP{A}{\nabla^p A}
                                   + 2A\IP{\nabla^pA}{\nabla^qA} }{
    \nabla_{qp} A
			}\gamma^s\,d\mu
\\&\qquad
  + 4s\int_M |A|^2\IP{A\nabla^p\gamma}{
    \nabla^q \nabla_q \nabla_p A
			}\gamma^{s-1}\,d\mu
  - 4s\int_M \IP{\nabla^p(A|A|^2)\nabla^q\gamma}{
    \nabla_q \nabla_p A
			}\gamma^{s-1}\,d\mu
\\&\qquad
+ \int_M P_3^0(A)*\big(P_3^2 + P_5^0 \big)(A)\gamma^s\,d\mu
\\
&=
  - 4\int_M |\nabla_{(2)}A|^2|A|^2\,\gamma^s\,d\mu
  - 8\int_M |\IP{A}{\nabla_{(2)}A}|^2\,\gamma^s\,d\mu
  +  \int_M (A*\nabla A*\nabla A*\nabla_{(2)}A)\,\gamma^s\,d\mu
\\&\qquad
  + 4s\int_M |A|^2\IP{A\nabla^p\gamma}{
    \nabla^q \nabla_q \nabla_p A
			}\gamma^{s-1}\,d\mu
  - 4s\int_M \IP{\nabla^p(A|A|^2)\nabla^q\gamma}{
    \nabla_q \nabla_p A
			}\gamma^{s-1}\,d\mu
\\&\qquad
+ \int_M P_3^0(A)*\big(P_3^2 + P_5^0 \big)(A)\gamma^s\,d\mu
\end{align*}
Using $\gamma = \tilde\gamma\circ f$, we combine this with the evolution of $\vn{A}_{4,\gamma^{s/4}}^4$ and estimate to find
\begin{align*}
\frac{d}{dt}\int_M |A|^4\gamma^s\,d\mu
&\le
  - 4\int_M |\nabla_{(2)}A|^2|A|^2\,\gamma^s\,d\mu
  - 8\int_M |\IP{A}{\nabla_{(2)}A}|^2\,\gamma^s\,d\mu
\\&\qquad
  + 4s\int_M |A|^2\IP{A\nabla^p\gamma}{
    \nabla^q \nabla_q \nabla_p A
			}\gamma^{s-1}\,d\mu
  - 4s\int_M \IP{\nabla^p(A|A|^2)\nabla^q\gamma}{
    \nabla_q \nabla_p A
			}\gamma^{s-1}\,d\mu
\\&\qquad
+ \int_M (A*\nabla A*\nabla A*\nabla_{(2)}A + P_3^0(A)*P_3^2(A) + P_8^0(A))\gamma^s\,d\mu
\\&\qquad
  + s\int_M |A|^4\gamma_t\gamma^{s-1}\,d\mu
\\
	&\le -4\int_M |\nabla_{(2)}A|^2|A|^2\,\gamma^s\,d\mu
+ \int_M (A*\nabla A*\nabla A*\nabla_{(2)}A + P_3^0(A)*P_3^2(A) + P_8^0(A))\gamma^s\,d\mu
	 \\&\qquad
	 + c(\cg)\int_M (|A|^4|\nabla_{(2)}A| + |A|^3|\nabla_{(3)}A| + |A|^2|\nabla A||\nabla_{(2)}A| + |A|^7)\,\gamma^{s-1}\,d\mu
	 \\
	&\le (-4+\delta_1+\delta_5)\int_M |\nabla_{(2)}A|^2|A|^2\,\gamma^s\,d\mu
	 + \delta_3\int_M |\nabla_{(3)}A|^2\,\gamma^s\,d\mu
\\&\qquad
	 + (\delta_2+\delta_4+\delta_7)\int_M |A|^8\,\gamma^s\,d\mu
	 + \delta_6\int_M |\nabla A|^4\,\gamma^s\,d\mu
	 \\&\qquad
+ \int_M (A*\nabla A*\nabla A*\nabla_{(2)}A + P_3^0(A)*P_3^2(A) + P_8^0(A))\gamma^s\,d\mu
	 \\&\qquad
	 + c(\cg)^4\int_M |A|^4\,\gamma^{s-4}\,d\mu
	 \,.
\end{align*}
In the above, we used the estimates ($c$ varies from line to line, is fixed depending on $\delta_i$, $s$, $\tilde\gamma$ to be chosen)
\begin{align*}
	 c(\cg)\int_M |A|^4|\nabla_{(2)}A|\,\gamma^{s-1}\,d\mu
	&\le \delta_1\int_M |\nabla_{(2)}A|^2|A|^2\,\gamma^s\,d\mu
	+ c(\cg)^2\int_M |A|^6\,\gamma^{s-2}\,d\mu
\\
	&\le \delta_1\int_M |\nabla_{(2)}A|^2|A|^2\,\gamma^s\,d\mu
	   + \delta_2\int_M |A|^8\,\gamma^s\,d\mu
	+ c(\cg)^4\int_M |A|^4\,\gamma^{s-4}\,d\mu\,,
\end{align*}
\begin{align*}
	 c(\cg)\int_M |A|^3|\nabla_{(3)}A|\,\gamma^{s-1}\,d\mu
	&\le \delta_3\int_M |\nabla_{(3)}A|^2\,\gamma^s\,d\mu
	+ c(\cg)^2\int_M |A|^6\,\gamma^{s-2}\,d\mu
\\
	&\le \delta_3\int_M |\nabla_{(2)}A|^2|A|^2\,\gamma^s\,d\mu
	   + \delta_4\int_M |A|^8\,\gamma^s\,d\mu
	+ c(\cg)^4\int_M |A|^4\,\gamma^{s-4}\,d\mu\,,
\end{align*}
\begin{align*}
	 c(\cg)\int_M |A|^2|\nabla A||\nabla_{(2)}A|\,\gamma^{s-1}\,d\mu
	&\le \delta_5\int_M |\nabla_{(2)}A|^2|A|^2\,\gamma^s\,d\mu
	+ c(\cg)^2\int_M |\nabla A|^2|A|^2\,\gamma^{s-2}\,d\mu
\\
	&\le \delta_5\int_M |\nabla_{(2)}A|^2|A|^2\,\gamma^s\,d\mu
	   + \delta_6\int_M |\nabla A|^4\,\gamma^s\,d\mu
	+ c(\cg)^4\int_M |A|^4\,\gamma^{s-4}\,d\mu\,,
\end{align*}
\begin{align*}
	 c(\cg)\int_M |A|^7\,\gamma^{s-1}\,d\mu
	&\le \frac{\delta_7}{2}\int_M |A|^8\,\gamma^s\,d\mu
	+ c(\cg)^2\int_M |A|^6\,\gamma^{s-2}\,d\mu
\\
	&\le \delta_7\int_M |A|^8\,\gamma^s\,d\mu
	+ c(\cg)^4\int_M |A|^4\,\gamma^{s-4}\,d\mu\,.
\end{align*}
Now let us deal with the $P$-style terms by estimating
\begin{align*}
\int_M &(A*\nabla A*\nabla A*\nabla_{(2)}A + P_3^0(A)*P_3^2(A) + P_8^0(A))\gamma^s\,d\mu
\\
 &\le 
        c\int_M (|A||\nabla A|^2|\nabla_{(2)}A| + |A|^3(|A|^2|\nabla_{(2)}A| + |A||\nabla A|^2) + |A|^8)\gamma^s\,d\mu
\\
 &\le \delta_8\int_M |\nabla_{(2)}A|^2|A|^2\,\gamma^s\,d\mu
   + c\int_M (|\nabla A|^4 + |A|^8)\gamma^s\,d\mu\,.
\end{align*}
Combining, we find
\begin{equation}
\label{EQfebfin1}
\begin{aligned}
\frac{d}{dt}\int_M |A|^4\gamma^s\,d\mu
	&\le (-4+\delta_1+\delta_5+\delta_8)\int_M |\nabla_{(2)}A|^2|A|^2\,\gamma^s\,d\mu
	 + \delta_3\int_M |\nabla_{(3)}A|^2\,\gamma^s\,d\mu
\\&\qquad
	 + (c+\delta_2+\delta_4+\delta_7)\int_M |A|^8\,\gamma^s\,d\mu
	 + (c+\delta_6)\int_M |\nabla A|^4\,\gamma^s\,d\mu
	 \\&\qquad
	 + c(\cg)^4\int_M |A|^4\,\gamma^{s-4}\,d\mu
	 \,.
\end{aligned}
\end{equation}

Now we turn to the next term, $\int_M |\nabla A|^2\,\gamma^s\,d\mu$.
Since this is an $L^2$-norm, the evolution is standard.
Unfortunately, the typical approach with Lemma \ref{LMenergyhighervanillaconstantricci} interpolates between $\nabla_{(3)}A$ and $A$ in $L^2$, whereas we wish to go down instead to $A$ in $L^4$.
So we calculate
\begin{align*}
\frac{d}{dt}\int_M |\nabla A|^2\gamma^s\,d\mu
 &= 2\int_M \IP{\nabla A}{
			-\nabla^p\Delta\nabla_p\nabla A + (P_3^3 + P_5^1)(A)
}\gamma^s\,d\mu
\\&\qquad
+ \int_M |\nabla A|^2\IP{\vec{H}}{\BF}\gamma^s\,d\mu
  + s\int_M |\nabla A|^2\gamma_t\gamma^{s-1}\,d\mu
\\
 &= -2\int_M \IP{\nabla A}{
			\nabla^p\Delta\nabla_p\nabla A 
}\gamma^s\,d\mu
+ \int_M (\nabla A*(P_3^3(A) + P_5^1(A)))\gamma^s\,d\mu
\\&\qquad
  + s\int_M |\nabla A|^2\gamma_t\gamma^{s-1}\,d\mu
\,.
\end{align*}
For the first two terms, we find
\begin{align*}
 -2&\int_M \IP{\nabla A}{
			\nabla^p\Delta\nabla_p\nabla A 
}\gamma^s\,d\mu
+ \int_M (\nabla A*(P_3^3(A) + P_5^1(A)))\gamma^s\,d\mu
\\
&=
  - 2\int_M |\nabla_{(3)}A|^2\,\gamma^s\,d\mu
  + 2s\int_M \IP{\nabla^r A\nabla^q\gamma}{
    \Delta \nabla_{qr} A}\gamma^{s-1}\,d\mu
  - 2s\int_M \IP{\nabla^{qr}A\nabla^p\gamma}{
     \nabla_{pqr} A
			}\gamma^{s-1}\,d\mu
\\&\qquad
+ \int_M (\nabla A*(P_3^3(A) + P_5^1(A)))\gamma^s\,d\mu
\\
&=
  - 2\int_M |\nabla_{(3)}A|^2\,\gamma^s\,d\mu
  - 2s\int_M \IP{\nabla^{qr}A\nabla^p\gamma}{
     \nabla_{pqr} A
			}\gamma^{s-1}\,d\mu
\\&\qquad
  - 2s\int_M \IP{\gamma\nabla^{pr} A\nabla^q\gamma
              +  \gamma\nabla^{r} A\nabla^{pq}\gamma
              +  (s-1)\nabla^{r} A\nabla^{p}\gamma\nabla^{q}\gamma
}{
    \nabla_{pqr} A}\gamma^{s-2}\,d\mu
\\&\qquad
+ \int_M (\nabla A*(P_3^3(A) + P_5^1(A)))\gamma^s\,d\mu\,.
\end{align*}
Note that
\begin{align*}
  - 2s(s-1)\int_M \IP{\nabla^{r} A\nabla^{p}\gamma\nabla^{q}\gamma }{
	  \nabla_{pqr} A}\gamma^{s-2}\,d\mu
  \le c(\cg)^2\int_M (\nabla_{(3)}A * \nabla A)\,\gamma^{s-2}\,d\mu
	\,.
\end{align*}
Classifying and estimating terms in this way, using also $\gamma = \tilde\gamma\circ f$, we combine this with the evolution of $\vn{\nabla A}_{2,\gamma^{s/2}}^2$ to find (note that the $c$ here depends on $s$ and $N$)
\begin{equation}
\label{EQfeb2}
\begin{aligned}
	\frac{d}{dt}\int_M &|\nabla A|^2\gamma^s\,d\mu
\le
  - 2\int_M |\nabla_{(3)}A|^2\,\gamma^s\,d\mu
+ \int_M (\nabla A*(P_3^3(A) + P_5^1(A)))\gamma^s\,d\mu
\\&\qquad
  + c(\cg)^2\int_M (\nabla_{(3)}A * \nabla A)
			\gamma^{s-2}\,d\mu
  + c(\cg)\int_M (\nabla_{(3)}A * (\nabla_{(2)}A + A*\nabla A))
			\gamma^{s-1}\,d\mu
\\&\qquad
+ c(\cg)\int_M |\nabla A|^2(|\nabla_{(2)}A| + |A|^3)\gamma^{s-1}\,d\mu
\,.
\end{aligned}
\end{equation}
For the $P$-style terms we estimate (here $c$ depends additionally on $\delta_9, \delta_{10}, \delta_{11}$)
\begin{equation}
\label{EQfeb200}
\begin{aligned}
\int_M&\left([P_3^3(A)+P_5^1(A)]*\nabla A\right)\gamma^{s}d\mu
\notag\\*
&\le c\int_M\Big(\big[|A|^2|\nabla_{(3)}A|
                    + |\nabla_{(2)}A||\nabla A||A|
                    + |\nabla A|^3
                    + |A|^4|\nabla A|
                 \big]\,|\nabla A|\Big)\gamma^{s} d\mu
\notag\\
&\le c\int_M\Big(|A|^2|\nabla A||\nabla_{(3)}A|
                    + |\nabla_{(2)}A||\nabla A|^2|A|
                    + |\nabla A|^4
                    + |A|^4|\nabla A|^2
            \Big)\,\gamma^{s}d\mu
\notag\\&\le
     \delta_9\int_M |\nabla_{(3)}A|^2\gamma^sd\mu
   + \delta_{10}\int_M |\nabla_{(2)}A|^2|A|^2\gamma^sd\mu
   + c\int_M\Big(
                      |\nabla A|^4
                    + |A|^4|\nabla A|^2
            \Big)\,\gamma^{s}d\mu
\notag\\&\le
     \delta_9\int_M |\nabla_{(3)}A|^2\gamma^sd\mu
   + \delta_{10}\int_M |\nabla_{(2)}A|^2|A|^2\gamma^sd\mu
   + \delta_{11}\int_M |A|^8\gamma^sd\mu
   + c\int_M |\nabla A|^4 \,\gamma^{s}d\mu
\,.
\end{aligned}
\end{equation}
We additionally observe the estimate
\begin{equation}
\label{EQfeb20}
\begin{aligned}
c(\cg)^2&\int_M (\nabla_{(3)}A * \nabla A)
			\gamma^{s-2}\,d\mu
+ c(\cg)\int_M (\nabla_{(3)}A * (\nabla_{(2)}A + A*\nabla A))
			\gamma^{s-1}\,d\mu
\\&\qquad\qquad
+ c(\cg)\int_M |\nabla A|^2(|\nabla_{(2)}A| + |A|^3)\gamma^{s-1}\,d\mu
\\&\le
    \delta_{12}\int_M |\nabla_{(3)}A|^2\,\gamma^s\,d\mu
    + \delta_{13}\int_M |A|^8\,\gamma^s\,d\mu
    + c\int_M |\nabla A|^4 \,\gamma^{s}d\mu
    + c(\cg)^2\int_M |\nabla_{(2)}A|^2\,\gamma^{s-2}\,d\mu
\\&\qquad
    + c(\cg)^4\int_M |\nabla A|^2\,\gamma^{s-4}\,d\mu
    + c(\cg)^4\int_M |A|^4\,\gamma^{s-4}\,d\mu
\,.
\end{aligned}
\end{equation}
Since
\[
    c(\cg)^2\int_M |\nabla_{(2)}A|^2\,\gamma^{s-2}\,d\mu
\le
    \delta_{14}\int_M |\nabla_{(3)}A|^2\,\gamma^s\,d\mu
    + c(\cg)^4\int_M |\nabla A|^2\,\gamma^{s-4}\,d\mu
\]
we refine \eqref{EQfeb20} to
\begin{equation}
\label{EQfeb201}
\begin{aligned}
c(\cg)&\int_M (\nabla_{(3)}A * (\nabla_{(2)}A + (\cg)\nabla A + A*\nabla A))
			\gamma^{s-1}\,d\mu
+ c(\cg)\int_M |\nabla A|^2(|\nabla_{(2)}A| + |A|^3)\gamma^{s-1}\,d\mu
\\&\le
    (\delta_{12}+\delta_{14})\int_M |\nabla_{(3)}A|^2\,\gamma^s\,d\mu
    + \delta_{13}\int_M |A|^8\,\gamma^s\,d\mu
    + c\int_M |\nabla A|^4 \,\gamma^{s}d\mu
\\&\qquad
    + c(\cg)^4\int_M |\nabla A|^2\,\gamma^{s-4}\,d\mu
    + c(\cg)^4\int_M |A|^4\,\gamma^{s-4}\,d\mu
\,.
\end{aligned}
\end{equation}
Combining \eqref{EQfeb200} and \eqref{EQfeb201} with \eqref{EQfeb2} we find
\begin{equation}
\label{EQfebfin2}
\begin{aligned}
\frac{d}{dt}\int_M |\nabla A|^2\gamma^s\,d\mu
&\le
  - (2-\delta_9-\delta_{12}-\delta_{14})\int_M |\nabla_{(3)}A|^2\,\gamma^s\,d\mu
   + \delta_{10}\int_M |\nabla_{(2)}A|^2|A|^2\gamma^sd\mu
\\&\qquad
    + (\delta_{11}+\delta_{13})\int_M |A|^8\,\gamma^s\,d\mu
    + c\int_M |\nabla A|^4 \,\gamma^{s}d\mu
\\&\qquad
    + c(\cg)^4\int_M |\nabla A|^2\,\gamma^{s-4}\,d\mu
    + c(\cg)^4\int_M |A|^4\,\gamma^{s-4}\,d\mu
\,.
\end{aligned}
\end{equation}
Taking the final estimates \eqref{EQfebfin1} and \eqref{EQfebfin2} together, we obtain
\begin{equation}
\label{EQfebfin3}
\begin{aligned}
\frac{d}{dt}\int_M (|A|^4+|\nabla A|^2)\gamma^s\,d\mu
	&\le
  - (2-\delta_3-\delta_9-\delta_{12}-\delta_{14})\int_M |\nabla_{(3)}A|^2\,\gamma^s\,d\mu
\\&\qquad
  - (4-\delta_1-\delta_5-\delta_8-\delta_{10})\int_M |\nabla_{(2)}A|^2|A|^2\,\gamma^s\,d\mu
\\&\qquad
	 + (c+\delta_2+\delta_4+\delta_7+\delta_{11}+\delta_{13})\int_M |A|^8\,\gamma^s\,d\mu
	 + (c+\delta_6)\int_M |\nabla A|^4\,\gamma^s\,d\mu
\\&\qquad
    + c(\cg)^4\int_M |\nabla A|^2\,\gamma^{s-4}\,d\mu
    + c(\cg)^4\int_M |A|^4\,\gamma^{s-4}\,d\mu
\,.
\end{aligned}
\end{equation}
With appropriate choices for $\delta_i$ we find
\begin{equation}
\label{EQfebfin4}
\begin{aligned}
\frac{d}{dt}\int_M (|A|^4+|\nabla A|^2)\gamma^s\,d\mu
	&\le
  - \frac32\int_M |\nabla_{(3)}A|^2\,\gamma^s\,d\mu
  - 3\int_M |\nabla_{(2)}A|^2|A|^2\,\gamma^s\,d\mu
\\&\qquad
	 + c\int_M |A|^8\,\gamma^s\,d\mu
	 + c\int_M |\nabla A|^4\,\gamma^s\,d\mu
\\&\qquad
    + c(\cg)^4\int_M |\nabla A|^2\,\gamma^{s-4}\,d\mu
    + c(\cg)^4\int_M |A|^4\,\gamma^{s-4}\,d\mu
\,.
\end{aligned}
\end{equation}
We deal with each of the integrals with a large coefficient in turn.
By the Michael-Simon Sobolev inequality we estimate
\begin{align*}
 \int_M |\nabla A|^4\,\gamma^s\,d\mu
&\le
	c\bigg(
		\int_M |\nabla A|^2|\nabla_{(2)}A|\,\gamma^{\frac{3s}{4}}\,d\mu
	 \bigg)^\frac43
      + c\bigg(
		\int_M |A||\nabla A|^3\,\gamma^{\frac{3s}{4}}\,d\mu
	 \bigg)^\frac43
\\&\qquad
      + c\bigg(
		(\cg)\int_M |\nabla A|^3\,\gamma^{\frac{3s-4}{4}}\,d\mu
	 \bigg)^\frac43
\\
&\le
	c\bigg(
		\int_M |\nabla A|^4\,\gamma^s\,d\mu
	 \bigg)^\frac23
	 \bigg(
		\int_M |\nabla_{(2)}A|^2\,\gamma^{\frac{s}{2}}\,d\mu
	 \bigg)^\frac23
      + c\bigg(
		\int_{[\gamma>0]} |A|^4d\mu
	 \bigg)^\frac13
		\int_M |\nabla A|^4\,\gamma^s\,d\mu
\\&\qquad
      + c(\cg)^\frac43\bigg(
		\int_{[\gamma>0]} |\nabla A|^2\,d\mu
	 \bigg)^\frac23
         \bigg(
		\int_M |\nabla A|^4\,\gamma^s\,d\mu
	 \bigg)^\frac23
\\
&\le
	c\bigg(
		\int_M |\nabla_{(2)}A|^2\,\gamma^{\frac{s}{2}}\,d\mu
	 \bigg)^2
      +  \Big(\frac12 + c\varepsilon_0^\frac13
	 \Big)
		\int_M |\nabla A|^4\,\gamma^s\,d\mu
      + c(\cg)^4\bigg(
		\int_{[\gamma>0]} |\nabla A|^2\,d\mu
	 \bigg)^2
\end{align*}
Now observe the intermediate estimate
\begin{align*}
	\bigg(
		\int_M |\nabla_{(2)}A|^2\,\gamma^{\frac{s}{2}}\,d\mu
	 \bigg)^2
&\le
	c\bigg(
		\int_M |\nabla_{(3)}A||\nabla A|\,\gamma^{\frac{s}{2}}\,d\mu
	 \bigg)^2
      + c\bigg(
		(\cg)\int_M |\nabla_{(2)}A||\nabla A|\,\gamma^{\frac{s-2}{2}}\,d\mu
	 \bigg)^2
\\
&\le \frac12
	\bigg(
		\int_M |\nabla_{(2)}A|^2\,\gamma^{\frac{s}{2}}\,d\mu
	 \bigg)^2
 + 
	c\varepsilon_0
		\int_M |\nabla_{(3)}A|^2\,\gamma^s\,d\mu
      + c\bigg(
		(\cg)^2\int_M |\nabla A|^2\,\gamma^{\frac{s-4}{2}}\,d\mu
	 \bigg)^2\,,
\end{align*}
that is,
\[
	\bigg(
		\int_M |\nabla_{(2)}A|^2\,\gamma^{\frac{s}{2}}\,d\mu
	 \bigg)^2
	\le c\varepsilon_0
		\int_M |\nabla_{(3)}A|^2\,\gamma^s\,d\mu
      + c(\cg)^4\bigg(
		\int_M |\nabla A|^2\,\gamma^{\frac{s-4}{2}}\,d\mu
	 \bigg)^2\,.
\]
Combining this with the above yields
\begin{equation}
\label{EQfebfin5}
\begin{aligned}
 \int_M |\nabla A|^4\,\gamma^s\,d\mu
&\le
	 c\varepsilon_0
		\int_M |\nabla_{(3)}A|^2\,\gamma^s\,d\mu
      +  \Big(\frac12 + c\varepsilon_0^\frac13
	 \Big)
		\int_M |\nabla A|^4\,\gamma^s\,d\mu
\\&\qquad
      + c(\cg)^4\bigg(
		\int_{[\gamma>0]} |\nabla A|^2\,d\mu
	 \bigg)^2\,.
\end{aligned}
\end{equation}
A similar argument applies to the integral $\int_M|A|^8\,\gamma^s\,d\mu$ (see the derivation of \eqref{EQinter} for details), yielding:
\begin{equation}
\label{EQfebfin6}
\begin{aligned}
\int_M |A|^8\,\gamma^s\,d\mu
&\le
     c\vn{A}_{4,[\gamma>0]}^\frac43
      \int_M \big(|\nabla A|^4 + |A|^8\big)\,\gamma^s\,d\mu
   + c(c_\gamma)^4
      \vn{A}_{4,[\gamma>0]}^\frac{16}3
\,.
\end{aligned}
\end{equation}
Combining \eqref{EQfebfin5}, \eqref{EQfebfin6} and taking a sufficiently small $\varepsilon_0$ such that the left hand side absorbs, we have
\begin{equation}
\label{EQfebfin7}
\begin{aligned}
 \int_M (|A|^8+|\nabla A|^4)\,\gamma^s\,d\mu
&\le
	 c\varepsilon_0
		\int_M |\nabla_{(3)}A|^2\,\gamma^s\,d\mu
      + c(\cg)^4\varepsilon_0
      \,.
\end{aligned}
\end{equation}
Combining \eqref{EQfebfin7} with \eqref{EQfebfin4} above and choosing $\varepsilon_0$ again if necessary, we finally arrive at the estimate
\begin{align*}
&\rD{}{t}\int_M (|A|^4+|\nabla A|^2)\gamma^sd\mu
   +   \int_M \Big(|\nabla_{(3)}A|^2
                    + |\nabla_{(2)}A|^2|A|^2
                    + |\nabla A|^4
                    + |A|^4|\nabla A|^2
                    + |A|^8
            \Big)\,\gamma^{s}d\mu
\\
&\qquad \le c(\cg)^4\varepsilon_0
\,.
\end{align*}
Integrating the above finishes the proof.
\end{proof}%}}}

\begin{prop}
\label{p:prop44n42}
Suppose $f:M^4\times[0,T^*]\rightarrow\R^{N}$ evolves by
\eqref{EQwf} and $\gamma$ is a cutoff function as in \eqref{EQgamma}.
Then there is an $\varepsilon_0 = \varepsilon_0(N)$ such that if
\begin{equation}
\label{eq9n42}
\varepsilon = \sup_{[0,T^*]}\int_{[\gamma>0]}|A|^4d\mu\le\varepsilon_0
\end{equation}
then for any $t\in[0,T^*]$ we have
\begin{equation}
\label{eq10n42}
  \begin{split}
&\int_{[\gamma=1]} |\nabla_{(2)} A|^2 d\mu
 + \int_0^t\int_{[\gamma=1]} 
               |\nabla_{(4)}A|^2
                    + |\nabla_{(3)}A|^2|A|^2
                    + |\nabla_{(2)}A|^2|A|^4
\\
&\hskip+5cm
                    + |\nabla_{(2)}A|^2|\nabla A|^2
                    + |\nabla A|^4|A|^2
                    + |\nabla A|^2|A|^6
                    + |A|^{10}
                 \,          d\mu\,d\tau
\\
&\qquad\qquad\qquad 
  \le \int_{[\gamma > 0]} |\nabla_{(2)} A|^2 d\mu\Big|_{t=0}
   + ct(\cg)^6\varepsilon^\frac12
            \big(1+[(\cg)^4\mu(f_0)]^\frac12
            \big)
            \big(1 + \varepsilon^\frac12
	    \big)
\,,
  \end{split}
\end{equation}
where $c=c(N)$.
\end{prop}
\begin{proof}%{{{
Lemma \ref{LMenergyhighervanillaconstantricci} with $k=2$, $s=16$ gives
\begin{align*}
&\rD{}{t}\int_M |\nabla_{(2)} A|^2\gamma^{16}d\mu
 + (2-\theta)\int_M |\nabla_{(4)}A|^2\gamma^{16}d\mu
\\
&\qquad
\le c(\cg)^8\int_{[\gamma>0]}|A|^2\,d\mu
            + c\int_M \left([P_3^4(A)+P_5^2(A)]*\nabla_{(2)} A\right)\gamma^{16} d\mu
\,.
\end{align*}
We estimate the $P$-style terms as follows:
\begin{align}
c\int_M &\left([P_3^4(A)+P_5^2(A)]*\nabla_{(2)} A\right)\gamma^{16} d\mu
\notag\\
&\le \theta\int_M |\nabla_{(4)}A|^2\gamma^{16}d\mu
 + c\int_M |\nabla_{(2)}A|^2|A|^4\,\gamma^{16}d\mu
\notag\\&\qquad
 + c\int_M \nabla_{(2)}A*(\nabla_{(3)}A*\nabla A*A + 
                          \nabla_{(2)}A*\nabla A*\nabla A
                         )\,\gamma^{16}d\mu
\notag\\&\qquad
 + c\int_M \nabla_{(2)}A*(\nabla_{(2)}A*A*A*A*A 
                        + \nabla A*\nabla A*A*A*A
                         )\,\gamma^{16}d\mu
\notag\\
&\le \theta\int_M \big(|\nabla_{(4)}A|^2
 +  |\nabla_{(3)}A|^2|A|^2
 + |\nabla A|^4|A|^2\big)\,\gamma^{16}d\mu
\notag\\&\qquad
 + c\int_M \big(|\nabla_{(2)}A|^2|A|^4
 +  |\nabla_{(2)}A|^2|\nabla A|^2\big)\,\gamma^{16}d\mu
\,.
\label{EQnewer1}
\end{align}
The equality
\begin{align*}
\int_M |\nabla A|^4|A|^2\,\gamma^{16}d\mu
	&= -\int_M \IP{\Delta A}{A}|\nabla A|^2|A|^2\,\gamma^{16}d\mu
\\&\qquad
	  - 2\int_M \IP{\nabla_pA}{A}\IP{\nabla^p\nabla A}{\nabla A}|A|^2\,\gamma^{16}d\mu
\\&\qquad
	  - 2\int_M \IP{\nabla_pA}{A}|\nabla A|^2\IP{\nabla^pA}{A}\,\gamma^{16}d\mu
\\&\qquad
	  - 16\int_M \IP{\nabla_pA}{A\nabla^p\gamma}|\nabla A|^2|A|^2\,\gamma^{15}d\mu
\\
	&= -\int_M (\nabla_{(2)}A*A)*(\nabla A*\nabla A*A*A)\,\gamma^{16}d\mu
\\&\qquad
	  - \frac12\int_M \big|\nabla|A|^2\big|^2|\nabla A|^2\,\gamma^{16}d\mu
\\&\qquad
	  - 16\int_M \IP{\nabla_pA}{A\nabla^p\gamma}|\nabla A|^2|A|^2\,\gamma^{15}d\mu
\end{align*}
implies the estimate
\begin{align}
\int_M &|\nabla A|^4|A|^2\,\gamma^{16}d\mu
	  + \frac12\int_M \big|\nabla|A|^2\big|^2|\nabla A|^2\,\gamma^{16}d\mu
\notag\\
	&= -\int_M (\nabla_{(2)}A*A)*(\nabla A*\nabla A*A*A)\,\gamma^{16}d\mu
\notag\\&\qquad
	  - 16\int_M \IP{\nabla_pA}{A\nabla^p\gamma}|\nabla A|^2|A|^2\,\gamma^{15}d\mu
\notag\\ &\le
	\frac14\int_M |\nabla A|^4|A|^2\,\gamma^{16}d\mu
	+ c\int_M |\nabla_{(2)} A|^2|A|^4\,\gamma^{16}d\mu
	  - 16\int_M \IP{\nabla_pA}{A\nabla^p\gamma}|\nabla A|^2|A|^2\,\gamma^{15}d\mu
\,.
\label{EQnewer2}
\end{align}
To deal with the last term we use Young's inequality twice (on the first line with exponents 4 and $\frac43$, for the second with exponents $3$ and $\frac32$) to estimate
\begin{align}
	  - 16\int_M \IP{\nabla_pA}{A\nabla^p\gamma}|\nabla A|^2&|A|^2\,\gamma^{15}d\mu
\notag\\ &\le \frac14\int_M |\nabla A|^4|A|^2\,\gamma^{16}d\mu
	+ c(c_\gamma)^4\int_M |A|^6 \gamma^{12}d\mu
\notag\\
 &\le \frac14\int_M |\nabla A|^4|A|^2\,\gamma^{16}d\mu
	+ c\int_M |A|^{10}\gamma^{16}d\mu
	+ c(c_\gamma)^6\int_M |A|^4 \gamma^{10}d\mu
\,.
\label{EQnewer3}
\end{align}
Combining \eqref{EQnewer2} and \eqref{EQnewer3} we find
\begin{align*}
\int_M &|\nabla A|^4|A|^2\,\gamma^{16}d\mu
	  + \frac12\int_M \big|\nabla|A|^2\big|^2|\nabla A|^2\,\gamma^{16}d\mu
\\ &\le
	\frac12\int_M |\nabla A|^4|A|^2\,\gamma^{16}d\mu
	+ c\int_M (|\nabla_{(2)} A|^2|A|^4+|A|^{10})\,\gamma^{16}d\mu
	+ c(c_\gamma)^6\vn{A}^4_{4,[\gamma>0]}  
\,,
\end{align*}
which after absorbing yields
\begin{align}
\int_M |\nabla A|^4|A|^2\,\gamma^{16}d\mu
	  &+ \int_M \big|\nabla|A|^2\big|^2|\nabla A|^2\,\gamma^{16}d\mu
\notag\\ &\le
	  c\int_M (|\nabla_{(2)} A|^2|A|^4+|A|^{10})\,\gamma^{16}d\mu
	+ c(c_\gamma)^6\vn{A}^4_{4,[\gamma>0]}  
\,.
\label{EQnewer4}
\end{align}
Combining \eqref{EQnewer4} with \eqref{EQnewer1}, we find
\begin{align}
c\int_M &\left([P_3^4(A)+P_5^2(A)]*\nabla_{(2)} A\right)\gamma^{16} d\mu
+ c\int_M |\nabla A|^4|A|^2\,\gamma^{16}d\mu
\notag\\
&\le \theta\int_M \big(|\nabla_{(4)}A|^2
 +  |\nabla_{(3)}A|^2|A|^2\big)\,\gamma^{16}d\mu
\notag\\&\qquad
 + c\int_M \big(|\nabla_{(2)}A|^2|A|^4 + |\nabla_{(2)}A|^2|\nabla A|^2 + |A|^{10}\big)\,\gamma^{16}d\mu
	+ c(c_\gamma)^6\vn{A}^4_{4,[\gamma>0]}  
\,.
\label{EQnewer5}
\end{align}
Now we require the multiplicative Sobolev inequality in (iv) of Corollary \ref{MS1cor}.
This is particularly useful for estimating the right hand side of \eqref{EQnewer5}:
\begin{align*}
\int_M &\big(
          |\nabla_{(3)}A|^2|A|^2 + |\nabla_{(2)}A|^2|A|^4 + |\nabla_{(2)}A|^2|\nabla A|^2
        + |\nabla A|^2|A|^6 + |A|^{10}\big)\gamma^{16}\,d\mu
\\&\le
           (\theta + c\vn{A}_{4,[\gamma>0]}^4)\int_M
             |\nabla_{(4)}A|^2
                 \gamma^{16}\,d\mu
  + c(\cg)^6\vn{A}_{4,[\gamma>0]}^2
            \big(1+[(\cg)^4\mu_\gamma(f_t)]^\frac12
            \big)
            \big(1 + \vn{A}_{4,[\gamma>0]}^2
	    \big)
\,.
\end{align*}
where $\theta\in(0,1)$ and $c = c(s,\theta,N)$ is an absolute constant.
Applying this and absorbing, we find
\begin{align*}
&\rD{}{t}\int_M |\nabla_{(2)} A|^2\,\gamma^{16}d\mu
 + \int_M \big(
               |\nabla_{(4)}A|^2
                    + |\nabla_{(3)}A|^2|A|^2
                    + |\nabla_{(2)}A|^2|A|^4
\\
&\hskip+5cm
                    + |\nabla_{(2)}A|^2|\nabla A|^2
                    + |\nabla A|^4|A|^2
                    + |\nabla A|^2|A|^6
                    + |A|^{10}\big)
                 \,\gamma^{16}d\mu\,d\tau
\\
&\quad
\le c(\cg)^6[(\cg)^4\mu_\gamma(f_t)]^\frac12\varepsilon^\frac12
  + c(\cg)^6\varepsilon^\frac12
            \big(1+[(\cg)^4\mu_\gamma(f_t)]^\frac12
            \big)
            \big(1 + \varepsilon^\frac12
	    \big)
\\
&\quad
\le 
    c(\cg)^6\varepsilon^\frac12
            \big(1+[(\cg)^4\mu_\gamma(f_t)]^\frac12
            \big)
            \big(1 + \varepsilon^\frac12
	    \big)
\,.
\end{align*}
Integrating finishes the proof.
\end{proof}%}}}

For $L^\infty$ control we use the following estimated form of Lemma
\ref{LMenergyhighervanillaconstantricci}.
The proof of Proposition \ref{p:prop45} carries over essentially unchanged from
\cite{kuwert2002gfw}.
The $n=2$ case of Proposition \ref{p:prop46} is very similar to
\cite{kuwert2002gfw} for $n=2$.
Therefore we focus only on the case $n=4$ in the proof below.

\begin{prop}
\label{p:prop45}
Suppose $f:M^n\times[0,T^*]\rightarrow\R^{N}$ evolves by
\eqref{EQwf} and $\gamma$ is a cutoff function as in \eqref{EQgamma}.
Then, for $s \ge 2k+4$ the following estimate holds:
\begin{equation}
\label{e:prop45}
  \begin{split}
&\rD{}{t}\int_{M} |\nabla_{(k)}A|^2\gamma^s d\mu
 + \int_{M} |\nabla_{(k+2)}A|^2 \gamma^s d\mu \\
&\qquad\qquad\qquad
  \le c\vn{A}^4_{\infty,[\gamma>0]}\int_M|\nabla_{(k)}A|^2\gamma^sd\mu
       + c(\cg)^{2k}\vn{A}^2_{2,[\gamma>0]}(1+\vn{A}^4_{\infty,[\gamma>0]}) \\
  \end{split}
\end{equation}
where $c=c(N)$.
\end{prop}

\begin{prop}
\label{p:prop46}
Let $n\in\{2,4\}$.
Suppose $f:M^n\times[0,T^*]\rightarrow\R^{N}$ evolves by
\eqref{EQwf} and $\gamma$ is a cutoff function as in \eqref{EQgamma}.
Then there is an $\varepsilon_0 = \varepsilon_0(n,N)$ such that if
\begin{equation}
\label{eq13}
\sup_{[0,T^*]}\int_{[\gamma>0]}|A|^nd\mu\le\varepsilon_0\,,
\end{equation}
we can conclude
\begin{equation}
\label{eq14}
\vn{\nabla_{(k)}A}^2_{\infty,[\gamma=1]}
\le (\cg)^{2k+2}c\big(k,T^*,[(\cg)^4\mu(f_0)],N,\alpha_0(0),\ldots,\alpha_0(k+3))
\,,
\end{equation}
where $\alpha_0(j) = \mu(f_0)^{\frac{j-1}{2}}\vn{\nabla_{(j)}A}_{2,[\gamma>0]}^2\big|_{t=0}$.
\end{prop}
\begin{proof}%{{{
The idea is to use our previous estimates and then integrate.
We fix $\gamma$ and consider nested cutoff functions $\gamma_{\sigma,\tau}$.
Define for $0\le\sigma<\tau\le 1$ functions $\gamma_{\sigma,\tau} =
\psi_{\sigma,\tau}\circ\gamma$ satisfying $\gamma_{\sigma,\tau}=0$ for $\gamma\le\sigma$ and
$\gamma_{\sigma,\tau}=1$ for $\gamma\ge\tau$.  The function $\psi_{\sigma,\tau}$ is chosen such that
$\gamma_{\sigma,\tau}$ satisfies inequalities \eqref{EQgamma}, with the estimate
\[
 c_{\gamma_{\sigma,\tau}} = \vn{\nabla\psi_{\sigma,\tau}}_\infty\cdot \cg\,.
\]
Note that $\vn{\nabla\psi_{\sigma,\tau}}_\infty$ depends only on $\sigma$ and
$\tau$, so that when they are fixed we have $c_{\gamma_{\sigma,\tau}} \le
c\,\cg$. We use this below.

As noted above, we present the proof for $n=4$ only and refer to \cite{kuwert2002gfw} for $n=2$.
We first estimate
\begin{align*}
	(\cg)^\frac23&
	\bigg(\int_M |\nabla_{(3)}A|^2\,\gamma^s\,d\mu\bigg)^\frac43
\\
	&\le c(\cg)^\frac23\bigg(
                            \int_M |\nabla_{(4)}A|\,|\nabla_{(2)}A|\,\gamma^s\,d\mu
                           \bigg)^\frac{4}{3}
	   + c(\cg)^2
                           \bigg(
                            \int_M |\nabla_{(3)}A|\,|\nabla_{(2)}A|\,\gamma^{s-1} d\mu
                           \bigg)^\frac{4}{3}
	   \\
	   &\le c(\cg)^\frac23\bigg(\int_M |\nabla_{(4)}A|^2\,\gamma^s d\mu\bigg)^\frac23
	      \bigg(\int_M |\nabla_{(2)}A|^2\,\gamma^{s} d\mu\bigg)^\frac23
	   \\&\qquad
	   + c(\cg)^2
                           \bigg(
                            \int_M |\nabla_{(3)}A|^2\,\gamma^s d\mu
                           \bigg)^\frac{2}{3}
                           \bigg(
                            \int_M |\nabla_{(2)}A|^2\,\gamma^{s-2} d\mu
                           \bigg)^\frac{2}{3}
	   \\
	   &\le
	\frac12(\cg)^\frac23\bigg(\int_M |\nabla_{(3)}A|^2\,\gamma^s\,d\mu\bigg)^\frac43
	+ c\int_M |\nabla_{(4)}A|^2\,\gamma^s d\mu
	   \\&\qquad
	+ c(\cg)^2\bigg(
		       \int_M |\nabla_{(2)}A|^2\,\gamma^{s} d\mu
                  \bigg)^2
	+ c(\cg)^\frac{10}{3}\bigg(
                       \int_M |\nabla_{(2)}A|^2\,\gamma^{s-2} d\mu
                  \bigg)^\frac{4}{3}
\end{align*}
which upon absorption yields
\begin{align}
	(\cg)^\frac23&
	\bigg(\int_M |\nabla_{(3)}A|^2\,\gamma^s\,d\mu\bigg)^\frac43
	   \le 
	  c\int_M |\nabla_{(4)}A|^2\,\gamma^s d\mu
	   \notag\\&\qquad
	+ c(\cg)^2\bigg(
		       \int_M |\nabla_{(2)}A|^2\,\gamma^{s} d\mu
                  \bigg)^2
	+ c(\cg)^\frac{10}{3}\bigg(
                       \int_M |\nabla_{(2)}A|^2\,\gamma^{s-2} d\mu
                  \bigg)^\frac{4}{3}
	\,.
\label{EQa1}
\end{align}
We apply the estimate \eqref{EQa1} with $\gamma = \gamma_{\sfrac12,\sfrac34}$ to find
\begin{align}
	(\cg)^\frac23
	\vn{\nabla_{(3)}A}_{2,[\gamma\ge\sfrac34]}^\frac83
	   \le 
	  c\vn{\nabla_{(4)}A}_{2,[\gamma\ge\sfrac12]}^2
	+ c(\cg)^2
	   \vn{\nabla_{(2)}A}_{2,[\gamma\ge\sfrac12]}^4
	+ c(\cg)^\frac{10}{3}
	   \vn{\nabla_{(2)}A}_{2,[\gamma\ge\sfrac12]}^\frac83
	\,.
\label{EQa2}
\end{align}
Taking $\varepsilon_0$ as in \eqref{eq9n42}, we can apply the estimate
\eqref{eq10n42} of Proposition \ref{p:prop44n42} for $\gamma =
\gamma_{\sfrac14,\sfrac12}$.
In particular we have the estimate
\begin{equation}
\label{EQaa}
  \vn{\nabla_{(2)}A}_{2,[\gamma\ge\sfrac12]}^2
  \le \alpha_0(2)
   + cT^*(\cg)^6\varepsilon^\frac12
            \big(1+[(\cg)^4\mu(f_0)]^\frac12
            \big)
            \big(1 + \varepsilon^\frac12
	    \big)\,.
\end{equation}
Combining \eqref{EQa2} with \eqref{EQaa} we have
\[
	\vn{\nabla_{(3)}A}_{2,[\gamma\ge\sfrac34]}^\frac83
\le 
	  c\vn{\nabla_{(4)}A}_{2,[\gamma\ge\sfrac12]}^2
 + c
\]
where $c$ depends on $T^*$, $\alpha_0(2)$, $[(\cg)^4\mu(f_0)]$ and $N$ as in \eqref{eq14}.
We have also used $\varepsilon \le 1$.
From now until the rest of this proof all constants $c$ (that may vary from
line to line) shall depend on these quantities.
Later in the proof $c$ may additionally depend on $\alpha_0(k)$; when this
occurs it will be explicitly stated.

From Proposition \ref{MS2prop} we find, using $\gamma_{\sfrac34,\sfrac78}$ instead of $\gamma$,
\begin{align*}
	\int_0^t
	&\vn{A}_{\infty,[\gamma\ge\sfrac78]}^4
	\,d\tau
	\\
	&\le c\int_0^t\vn{A}_{2,[\gamma\ge\sfrac34]}^\frac43\bigg(
             \vn{\nabla_{(3)}A}^\frac83_{2,[\gamma\ge\sfrac34]}
+ (c_\gamma)^\frac{16}{3}\big(
                    1 + \vn{A}_{4,[\gamma\ge\sfrac34]}^\frac{16}{3}
                      + [(c_\gamma)^4\mu_\gamma(f_t)]^\frac43
              \big)
             \bigg)\,d\tau
	\\
	&\le c(\mu(f_0)\varepsilon_0)^\frac13
	\int_0^t\bigg(
             \vn{\nabla_{(4)}A}^2_{2,[\gamma\ge\sfrac34]}
		+ c
             \bigg)\,d\tau
	\\
	&\le c\varepsilon_0^\frac13
	\int_0^t
             \vn{\nabla_{(4)}A}^2_{2,[\gamma\ge\sfrac34]}
             \,d\tau
           + c\varepsilon_0^\frac13
\,.
\end{align*}
Now from Proposition \ref{p:prop44n42} with $\gamma = \gamma_{\sfrac12,\sfrac34}$ we have the estimate
\[
\int_0^t
   \vn{\nabla_{(4)}A}^2_{2,[\gamma\ge\sfrac34]}
\,d\tau
  \le \alpha_0(2)
   + cT^*(\cg)^6\varepsilon^\frac12
            \big(1+[(\cg)^4\mu(f_0)]^\frac12
            \big)
            \big(1 + \varepsilon^\frac12
	    \big)
 \le c
     \,.
\]
In the above we used $\varepsilon \le 1$.
This implies
\begin{align}
	\int_0^t
	\vn{A}_{\infty,[\gamma\ge\sfrac78]}^4
	\,d\tau
	\le c\varepsilon^\frac13
        \,.
\label{EQA4inLinf}
\end{align}
Now, integrating Proposition \ref{p:prop45} with $\gamma =
\gamma_{\sfrac78,\sfrac{15}{16}}$ yields an inequality of the form
\[
  \alpha(t) \le \beta(t) + \int_c^t \lambda(\tau)\alpha(\tau)d\tau,
\]
where
\begin{align*}
\alpha(t) &= \vn{\nabla_{(k)}A}_{2,[\gamma\ge\sfrac{15}{16}]}^2\,,
\\
\beta(t) &= \vn{\nabla_{(k)}A}_{2,[\gamma\ge\sfrac{7}{8}]}^2\Big|_{t=0}
     + c\int_0^t
          \left[\vn{A}^2_{2,[\gamma\ge\sfrac{7}{8}]}
                \Big(1+\vn{A}^4_{\infty,[\gamma\ge\sfrac{7}{8}]}\Big)\right]d\tau,
\intertext{and}
\lambda(t) &=
        \vn{A}^4_{\infty,[\gamma\ge\sfrac{7}{8}]}.
\end{align*}
Noting that $\beta$ and $\int \lambda d\tau$ are bounded as shown above, we can
invoke Gr\"onwall's inequality and conclude
\begin{equation}
 \vn{\nabla_{(k)}A}^2_{2,[\gamma\ge\sfrac{15}{16}]}
\le \beta(t) + \int_0^t \beta(\tau)\lambda(\tau)e^{\int_\tau^t\lambda(\nu)d\nu}d\tau
\le c\,,
\label{EQl2control}
\end{equation}
where now $c$ depends additionally on $\alpha_0(k)$.
Therefore using \eqref{MS2secondstatement} with $\gamma_{\sfrac{15}{16},\sfrac{31}{32}}$ we have
\begin{equation}
\vn{A}_{\infty,[\gamma\ge\sfrac{31}{32}]}
\le c\varepsilon_0^\frac12
\,.
\label{EQlinfty}
\end{equation}
Finally, using \eqref{MS2n4} with $T=\nabla_{(k)}A$ and $\gamma=\gamma_{\sfrac{31}{32},1}$ we
obtain for any $l\in\N_0$
\begin{align*}
	\vn{\nabla_{(l)}A}_{\infty,[\gamma=1]}^3
	&\le c\vn{\nabla_{(l)}A}_{2,[\gamma\ge\sfrac{31}{32}]}\Big(
	\vn{\nabla_{(l+3)}A}_{2,[\gamma\ge\sfrac{31}{32}]}^2
	+ \vn{A^3\,\nabla_{(l)}A}_{2,[\gamma\ge\sfrac{31}{32}]}^2
	+ (c_\gamma)^2\vn{A\,\nabla_{(l+1)}A}_{2,[\gamma\ge\sfrac{31}{32}]}^2
\notag\\&\qquad\qquad
	+ (c_\gamma)^2\vn{\nabla_{(l)}A\nabla A}_{2,[\gamma\ge\sfrac{31}{32}]}^2
	+ (c_\gamma)^4\vn{\nabla_{(l+1)}A}_{2,[\gamma\ge\sfrac{31}{32}]}^2
	+ (c_\gamma)^6\vn{\nabla_{(l)}A}_{2,[\gamma\ge\sfrac{31}{32}]}^2
             \Big)
\\
	&\le c\vn{\nabla_{(l)}A}_{2,[\gamma\ge\sfrac{31}{32}]}\Big(
	\vn{\nabla_{(l+3)}A}_{2,[\gamma\ge\sfrac{31}{32}]}^2
	+
 \vn{A}_{\infty,[\gamma\ge\sfrac{31}{32}]}^6
 \vn{\nabla_{(l)}A}_{2,[\gamma\ge\sfrac{31}{32}]}^2
\notag\\&\qquad\qquad
	+ (c_\gamma)^2
 \vn{A}_{\infty,[\gamma\ge\sfrac{31}{32}]}^2
 \vn{\nabla_{(l+1)}A}_{2,[\gamma\ge\sfrac{31}{32}]}^2
	+ (c_\gamma)^2
 \vn{\nabla_{(l)}A}_{2,[\gamma\ge\sfrac{31}{32}]}
 \vn{\nabla A}_{2,[\gamma\ge\sfrac{31}{32}]}
\notag\\&\qquad\qquad
	+ (c_\gamma)^4\vn{\nabla_{(l+1)}A}_{2,[\gamma\ge\sfrac{31}{32}]}^2
	+ (c_\gamma)^6\vn{\nabla_{(l)}A}_{2,[\gamma\ge\sfrac{31}{32}]}^2
             \Big)
\,.
\end{align*}
The estimate \eqref{EQl2control}, applied for $k = 1, k, k+1, k+3$ then yields
\[
	\vn{\nabla_{(l)}A}_{\infty,[\gamma=1]} \le c\,.
\]
Tracing through the dependence of the above $c$ on $(\cg)$ and the
scale-invariant $[(\cg)^4\mu(f_0)]$ reveals the structure of the constant given
in \eqref{eq14}.
This completes the proof of the proposition.
\end{proof}%}}}

\begin{proof}[Proof of Theorem \ref{TMlifespan}]%{{{
The proof for $n=2$ follows exactly as in \cite{kuwert2002gfw}.
It should be noted that the argument given in \cite{mythesis} results in a
constant that depends on the measure of the initial immersion.
That was natural in the setting of \cite{mythesis} where volume was a-priori
along the flow possibly not controlled, depending on the given global force
field.
Here, we have no external forcing term, and so it is desirable to obtain the
theorem with universal constants not depending on the initial data.

This improvement is possible due to the validity of Proposition \ref{newl4}.
We make the definition
\begin{equation}
\label{e:epsfuncdef}
\eta(t) = \sup_{x\in\R^N}\int_{f^{-1}(B_\rho(x))} |A|^4 + |\nabla A|^2\,d\mu\,.
\end{equation}
By covering $B_1(x) \subset \R^N$ with several translated copies of
$B_{\sfrac{1}{2}}$ there is a constant $c_{\eta}$ depending only on $N$ such that
\begin{equation}
\label{e:epscovered}
\eta(t) \le c_{\eta}\sup_{x\in\R^N}\int_{f^{-1}(B_{\sfrac{\rho}{2}}(x))}|A|^4 + |\nabla A|^2\,d\mu\,.
\end{equation}

By short time existence the function $\eta:[0,T)\rightarrow\R$ is continuous.  We now define
\begin{equation}
\label{e:struweparameter}
t_0
=
  \sup\{0\le t\le\min(T,\lambda) : \eta(\tau)\le \delta
                                         \ \text{ for }\ 0\le\tau\le t\},
\end{equation}
where $\lambda$, $\delta$ are parameters to be specified later.

The proof continues in three steps.
\begin{align}
\label{e:7}
t_0 &= \min(T,\lambda),\\
\label{e:8}
t_0 &= \lambda \quad\Longrightarrow\quad\text{Lifespan Theorem},\\
\label{e:9}
T &\ne \infty\hskip+2.42mm \Longrightarrow\quad t_0 \ne T\,.
\end{align}
The three statements \eqref{e:7}, \eqref{e:8}, \eqref{e:9} together imply the Lifespan Theorem.
The argument is as follows: first notice that by \eqref{e:7}
$t_0 = \lambda$ or $t_0 = T$, and if $t_0 = \lambda$ then by \eqref{e:8} we
have the Lifespan Theorem.  Also notice that if $t_0 = \infty$ then $T = \infty$ and the
Lifespan Theorem follows from estimate \eqref{EQltproof} below (used to prove statement
\eqref{e:8}).  Therefore the only remaining case where the Lifespan Theorem may fail to be true is
when $t_0 = T < \infty$.  But this is impossible by statement \eqref{e:9}, so we are finished.

To prove step 1, suppose it is false.
This means that $t_0 < \min(\lambda,T)$, so that on $[0,t_0)$ we have $\eta(t) \le \delta$, and
\begin{equation}
\label{tobecontrad}
\eta(t_0) = \delta
\,.
\end{equation}
Setting $\tilde\gamma$ to be a cutoff function that is identically one on
$B_{\sfrac\rho2}(x)$ and zero outside $B_\rho(x)$, so that $\gamma$ has the
corresponding properties on the preimages of these balls under $f$,
Proposition \ref{newl4} implies
\[
	\int_{f^{-1}(B_\frac\rho2(x))}|A|^4 + |\nabla A|^2\,d\mu
	\le 
	   \int_{f^{-1}(B_\rho(x))}|A|^4 + |\nabla A|^2\,d\mu\bigg|_{t=0}
	 + c_0t\rho^{-4}\varepsilon_1
	 \,,\quad t\in[0,t_0)\,.
\]
A covering argument implies
\[
\eta(t) \le c_{\eta}\sup_{x\in\R^N}\int_{f^{-1}(B_{\sfrac{\rho}{2}}(x))}|A|^4+|\nabla A|^2\,d\mu
\]
so that
\begin{equation}
	\int_{f^{-1}(B_1(x))}|A|^4+|\nabla A|^2\,d\mu
	< 
	c_\eta \varepsilon_1
	 + c_\eta c_0\lambda\rho^{-4}\varepsilon_1
	 \,.
\label{EQltproof}
\end{equation}
We choose $\delta = 3c_\eta \varepsilon_1$, and
$\varepsilon_1$ small enough such that $\delta \le \varepsilon_0$ where
$\varepsilon_0$ is the smaller of those appearing in Proposition \ref{newl4} and
Proposition \ref{p:prop46}.
Then, for $\lambda \le \rho^4/c_0$, the above estimate implies
\[
	\eta(t) < 2c_\eta \varepsilon_1
	\,,
\]
for all $t\in[0,t_0)$.
Therefore (recall that $t_0 < T$) $\lim_{t\rightarrow t_0}\eta(t) \le 2c_\eta\varepsilon_1$.
This is a contradiction with \eqref{tobecontrad}.

This establishes step one \eqref{e:7}.
We have also proved the second step \eqref{e:8}. Observe that if $t_0 =
\lambda$ then by the definition \eqref{e:struweparameter} of $t_0$, 
\[
	T\ge\lambda\,,
\]
which is the lower bound for maximal time claimed by the lifespan theorem.
The estimate \eqref{eqgoalestn4} follows from \eqref{EQltproof}.
That is, we have proved if $t_0 = \lambda$, then the lifespan theorem holds, which is the second step.

We assume
\[
	t_0=T\ne\infty;
\]
since if $T=\infty$ then the lower bound on $T$ holds automatically and again the previous estimates imply
the a-priori control on $\vn{A}^4_{4,f^{-1}(B_1(x))}+\vn{\nabla A}^2_{2,f^{-1}(B_1(x))}$.
Note also that we can safely assume $T < \lambda$, since otherwise
we can apply step two to conclude the Lifespan Theorem.

In this case, Proposition \ref{p:prop46} implies that the flow exists smoothly up to and including time $T$.
The proof of this claim follows exactly as in \cite{kuwert2002gfw}.
In particular, we have uniform control in $C^\infty$ for the flow, allowing us
to reapply short time existence and extend the flow.
This contradicts the maximality of $T$, and finishes the proof.
\end{proof}
%}}}

All steps in the proof rely only on the flow having the form
\[
\BF = \Delta^\perp \vec{H} + P_3^0(A)\,.
\]
Both the surface diffusion flow and the Willmore flow have this form, in
addition to Chen's flow.
The work in this section extends results from
\cite{kuwert2002gfw,mythesis,mySDLTE} to the case where $n=4$ for the flows
considered there.
We state a general version of the lifespan theorem here incorporating this.

\begin{thm}
\label{TMgenerallifespan}
Let $n\in\{2,4\}$.
There exist constants $\varepsilon_1>0$ and $c<\infty$ depending only
on $n$ and $N$ with the following property.
Consider a curvature flow $f:M^n \times[0,T)\rightarrow \R^{N}$ with smooth
initial data satisfying
\[
(\partial_t f)^\perp = -\BF
\]
where $\BF = \Delta^\perp \vec{H} + P_3^0(A)$.

({\bf Case 1: $\mathbf{n=2}$.})
Let $\rho$ be chosen such that
\begin{equation*}
\int_{f^{-1}(B_\rho(x))} |A|^2 d\mu\Big|_{t=0} = \varepsilon(x) \le \varepsilon_1
\qquad
\text{ for all $x\in \R^{N}$}\,.
\end{equation*}
Then the maximal time $T$ of smooth existence satisfies
\begin{equation*}
T \ge \frac1{c}\rho^4\,,
\end{equation*}
and we have the estimate
\begin{equation*}
\int_{f^{-1}(B_\rho(x))} |A|^2 d\mu \le c\varepsilon_1
\qquad\qquad\qquad\hskip-1mm\text{ for all }
t \in \Big[0, \frac1{c}\rho^4\Big].
\end{equation*}

({\bf Case 2: $\mathbf{n=4}$.})
Let $\rho$ be chosen such that
\begin{equation*}
\int_{f^{-1}(B_\rho(x))} |A|^4+|\nabla A|^2 d\mu\Big|_{t=0} = \varepsilon(x) \le \varepsilon_1
\qquad
\text{ for all $x\in \R^{N}$}\,.
\end{equation*}
Then the maximal time $T$ of smooth existence satisfies
\begin{equation*}
T \ge \frac1{c}\rho^4\,,
\end{equation*}
and we have the estimate
\begin{equation*}
\int_{f^{-1}(B_\rho(x))} |A|^4+|\nabla A|^2 d\mu \le c\varepsilon_1
\qquad\qquad\qquad\hskip-1mm\text{ for all }
t \in \Big[0, \frac1{c}\rho^4\Big].
\end{equation*}
\end{thm}

%}}}

\section{Global analysis of the flow}%{{{

Now we move from a local condition on the concentration of curvature for the
initial data, to a global condition on the tracefree second fundamental form.
Unlike the estimates we have already discussed, we are now restricted to $n=2$.
We follow the same strategy as in \cite{MW16}, where asymptotic convergence to
a round point is proved for a Willmore/Helfrich flow.
The key difference here is in showing that the energy is monotone.
This is where the restriction on dimension arises.

\begin{lem}
Let $f:M^2\times[0,T)\rightarrow\R^N$ be Chen's flow.
There exists an absolute constant $\varepsilon_2 > 0$ such that if
\[
\int_M|A^o|^2\,d\mu \le \varepsilon_2
\]
then
\begin{align*}
	\frac{d}{dt}\int_M |A^o|^2\,d\mu
	&\le -\int_M |\Delta H|^2\,d\mu
  - \frac12\int_M |A^o|^2H^4\,d\mu
  - \frac{17}4\int_M |\nabla A^o|^2H^2\,d\mu
  \,.
\end{align*}
\end{lem}
\begin{proof}%{{{
We first compute
\begin{align*}
	\frac{d}{dt}\int_M |A^o|^2\,d\mu
	 &= 
	\frac{d}{dt}\frac12\int_M |\vec{H}|^2\,d\mu
	\\
	&= -\int_M \IP{\Delta^\perp\vec{H} + Q(A^o)\vec{H}}{\BF}\,d\mu
	\\
	&= -\int_M \IP{\Delta^\perp\vec{H} + Q(A^o)\vec{H}}{\Delta^\perp\vec{H} - Q(A)\vec{H}}\,d\mu
	\\
	&= -\int_M |\Delta^\perp\vec{H}|^2\,d\mu
          + \int_M \IP{Q(A^o)\vec{H}}{Q(A)\vec{H}}\,d\mu
	  \\&\quad
	  + \int_M \IP{\Delta^\perp\vec{H}}{Q(A)\vec{H}}\,d\mu
	  - \int_M \IP{Q(A^o)\vec{H}}{\Delta^\perp\vec{H}}\,d\mu
	  \,.
\end{align*}
Note that
\begin{align*}
	\int_M \IP{\Delta^\perp\vec{H}}{Q(A)\vec{H} - Q(A^o)\vec{H}}\,d\mu
	&= \frac12\int_M \IP{\Delta^\perp\vec{H}}{ |\vec{H}|^2\vec{H}}\,d\mu
	\\
&= -\frac12\int_M |\nabla\vec{H}|^2|\vec{H}|^2\,d\mu
	-\int_M \Big|\IP{\nabla\vec{H}}{\vec{H}}\Big|^2\,d\mu
	\\
&= -\frac32\int_M |\nabla\vec{H}|^2|\vec{H}|^2\,d\mu
	\,.
\end{align*}
Therefore we find
\begin{align*}
	\frac{d}{dt}&\int_M |A^o|^2\,d\mu
	\le -\int_M |\Delta^\perp \vec H|^2\,d\mu
             -\frac32\int_M |\nabla\vec{H}|^2|\vec{H}|^2\,d\mu
          + \int_M |A^o|^2|A|^2|\vec H|^2\,d\mu
  \,.
\end{align*}
We now use estimate \cite{MWannalen}[(14)] (see also \cite{kuwert2001wfs}[(17)]), valid analogously in high codimension, which reads
\begin{align*}
	\Big(1-\delta\Big)\int_{M} |\vec H|^2&|\nabla A^o|^2\gamma^4d\mu
	+ \Big(\frac12-2\delta\Big)\int_{M} |\vec H|^4|A^o|^2\gamma^4d\mu
\\
&\le
\Big(\frac12+3\delta\Big)\int_{M} |\vec H|^2|\nabla \vec H|^2\gamma^4d\mu
\\
&\quad
+ c\int_{M}\big(|A^o|^6 + |A^o|^2|\nabla A^o|^2\big)\gamma^4d\mu
  + c_\gamma^4c\int_{[\gamma>0]}|A^o|^2d\mu,
\end{align*}
for $\delta>0$, where $c$ is a constant depending only on $\delta$.

Rarranging this with $\gamma \equiv 1$ yields
\begin{align*}
	- \int_{M} |\vec H|^2|\nabla \vec H|^2d\mu
&\le
- \frac{2-2\delta}{1+6\delta}
\int_{M} |\vec H|^2|\nabla A^o|^2d\mu
     - \frac{1-4\delta}{1+6\delta}
     \int_{M} |\vec H|^4|A^o|^2d\mu
     \\&\quad+ c\int_{M}\big(|A^o|^6 + |A^o|^2|\nabla A^o|^2\big)d\mu
  \,.
\end{align*}
In order to absorb the bad term we need
\[
	\frac32\frac{1-4\delta}{1+6\delta} > \frac12
	\qquad
	\Longleftrightarrow
	\qquad
	1-4\delta > \frac13 + 2\delta
	\qquad
	\Longleftrightarrow
	\qquad
	\frac23 > 6\delta
	\,.
\]
This is satisfied for $\delta < \frac19$, so let's pick $\delta = \frac1{18}$.
This implies
\begin{align*}
	- \frac32\int_{M} |\vec H|^2|\nabla \vec H|^2d\mu
&\le
- \frac{17}{8}
\int_{M} |\vec H|^2|\nabla A^o|^2d\mu
     - \frac{7}{8}
     \int_{M} |\vec H|^4|A^o|^2d\mu
     \\&\quad+ c\int_{M}\big(|A^o|^6 + |A^o|^2|\nabla A^o|^2\big)d\mu
  \,.
\end{align*}
The evolution of $\vn{A^o}_2^2$ can then be estimated by
\begin{align*}
	\frac{d}{dt}\int_M |A^o|^2\,d\mu
	&\le -\int_M |\Delta^\perp \vec H|^2\,d\mu
          + \int_M |A^o|^4|\vec H|^2\,d\mu
  + \frac12\int_M |A^o|^2|\vec H|^4\,d\mu
	  \\&\quad
  - \frac{17}8\int_M |\nabla A^o|^2|\vec H|^2\,d\mu
  - \frac{7}{8}\int_{M} |A^o|^2|\vec H|^4d\mu
  + c\int_{M}\big(|A^o|^6 + |A^o|^2|\nabla A^o|^2\big)d\mu
\\
	&\le -\int_M |\Delta^\perp \vec H|^2\,d\mu
          + \int_M |A^o|^4|\vec H|^2\,d\mu
  - \frac38\int_M |A^o|^2|\vec H|^4\,d\mu
	  \\&\quad
  - \frac{17}8\int_M |\nabla A^o|^2|\vec H|^2\,d\mu
  + c\int_{M}\big(|A^o|^6 + |A^o|^2|\nabla A^o|^2\big)d\mu
  \,.
\end{align*}
Estimating $\int_M |A^o|^4|\vec H|^2\,d\mu \le \frac18\int_{M} |\vec H|^4|A^o|^2d\mu +
c\int_M |A^o|^6\,d\mu$, this becomes
\begin{align*}
	\frac{d}{dt}\int_M |A^o|^2\,d\mu
	&\le -\int_M |\Delta^\perp \vec H|^2\,d\mu
  - \frac14\int_M |A^o|^2|\vec H|^4\,d\mu
  - \frac{17}8\int_M |\nabla A^o|^2|\vec H|^2\,d\mu
	  \\&\quad
  + c\int_{M}\big(|A^o|^6 + |A^o|^2|\nabla A^o|^2\big)d\mu
  \,.
\end{align*}
Now we use the smallness assumption, so that the Sobolev inequalities
\begin{align*}
	c\int_{M}\big(&|A^o|^6 + |A^o|^2|\nabla A^o|^2\big)d\mu
\le c\vn{A^o}_2^2\int_{M}\big(|\nabla_{(2)}A|^2 + |A|^2|\nabla A|^2 + |A|^4|A^o|^2\big)d\mu
\end{align*}
and
\begin{align*}
	\int_{M}\big(|\nabla_{(2)}A|^2 + |A|^2|\nabla A|^2 + |A|^4|A^o|^2\big)d\mu
\le c\int_M |\Delta^\perp\vec{H}|^2\,d\mu
\end{align*}
from \cite{kuwert2001wfs} and \cite{mySDLTE} respectively, become valid.
Combining these together we find that, for $\varepsilon_2$ sufficiently small,
\begin{align*}
	\frac{d}{dt}\int_M |A^o|^2\,d\mu
	&\le -\frac12\int_M |\Delta \vec H|^2\,d\mu
  - \frac14\int_M |A^o|^2\vec H^4\,d\mu
  - \frac{17}8\int_M |\nabla A^o|^2\vec H^2\,d\mu
  \,,
\end{align*}
as required.
\end{proof}%}}}

\begin{rmk}
The integral identity we use, as well as the relationship between $A$, $A^o$ and $\vec{H}$, are only valid for $n=2$.
\end{rmk}

Interior estimates for the flow follow using an argument analogous to
\cite[Theorem 3.11]{MW16}.

\begin{thm}
\label{Tie}
Suppose $f:M^2\times(0,\delta]\rightarrow\R^N$ flows by \eqref{EQwf} and satisfies
\[
\sup_{0<t\le\delta} \int_{f^{-1}(B_{2\rho}(0))} |A|^2 d\mu \le \varepsilon < \varepsilon_0,
\]
where $\delta \le c\rho^4$.  Then for any $k\in\N_0$ and $t\in(0,\delta)$ we have
\begin{align*}
\vn{\nabla_{(k)}A}_{2,f^{-1}(B_\rho(0))} &\le c_k\sqrt{\varepsilon}t^{-\frac{k}{4}}
\\
\vn{\nabla_{(k)}A}_{\infty,f^{-1}(B_\rho(0))} &\le c_k\sqrt{\varepsilon}t^{-\frac{k+1}{4}}
\end{align*}
where $c_k$ is an absolute constant for each $k$.
\end{thm}

We know by the Lifespan Theorem that for any sequence of radii $r_j\searrow 0$
there exists a sequence of times $t_j\nearrow T$ such that
\begin{equation*}
t_j = \inf\Big\{t\ge0: \sup_{x\in\R^N} \int_{f^{-1}(B_{r_j}(x))}|A|^2d\mu
 > \varepsilon_3\Big\} < T,
\end{equation*}
where $\varepsilon_3 = \varepsilon_1/c_1$ and $\varepsilon_1, c_1$ are the constants from the Lifespan Theorem.
Curvature is quantised along $f(\cdot,t_j)$ so that
\[
\int_{f^{-1}(B_{r_j}(x))}|A|^2d\mu\bigg|_{t=t_j} \le \varepsilon_3
\text{  for any }x\in\R^N,
\]
and
\begin{equation}
\int_{f^{-1}\overline{(B_{r_j}(x_j))}}|A|^2d\mu\bigg|_{t=t_j} \ge \varepsilon_3
\text{  for some }x_j\in\R^N.
\label{EQbusomecurv}
\end{equation}

Consider the rescaled immersions
\[
f_j:M^2\times\big[-r_j^{-4}t_j, r_j^{-4}(T-t_j)\big)\rightarrow\R^N,
\qquad
f_j(p,t) = \frac{1}{r_j}\big(f(p,t_j+r_j^4t)-x_j\big).
\]
The Lifespan Theorem implies $r_j^{-4}(T-t_j) \ge c_0$ for any $j$ and also that
\[
\sup_{x\in\R^N}\int_{f_j^{-1}(B_{1}(x))}|A|^2d\mu \le \varepsilon_0
\text{  for }0<t\le c_0.
\]
Interior estimates on parabolic cylinders $B_1(x)\times(t-1,t]$ yields
\[
\vn{\nabla_{(k)}A}_{\infty,f_j} \le c(k)\quad\text{for}\quad-r_j^{-4}t_j+1\le t\le c_0.
\]
The Willmore energy is bounded and so a local area bound may be obtained by
Simon's estimate \cite{simon1993esm}.
Therefore applying Kuwert-Sch\"atzle's compactness theorem \cite[Theorem 4.2]{kuwert2001wfs}
(see also \cite{breuning,CooperCompactness}) to the sequence $f_j =
f_j(\cdot,0):M^2\rightarrow\R^N$ we recover a limit immersion
$\hat{f}_0:\hat{M^2}\rightarrow\R^N$, where $\hat{M^2} \cong M^2$.

We also obtain the diffeomorphisms $\phi_j:\hat{M^2}(j)\rightarrow U_j\subset
M^2$, such that the reparametrisation
\begin{equation*}
  f_j(\phi_j,\cdot):\hat{M^2}(j)\times[0,c_0]\rightarrow\R^N
\end{equation*}
is a Chen flow with initial data
\begin{equation*}
f_j(\phi_j,0) = \hat{f}_0+u_j:\hat{M^2}(j)\rightarrow\R^N.
\end{equation*}
We obtain the locally smooth convergence
\begin{equation}
f_j(\phi_j,\cdot) \rightarrow \hat{f},
\label{C7E5}
\end{equation}
where $\hat{f}:\hat{M^2}\times[0,c_0]\rightarrow\R^N$ is a Chen flow with
initial data $\hat{f}_0$.

\begin{thm}
Let $f:M^2\times[0,T)\rightarrow\R^N$ be a Chen flow satisfying the smallness hypothesis.
	Then the blowup $\tilde{f}$ as constructed above satisfies $\vn{\BQ}_2^2 \equiv 0$, where $(\vn{A^o})_2^2)' \le -2\vn{\BQ}_2^2$.
\label{Tbu}
\end{thm}
\begin{proof}%{{{
The monotonicity calculation implies
\begin{align*}
	2\int_0^{c_0}\int_{\tilde{{M^2}}(j)}&|\BQ(f_j(\phi_j,t))|^2d\mu_{f_j(\phi_j,\cdot)}dt
= 2\int_0^{c_0}\int_{U_j}|\BQ_j|^2d\mu_jdt
\\
&\le
\int_{M^2} |A^o_j(0)|^2d\mu_j
- \int_{M^2} |A^o_j(c_0)|^2d\mu_j
\\
&= 
\int_{M^2} |A^o(t_j)|^2d\mu
- \int_{M^2} |A^o(t_j+r_j^4c_0)|^2d\mu,
\end{align*}
and this converges to zero as $j\rightarrow\infty$.
\end{proof}%}}}

\begin{proof}[Proof of Theorem \ref{TMglobal}.]
Theorem \ref{Tbu} implies
\[
	\int_{M^2}\big(|\Delta^\perp \vec H|^2 + |\nabla A^o|^2|\vec H|^2 + |\vec H|^4|A^o|^2\big)d\mu = 0
\]
and so the blowup is a union of embedded spheres and planes.
Ruling out disconnected components using \cite[Lemma 4.3]{kuwert2001wfs} and
noting that by \eqref{EQbusomecurv} we have $\vn{\tilde{A}}_2^2 > 0$, we
conclude that $\tilde{f}$ is a round sphere.

As the sequence of radii was arbitrary and area is monotone, this shows that
$\mu(f_t)\searrow0$ and that $f_t$ is asymptotic to a round point.
\end{proof}

%}}}

\appendix

\section*{Appendix}

\begin{proof}[Proof of Lemma \ref{MS1lem}.]
	%{{{ 
Statement (i) is Lemma 4.2 in \cite{kuwert2002gfw}.

Let us now prove (ii).
We estimate
\begin{align*}
	\int_M |A|^6\,\gamma^s\,d\mu
	&\le c\bigg(
	      \int_M |\nabla A|\,|A|^\frac72\,\gamma^\frac{3s}{2}\,d\mu
	      + \int_M |A|^\frac{11}2\,\gamma^\frac{3s}{2}\,d\mu
	    + (\cg)\int_M |A|^\frac92\,\gamma^\frac{3s-2}{2}\,d\mu
	     \bigg)^\frac43
	     \\
	&\le 
	     c\bigg(
	     \int_M |\nabla A|^2\,|A|^2\,\gamma^s\,d\mu
	     \bigg)^\frac23
	     \bigg(
	     \int_M |A|^5\,\gamma^{2s}\,d\mu
	     \bigg)^\frac23
	     \\&\qquad
	     + c\vn{A}_{4,[\gamma>0]}^\frac43
	      \int_M |A|^6\,\gamma^s\,d\mu
	     + c(\cg)^\frac43\bigg(
	     \int_M |A|^\frac92\,\gamma^\frac{3s-2}{2}\,d\mu
	     \bigg)^\frac43
	     \\
	&\le
	     c\bigg(
	     \int_M |\nabla A|^2\,|A|^2\,\gamma^s\,d\mu
	     \bigg)^\frac23
	     \bigg(
	     \int_M |A|^5\,\gamma^{2s}\,d\mu
	     \bigg)^\frac23
	     \\&\qquad
	     + c\vn{A}_{4,[\gamma>0]}^\frac43
	      \int_M |A|^6\,\gamma^s\,d\mu
	     + c(\cg)^\frac43\bigg(
	     \int_{[\gamma>0]} |A|^4\,d\mu
	     \bigg)^\frac23
	     \bigg(
	     \int_M |A|^5\,\gamma^{3s-2}\,d\mu
	     \bigg)^\frac23
	     \\
	&\le
	     \theta\int_M |\nabla A|^2\,|A|^2\,\gamma^s\,d\mu
	      + c\bigg(
	     \int_M |A|^5\,\gamma^{2s}\,d\mu
	     \bigg)^2
	     \\&\qquad
	     + c\vn{A}_{4,[\gamma>0]}^\frac43
	      \int_M |A|^6\,\gamma^s\,d\mu
	      + c(\cg)^2\vn{A}_{4,[\gamma>0]}^4
	     \\
	&\le
	     \theta\int_M |\nabla A|^2\,|A|^2\,\gamma^s\,d\mu
	     + c(\vn{A}_{4,[\gamma>0]}^\frac43
	     + \vn{A}_{4,[\gamma>0]}^4)
	      \int_M |A|^6\,\gamma^s\,d\mu
	     \\&\qquad
	      + c(\cg)^2\vn{A}_{4,[\gamma>0]}^4
	     \,.
\end{align*}
Note that in the above we used $3s-2\ge 2s$.
Now for the other term we estimate
\begin{align*}
	\int_M |\nabla A|^2|A|^2\,\gamma^s\,d\mu
	&= - 2 \int_M (\nabla_iA_{jk})(A^{jk})(\nabla^iA_{lm})(A^{lm})\,\gamma^s\,d\mu
	  -  \int_M \Delta A*A*A*A\,\gamma^s\,d\mu
	  \\&\quad
	  -  \int_M \nabla A*A*A*A*D\gamma\,\gamma^{s-1}\,d\mu
	  \\
	&\le 
	  \theta\int_M |\nabla_{(2)}A|^2\,\gamma^s\,d\mu
	   + c\int_M |A|^6\,\gamma^s\,d\mu
	   + (\cg)\int_M |\nabla A|\,|A|^3\,\gamma^{s-1}\,d\mu
	  \\
	&\le 
	  \theta\int_M |\nabla_{(2)}A|^2\,\gamma^s\,d\mu
	   + c\int_M |A|^6\,\gamma^s\,d\mu
	   \frac12
	   \int_M |\nabla A|^2|A|^2\,\gamma^s\,d\mu
	   + (\cg)^2\vn{A}_{4,[\gamma>0]}^4
\end{align*}
where we used $s\ge2$.
Absorbing yields
\begin{align*}
	\int_M |\nabla A|^2|A|^2\,\gamma^s\,d\mu
	&\le 
	  \theta\int_M |\nabla_{(2)}A|^2\,\gamma^s\,d\mu
	   + c\int_M |A|^6\,\gamma^s\,d\mu
	   + (\cg)^2\vn{A}_{4,[\gamma>0]}^4
	   \,.
\end{align*}
We add these estimates together to obtain
\begin{align*}
	\int_M \big(|\nabla A|^2|A|^2 + |A|^6\big)\,\gamma^s\,d\mu
	&\le
	\theta\int_M |\nabla_{(2)}A|^2\,\gamma^s\,d\mu
	   + c\theta\int_M |\nabla A|^2\,|A|^2\,\gamma^s\,d\mu
	   \\&\quad
	     + c(\vn{A}_{4,[\gamma>0]}^\frac43
	     + \vn{A}_{4,[\gamma>0]}^4)
	      \int_M |A|^6\,\gamma^s\,d\mu
	   + (\cg)^2\vn{A}_{4,[\gamma>0]}^4\,.
\end{align*}
Absorbing again for $\theta$ sufficiently small yields
\begin{align*}
	\int_M \big(|\nabla A|^2|A|^2 + |A|^6\big)\,\gamma^s\,d\mu
	\le
	&\ \theta\int_M |\nabla_{(2)}A|^2\,\gamma^s\,d\mu
	     + c(\vn{A}_{4,[\gamma>0]}^\frac43
	     + \vn{A}_{4,[\gamma>0]}^4)
	      \int_M |A|^6\,\gamma^s\,d\mu
	     \\&
	   + (\cg)^2\vn{A}_{4,[\gamma>0]}^4\,.
\end{align*}
as required.

For the second estimate in (ii) we begin with
\begin{align*}
	\int_M |\nabla A|^2|A|^3\,\gamma^s
	&\le c\int_M |\nabla_{(2)}A|\,|A|^4\,\gamma^s\,d\mu
	   + c(\cg)\int_M |\nabla A|\,|A|^4\gamma^{s-1}\,d\mu
	   - \frac34\int_M |\nabla |A|^2|^2\,|A|\,\gamma^s\,d\mu
	\\
	&\le \frac12\int_M |\nabla A|^2|A|^3\,\gamma^s
	   + \theta\int_M |\nabla_{(2)}A|^2\,|A|\,\gamma^s\,d\mu
	   + c\int_M |A|^7\,\gamma^s\,d\mu
	   + c(\cg)^2\int_M |A|^5\gamma^{s-2}\,d\mu
	\\
	&\le \frac12\int_M |\nabla A|^2|A|^3\,\gamma^s
	   + \theta\int_M |\nabla_{(2)}A|^2\,|A|\,\gamma^s\,d\mu
	   + c\int_M |A|^7\,\gamma^s\,d\mu
	   + c(\cg)^4\vn{A}_{3,[\gamma>0]}^3
\end{align*}
which upon absorbing yields
\begin{equation}
	\int_M |\nabla A|^2|A|^3\,\gamma^s
	 \le
	     \theta\int_M |\nabla_{(2)}A|^2\,|A|\,\gamma^s\,d\mu
	   + c\int_M |A|^7\,\gamma^s\,d\mu
	   + c(\cg)^4\vn{A}_{3,[\gamma>0]}^3
	   \,.
\label{eqprev0}
\end{equation}
Note that in the estimate above we used $2s-4\ge s$.
Now Michael-Simon yields
\begin{align}
	\int_M |A|^7\,\gamma^s\,d\mu
	\notag&\le
	   c\bigg(
	   \int_M |\nabla A|\,|A|^\frac{11}{3}\,\gamma^\frac{2s}{3}\,d\mu
	    \bigg)^\frac32
	 + c\bigg(
	   \int_M |A|^\frac{17}{3}\,\gamma^\frac{2s}{3}\,d\mu
	    \bigg)^\frac32
	    + c(\cg)^\frac32\bigg(
	   \int_M |A|^\frac{14}{3}\,\gamma^\frac{2s-3}{3}\,d\mu
	    \bigg)^\frac32
	\notag    \\
	\notag&\le
	    c\vn{A}_{3,[\gamma>0]}^\frac32
	   \int_M |\nabla A|^\frac32\,|A|^4,\gamma^s\,d\mu
	  + c\vn{A}_{3,[\gamma>0]}^\frac32
	   \int_M |A|^7\,\gamma^s\,d\mu
	\notag   \\&\qquad
	  + c(\cg)^\frac32
	   \int_M |A|^\frac{7}{2}\,\gamma^\frac{s-3}{2}\,d\mu
            \bigg(
	   \int_M |A|^7\,\gamma^s\,d\mu
	    \bigg)^\frac12
	   \notag \\
	\notag&\le
	    c\vn{A}_{3,[\gamma>0]}^2
	   \int_M |\nabla A|^2\,|A|^3,\gamma^s\,d\mu
	   + \big(c\vn{A}_{3,[\gamma>0]}^\frac32 + \theta)
	   \int_M |A|^7\,\gamma^s\,d\mu
	\notag   \\&\qquad
	  + c(\cg)^3
            \bigg(
	   \int_M |A|^\frac{7}{2}\,\gamma^\frac{s-3}{2}\,d\mu
	    \bigg)^2
	    \,.
\label{eqprev}
\end{align}
To deal with the last term we estimate
\begin{align*}
	  c(\cg)^3
            \bigg(
	   \int_M |A|^\frac{7}{2}\,\gamma^\frac{s-3}{2}\,d\mu
	    \bigg)^2
	  &\le 
	    c(\cg)^3\vn{A}_{3,[\gamma>0]}^3
	   \int_M |A|^4\,\gamma^{s-3}\,d\mu
	  \\
	  &\le 
	    c(\cg)^3\vn{A}_{3,[\gamma>0]}^\frac92
            \bigg(
	   \int_M |A|^5\,\gamma^{2s-6}\,d\mu
	    \bigg)^\frac12
	  \\
	  &\le 
	  c(\cg)^3\vn{A}_{3,[\gamma>0]}^\frac{21}{4}
            \bigg(
	   \int_M |A|^7\,\gamma^{4s-12}\,d\mu
	    \bigg)^\frac14
	  \\
	  &\le 
	   \theta\int_M |A|^7\,\gamma^s\,d\mu
	  + c(\cg)^4\vn{A}_{3,[\gamma>0]}^7
\end{align*}
where we used $4s-12\ge s$.
This allows us to improve estimate \eqref{eqprev} to
\begin{align*}
	\int_M |A|^7\,\gamma^s\,d\mu
	\notag&\le
	    c\vn{A}_{3,[\gamma>0]}^\frac32
	   \int_M |\nabla A|^\frac32\,|A|^4,\gamma^s\,d\mu
	   + \big(c\vn{A}_{3,[\gamma>0]}^\frac32 + \theta)
	     \int_M |A|^7\,\gamma^s\,d\mu
	   \\&\qquad
	  + c(\cg)^4\vn{A}_{3,[\gamma>0]}^7
	    \,.
\end{align*}
Combining the above with estimate \eqref{eqprev0}, we have
\begin{align*}
	\int_M \big(|\nabla A|^2|A|^3 + |A|^7\big)\,\gamma^s\,d\mu
	&\le
	 \big(c\vn{A}_{3,[\gamma>0]}^2 + \theta\big)\int_M \big(|\nabla_{(2)}A|^2|A| + |\nabla A|^2|A|^3 + |A|^7\big)\,\gamma^s\,d\mu
	     \\&\qquad
	   + (\cg)^4\vn{A}_{3,[\gamma>0]}^3\,,
\end{align*}
as required.

For (iii) we begin by noting
\begin{align*}
-2\int_M &\IP{A_{ij}}{\nabla^pA^{ij}}\IP{\nabla^q A_{kl}}{\nabla_{pq}A^{kl}}\,\gamma^s\,d\mu
\\
&= 
 \int_M \IP{A_{ij}}{\Delta^\perp A^{ij}}|\nabla A|^2\,\gamma^s\,d\mu
         + \int_M |\nabla A|^4\,\gamma^s\,d\mu
\\
  &\qquad+ s\int_M 
          \IP{A_{ij}}{\nabla^pA^{ij}}\IP{A_{kl}}{\nabla_{pq}A^{kl}}\nabla^q\gamma
                            \,\gamma^{s-1}\,d\mu
\end{align*}
and
\begin{align*}
-2\int_M &\IP{A_{ij}}{\nabla^pA^{ij}}\IP{\nabla^q A_{kl}}{\nabla_{pq}A^{kl}}\,\gamma^s\,d\mu
\\
&= \int_M |A|^2|\nabla_{(2)}A|^2\,\gamma^s\,d\mu
 + \int_M |A|^2\IP{\nabla^qA_{kl}}{\Delta^\perp\nabla_qA^{kl}}\,\gamma^s\,d\mu
\\
  &\qquad+ s\int_M 
          |A|^2\IP{\nabla^qA_{kl}}{\nabla_{pq}A^{kl}}\nabla^p\gamma
                            \,\gamma^{s-1}\,d\mu
\end{align*}
so that
\begin{align}
\int_M |\nabla A|^4\,\gamma^s\,d\mu
&=
-2\int_M \IP{A_{ij}}{\nabla^pA^{ij}}\IP{\nabla^q A_{kl}}{\nabla_{pq}A^{kl}}\,\gamma^s\,d\mu
\label{EQibpuse}\\&\qquad
 - \int_M \IP{A_{ij}}{\Delta^\perp A^{ij}}|\nabla A|^2\,\gamma^s\,d\mu
\notag\\
  &\qquad- s\int_M 
          \IP{A_{ij}}{\nabla^pA^{ij}}\IP{A_{kl}}{\nabla_{pq}A^{kl}}\nabla^q\gamma
                            \,\gamma^{s-1}\,d\mu
\notag\\&=
   \int_M |A|^2|\nabla_{(2)}A|^2\,\gamma^s\,d\mu
 + \int_M |A|^2\IP{\nabla^qA_{kl}}{\Delta^\perp\nabla_qA^{kl}}\,\gamma^s\,d\mu
\notag\\&\qquad
 - \int_M \IP{A_{ij}}{\Delta^\perp A^{ij}}|\nabla A|^2\,\gamma^s\,d\mu
\notag\\
  &\qquad+ s\int_M 
          |A|^2\IP{\nabla^qA_{kl}}{\nabla_{pq}A^{kl}}\nabla^p\gamma
                            \,\gamma^{s-1}\,d\mu
\notag\\
  &\qquad- s\int_M 
          \IP{A_{ij}}{\nabla^pA^{ij}}\IP{A_{kl}}{\nabla_{pq}A^{kl}}\nabla^q\gamma
                            \,\gamma^{s-1}\,d\mu
\,.
\end{align}
We estimate
\begin{align*}
 \int_M |A|^2\IP{\nabla^qA_{kl}}{\Delta^\perp\nabla_qA^{kl}}\,\gamma^s\,d\mu
&\le  \theta \int_M |\nabla_{(3)}A|^2\,\gamma^s\,d\mu
   + \frac14 \int_M |\nabla A|^4\,\gamma^s\,d\mu
\\&\qquad
   + c_\theta \int_M |A|^8\,\gamma^s\,d\mu
\,;
\end{align*}
and, recalling $s\ge4$,
\begin{align*}
  s\int_M 
          &|A|^2\IP{\nabla^qA_{kl}}{\nabla_{pq}A^{kl}}\nabla^p\gamma
                            \,\gamma^{s-1}\,d\mu
\\
  &- s\int_M 
          \IP{A_{ij}}{\nabla^pA^{ij}}\IP{A_{kl}}{\nabla_{pq}A^{kl}}\nabla^q\gamma
                            \,\gamma^{s-1}\,d\mu
\\&\le 
   \int_M |A|^2|\nabla_{(2)}A|^2\,\gamma^s\,d\mu
   + \frac14 \int_M |\nabla A|^4\,\gamma^s\,d\mu
   + cs^2(c_\gamma)^4\int_M |A|^4\,\gamma^{s-4}\,d\mu
\,.
\end{align*}
Combining with the above we find
\begin{align}
\int_M |\nabla A|^4\,\gamma^s\,d\mu
&\le
    \theta \int_M |\nabla_{(3)}A|^2\,\gamma^s\,d\mu
   + c\int_M |A|^2|\nabla_{(2)}A|^2\,\gamma^s\,d\mu
\notag\\
  &\qquad
   + cs^2(c_\gamma)^4\vn{A}_{4,[\gamma>0]}^4
\,.
\label{EQnabla4}
\end{align}
Now let us estimate the second term on the right.
By the Michael-Simon Sobolev inequality we find
\begin{align}
\int_M |A|^2|\nabla_{(2)}A|^2\,\gamma^s\,d\mu
&\le c\bigg(
 	\int_M |A|^\frac12\,|\nabla A|\,|\nabla_{(2)}A|^\frac32\,\gamma^{\frac{3s}{4}}\,d\mu
      \bigg)^{\frac43}
\notag      \\&\qquad
   + c\bigg(
 	\int_M |A|^\frac32\,|\nabla_{(2)}A|^\frac12\,|\nabla_{(3)}A|\,\gamma^{\frac{3s}{4}}\,d\mu
      \bigg)^{\frac43}
\notag      \\&\qquad
   + c\bigg(
   \frac{3s}{4}(c_\gamma)\int_M |A|^\frac32\,|\nabla_{(2)}A|^\frac32\,\gamma^{\frac{3s}{4}-1}\,d\mu
      \bigg)^{\frac43}
\notag      \\&\qquad
   + c\bigg(
                         \int_M |A|^\frac52\,|\nabla_{(2)}A|^\frac32\,\gamma^{\frac{3s}{4}}\,d\mu
      \bigg)^{\frac43}
\notag\\
&\le c\bigg(
 	\int_M |A|^\frac12\,|\nabla A|\,|\nabla_{(2)}A|^\frac32\,\gamma^{\frac{3s}{4}}\,d\mu
      \bigg)^{\frac43}
     \notag \\&\qquad
   + c\bigg(
 	\int_M |A|^\frac32\,|\nabla_{(2)}A|^\frac12\,|\nabla_{(3)}A|\,\gamma^{\frac{3s}{4}}\,d\mu
      \bigg)^{\frac43}
     \notag \\&\qquad
   + c\bigg(
   s(c_\gamma)\int_M |A|^\frac32\,|\nabla_{(2)}A|^\frac32\,\gamma^{\frac{3s}{4}-1}\,d\mu
      \bigg)^{\frac43}
     \notag \\&\qquad
   + c\bigg(
                         \int_{[\gamma>0]} |A|^4\,d\mu
      \bigg)^{\frac13}
                         \int_M |A|^2\,|\nabla_{(2)}A|^2\,\gamma^{s}\,d\mu
      \,.
\label{EQpreest}
\end{align}
We work on the first term by estimating:
\begin{align}
     c\bigg(
     \int_M &|A|^\frac12\,|\nabla A|\,|\nabla_{(2)}A|^\frac32\,\gamma^{\frac{3s}{4}}\,d\mu
     \bigg)^{\frac43}
\notag     \\
     &\le
     \bigg(\int_{[\gamma>0]} |A|^4\,d\mu\bigg)^\frac16
     \bigg(\int_M |\nabla A|^{\frac87}|\nabla_{(2)}A|^{\frac{12}7}\,\gamma^{\frac{6s}{7}}\,d\mu\bigg)^\frac76
\notag     \\
     &\le 
     \bigg(\int_{[\gamma>0]} |A|^4\,d\mu\bigg)^\frac16
     \bigg(\int_M |\nabla A|^4\,\gamma^s\,d\mu\bigg)^\frac13
     \bigg(\int_M |\nabla_{(2)}A|^{\frac{12}5}\,\gamma^{\frac{4s}{5}}\,d\mu\bigg)^\frac56
\notag     \\
     &\le 
     \frac13\vn{A}^\frac23_{4,[\gamma>0]}
     \int_M |\nabla A|^4\,\gamma^s\,d\mu
     + \frac23\vn{A}^\frac23_{4,[\gamma>0]}
       \bigg(\int_M |\nabla_{(2)}A|^{\frac{12}5}\,\gamma^{\frac{4s}{5}}\,d\mu\bigg)^\frac54
\,.
\label{EQfirst}
\end{align}
Estimating
\begin{align*}
       \int_M |\nabla_{(2)}A|^{\frac{12}5}\,\gamma^{\frac{4s}{5}}\,d\mu
&\le
       c\int_M |\nabla A|\,|\nabla_{(2)}A|^{\frac{2}5}\,|\nabla_{(3)}A|\,\gamma^{\frac{4s}{5}}\,d\mu
\\&\quad
     + cs(c_\gamma)\int_M |\nabla A|\,|\nabla_{(2)}A|^{\frac{7}5}\,\gamma^{\frac{4s}{5}-1}\,d\mu
\\
&\le
       c\bigg(\int_M |\nabla_{(2)}A|^{\frac{12}5}\,\gamma^{\frac{4s}{5}}\,d\mu\bigg)^\frac16
       \bigg[
        \bigg(
        \int_M |\nabla A|^\frac65\,|\nabla_{(3)}A|^{\frac{6}5}\,\gamma^{\frac{4s}{5}}\,d\mu
        \bigg)^\frac56
\\&\qquad
       + s(c_\gamma)
        \bigg(
        \int_M |\nabla A|^\frac65\,|\nabla_{(2)}A|^{\frac{6}5}\,\gamma^{\frac{4s-6}{5}}\,d\mu
        \bigg)^\frac56
       \bigg]
\end{align*}
and absorbing yields
\begin{align*}
       \int_M |\nabla_{(2)}A|^{\frac{12}5}\,\gamma^{\frac{4s}{5}}\,d\mu
&\le
        c\int_M |\nabla A|^\frac65\,|\nabla_{(3)}A|^{\frac{6}5}\,\gamma^{\frac{4s}{5}}\,d\mu
\\&\qquad
       + c(s(c_\gamma))^\frac65
        \int_M |\nabla A|^\frac65\,|\nabla_{(2)}A|^{\frac{6}5}\,\gamma^{\frac{4s-6}{5}}\,d\mu
\\
&\le
        c\int_M |\nabla A|^\frac65\,|\nabla_{(3)}A|^{\frac{6}5}\,\gamma^{\frac{4s}{5}}\,d\mu
\\&\qquad
       \frac12\int_M |\nabla_{(2)}A|^{\frac{12}5}\,\gamma^{\frac{4s}{5}}\,d\mu
       + c(s(c_\gamma))^\frac{12}{5}
        \int_M |\nabla A|^\frac{12}5\,\gamma^{\frac{4s-12}{5}}\,d\mu
\,.
\end{align*}
Absorbing another time gives
\begin{align}
       \int_M |\nabla_{(2)}A|^{\frac{12}5}\,\gamma^{\frac{4s}{5}}\,d\mu
&\le
        c\int_M |\nabla A|^\frac65\,|\nabla_{(3)}A|^{\frac{6}5}\,\gamma^{\frac{4s}{5}}\,d\mu
\notag
\\&\qquad
       + c(s(c_\gamma))^\frac{12}{5}
        \int_M |\nabla A|^\frac{12}5\,\gamma^{\frac{4s-12}{5}}\,d\mu
\label{EQplace2}
\,.
\end{align}
To close this estimate we must use the same technique again on the last term with $c_\gamma$.
The first step is
\begin{align}
        \int_M &|\nabla A|^\frac{12}5\,\gamma^{\frac{4s-12}{5}}\,d\mu
\notag
\\
&\le
        c\int_M |A|\,|\nabla A|^\frac{2}5\,|\nabla_{(2)}A|\,\gamma^{\frac{4s-12}{5}}\,d\mu
      + cs(c_\gamma)\int_M |A|\,|\nabla A|^\frac{7}5\,\gamma^{\frac{4s-17}{5}}\,d\mu
\notag
\\
&\le
        \bigg(
          \int_M |\nabla A|^\frac{12}5\,\gamma^{\frac{4s-12}{5}}\,d\mu
        \bigg)^\frac16
      \bigg[
        \bigg(
          \int_M |A|^\frac65\,|\nabla_{(2)}A|^\frac65\,\gamma^{\frac{4s-12}{5}}\,d\mu
        \bigg)^\frac65
\notag
\\&\qquad
      + cs(c_\gamma)
        \bigg(
          \int_M |A|^\frac65\,|\nabla A|^\frac65\,\gamma^{\frac{4s-18}{5}}\,d\mu
        \bigg)^\frac65
      \bigg]
\,.
\notag
\intertext{Absorbing gives}
        \int_M &|\nabla A|^\frac{12}5\,\gamma^{\frac{4s-12}{5}}\,d\mu
\notag
        \\
&\le 
          \int_M |A|^\frac65\,|\nabla_{(2)}A|^\frac65\,\gamma^{\frac{4s-12}{5}}\,d\mu
      + c(s(c_\gamma))^\frac65
          \int_M |A|^\frac65\,|\nabla A|^\frac65\,\gamma^{\frac{4s-18}{5}}\,d\mu
\,.
\label{EQplace}
\end{align}
Now this whole term that we are estimating is raised to the power $\frac54$ and
has a coefficient involving $c_\gamma$, which scales.
Incorporating this, we find
\[
\bigg(
 (c_\gamma)^\frac{18}{5}\int_M |A|^\frac65\,|\nabla A|^\frac65\,\gamma^{\frac{4s-18}{5}}\,d\mu
\bigg)^\frac54
\le 
 (c_\gamma)^\frac{9}{2}
 \bigg(
 \int_{[\gamma>0]} |A|^\frac{12}{7}\,d\mu
 \bigg)^\frac{7}{8}
 \bigg(
 \int_M |\nabla A|^4\,\gamma^{\frac{8s-36}{3}}\,d\mu
 \bigg)^\frac{3}{8}
\,.
\]
Since $s\ge\frac{36}{5}$ we have $\frac{8s-36}{3}\ge s$, this term is estimated by
\begin{align}
\bigg(
  (c_\gamma)^\frac{18}{5}\int_M &|A|^\frac65\,|\nabla A|^\frac65\,\gamma^{\frac{4s-18}{5}}\,d\mu
\bigg)^\frac54
\notag
\\&\le
 (c_\gamma)^\frac{9}{2}
 \mu_\gamma(f)^\frac{1}{2}
 \bigg(
 \int_{[\gamma>0]} |A|^4\,d\mu
 \bigg)^\frac{3}{8}
 \bigg(
 \int_M |\nabla A|^4\,\gamma^{s}\,d\mu
 \bigg)^\frac{3}{8}
\notag
\\&\le
  \int_M |\nabla A|^4\,\gamma^s\,d\mu
 + 
 c(c_\gamma)^\frac{36}{5}\mu_\gamma(f)^\frac45\bigg(
 \int_{[\gamma>0]} |A|^4\,d\mu
 \bigg)^\frac{3}{5}
\,.
\label{EQ1}
\end{align}
In the above we used $\mu_\gamma := \mu|_{[\gamma>0]}$.
Now let us move on to the first term in \eqref{EQplace}.
We estimate it by
\begin{align*}
c(s(c_\gamma))^\frac{12}{5}\int_M &|A|^\frac65\,|\nabla_{(2)}A|^\frac65\,\gamma^{\frac{4s-12}{5}}\,d\mu
\\
&\le 
       \frac12\int_M |\nabla_{(2)}A|^{\frac{12}5}\,\gamma^{\frac{4s}{5}}\,d\mu
 + c(s(c_\gamma))^\frac{24}{5}\int_M |A|^\frac{12}{5}\,\gamma^{\frac{4s-24}{5}}\,d\mu
\\
&\le 
       \frac12\int_M |\nabla_{(2)}A|^{\frac{12}5}\,\gamma^{\frac{4s}{5}}\,d\mu
 + c(s(c_\gamma))^\frac{24}{5}\mu_\gamma(f)^\frac25\vn{A}_{4,[\gamma>0]}^\frac{12}{5}
\,.
\end{align*}
Combining the above with \eqref{EQ1}, \eqref{EQplace} and absorbing in \eqref{EQplace2} gives
\begin{align}
       \int_M |\nabla_{(2)}A|^{\frac{12}5}\,\gamma^{\frac{4s}{5}}\,d\mu
&\le
        c\int_M |\nabla A|^\frac65\,|\nabla_{(3)}A|^{\frac{6}5}\,\gamma^{\frac{4s}{5}}\,d\mu
 + \bigg(\int_M |\nabla A|^4\,\gamma^s\,d\mu\bigg)^\frac45
\notag
\\&\hskip-5mm
 + c(c_\gamma)^\frac{16}{5}\Big([(c_\gamma)^4\mu_\gamma(f)]^\frac{16}{25}
   \vn{A}_{4,[\gamma>0]}^\frac{48}{25}
 + [(c_\gamma)^4\mu_\gamma(f)]^\frac{2}{5}
   \vn{A}_{4,[\gamma>0]}^\frac{12}{5}\Big)
\,.
\label{EQterm}
\end{align}
Note that $[(c_\gamma)^4\mu_\gamma(f)]$ is scale invariant, and that in
\eqref{EQterm} the constant $c$ depends on $s$, $n$ and $N$.
Incorporating the eventual $\frac54$ power, we estimate the first term on the
right in \eqref{EQterm} by
\begin{align}
\notag
c\bigg(\int_M |\nabla A|^\frac65\,|\nabla_{(3)}A|^{\frac{6}5}\,\gamma^{\frac{4s}{5}}\,d\mu\bigg)^\frac54
&\le
c\bigg(\int_M |\nabla A|^3\,\gamma^{\frac{s}{2}}\,d\mu\bigg)^\frac12
 \bigg(\int_M |\nabla_{(3)}A|^2\,\gamma^{s}\,d\mu\bigg)^\frac34
\\&\hskip-2cm\le
c\bigg(\int_M |\nabla A|^3\,\gamma^{\frac{s}{2}}\,d\mu\bigg)^2
 + \frac12\int_M |\nabla_{(3)}A|^2\,\gamma^{s}\,d\mu
\,.
\label{EQplace3}
\end{align}
Using a similar strategy as before, we estimate
\begin{align*}
c\bigg(\int_M |\nabla A|^3\,\gamma^{\frac{s}{2}}\,d\mu\bigg)^2
&\le c\bigg(\int_M |A|\,|\nabla A|\,|\nabla_{(2)}A|\,\gamma^\frac{s}{2}\,d\mu\bigg)^2
\\&\quad
   + cs(c_\gamma)^2\bigg(\int_M |A|\,|\nabla A|^2\,\gamma^\frac{s-2}{2}\,d\mu\bigg)^2
\\
&\le c\bigg(\int_M |A|^\frac{12}{7}\,|\nabla A|^\frac{12}{7}\,\gamma^\frac{2s}{7}\,d\mu\bigg)^\frac{7}{6}
\bigg(\int_M |\nabla_{(2)}A|^\frac{12}{5}\,\gamma^\frac{4s}{5}\,d\mu\bigg)^\frac56
\\&\quad
   + cs(c_\gamma)^2\bigg(\int_M |\nabla A|^3\,\gamma^\frac{s}{2}\,d\mu\bigg)^\frac43
                   \bigg(\int_M |A|^3\,\gamma^\frac{s-6}{2}\,d\mu\bigg)^\frac23
\\
&\le
\frac12\bigg(\int_M |\nabla_{(2)} A|^{\frac{12}{5}} \gamma^{\frac{4s}{5}}\,d\mu\bigg)^\frac54
  + \frac12\bigg(\int_M |\nabla A|^3\,\gamma^\frac{s}{2}\,d\mu\bigg)^2
\\&\quad
  + c\bigg(\int_M |A|^\frac{12}7\,|\nabla A|^\frac{12}7\,\gamma^\frac{2s}{7}\,d\mu\bigg)^\frac72
   + cs^3(c_\gamma)^6
                   \bigg(\int_M |A|^3\,\gamma^\frac{s-6}{2}\,d\mu\bigg)^2
\\
&\le
\frac12\bigg(\int_M |\nabla_{(2)} A|^{\frac{12}{5}} \gamma^{\frac{4s}{5}}\,d\mu\bigg)^\frac54
  + \frac12\bigg(\int_M |\nabla A|^3\,\gamma^\frac{s}{2}\,d\mu\bigg)^2
\\&\quad
  + c\bigg(\int_{[\gamma>0]} |A|^4\,d\mu\bigg)^\frac32
     \bigg(\int_M |\nabla A|^3\,\gamma^\frac{s}{2}\,d\mu\bigg)^2
\\&\quad
   + cs^3(c_\gamma)^6
                   \bigg(\int_M |A|^3\,\gamma^\frac{s-6}{2}\,d\mu\bigg)^2
\end{align*}
and absorb to find
\begin{align*}
\bigg(\int_M |\nabla A|^3\,\gamma^{\frac{s}{2}}\,d\mu\bigg)^2
&\le
\frac12\bigg(\int_M |\nabla_{(2)} A|^{\frac{12}{5}} \gamma^{\frac{4s}{5}}\,d\mu\bigg)^\frac54
  + c
   \vn{A}_{4,[\gamma>0]}^6
     \bigg(\int_M |\nabla A|^3\,\gamma^\frac{s}{2}\,d\mu\bigg)^2
\\&\quad
  + cs^3(c_\gamma)^6\mu_\gamma(f)^\frac12
   \vn{A}_{4,[\gamma>0]}^6
   \,.
\end{align*}
We combine this estimate with \eqref{EQplace3}, \eqref{EQterm} and absorb to obtain
\begin{align}
   \bigg(&\int_M |\nabla_{(2)}A|^{\frac{12}5}\,\gamma^{\frac{4s}{5}}\,d\mu\bigg)^\frac54
 + \bigg(\int_M |\nabla A|^3\,\gamma^{\frac{s}{2}}\,d\mu\bigg)^2
\notag\\
&\le
   \frac12\int_M |\nabla_{(3)}A|^2\,\gamma^{s}\,d\mu
 + c\int_M |\nabla A|^4\,\gamma^s\,d\mu
 + c\vn{A}_{4,[\gamma>0]}^6
    \bigg(\int_M |\nabla A|^3\,\gamma^\frac{s}{2}\,d\mu\bigg)^2
\notag\\&\quad
 + c(c_\gamma)^4[(c_\gamma)^4\mu_\gamma(f)]^\frac{1}{2}\Big(
   \vn{A}_{4,[\gamma>0]}^6
 + \vn{A}_{4,[\gamma>0]}^3
 + [(c_\gamma)^4\mu_\gamma(f)]^\frac{3}{10}
   \vn{A}_{4,[\gamma>0]}^\frac{12}{5}\Big)
   \,.
\label{EQinter1}
\end{align}
The estimate \eqref{EQinter1} is now combined with \eqref{EQfirst}:
In the above $c$ depends only on $n$, $N$, and $s$.

We combine \eqref{EQplace3} above with the estimate and use Young's inequality
to obtain
\begin{align*}
     \bigg(
     \int_M &|A|^\frac12\,|\nabla A|\,|\nabla_{(2)}A|^\frac32\,\gamma^{\frac{3s}{4}}\,d\mu
     \bigg)^{\frac43}
 + \bigg(\int_M |\nabla_{(2)}A|^{\frac{12}5}\,\gamma^{\frac{4s}{5}}\,d\mu\bigg)^\frac54
 + \bigg(\int_M |\nabla A|^3\,\gamma^{\frac{s}{2}}\,d\mu\bigg)^2
\\
     &\le 
     c\vn{A}^\frac23_{4,[\gamma>0]}
     \int_M \big(|\nabla_{(3)}A|^2 + |\nabla A|^4\big)\,\gamma^s\,d\mu
 + c\vn{A}_{4,[\gamma>0]}^\frac{20}{3}
    \bigg(\int_M |\nabla A|^3\,\gamma^\frac{s}{2}\,d\mu\bigg)^2
\notag\\&\quad
 + c(c_\gamma)^4[(c_\gamma)^4\mu_\gamma(f)]^\frac{1}{2}\Big(
   \vn{A}_{4,[\gamma>0]}^\frac{20}{3}
 + \vn{A}_{4,[\gamma>0]}^\frac{11}{3}
 + [(c_\gamma)^4\mu_\gamma(f)]^\frac{3}{10}
   \vn{A}_{4,[\gamma>0]}^\frac{46}{15}\Big)
   \,.
\end{align*}
This estimates the first term in \eqref{EQpreest}.
Now let us work on the second:
\begin{align*}
\bigg(
 	\int_M &|A|^\frac32\,|\nabla_{(2)}A|^\frac12\,|\nabla_{(3)}A|\,\gamma^{\frac{3s}{4}}\,d\mu
\bigg)^{\frac43}
\le
 c\bigg(
 	\int_M |A|^3\,|\nabla_{(2)}A|\,\gamma^{\frac{s}{2}}\,d\mu
  \bigg)^\frac23
  \bigg(
 	\int_M |\nabla_{(3)}A|^2\,\gamma^{s}\,d\mu
  \bigg)^\frac23
\\&\le
 c\bigg(
 	\int_{[\gamma>0]} |A|^4\,d\mu
  \bigg)^\frac13
  \bigg(
 	\int_M |A|^2\,|\nabla_{(2)}A|^2\,\gamma^{s}\,d\mu
  \bigg)^\frac13
  \bigg(
 	\int_M |\nabla_{(3)}A|^2\,\gamma^{s}\,d\mu
  \bigg)^\frac23
\\&\le
c\vn{A}_{4,[\gamma>0]}^\frac43\int_M \big(|\nabla_{(3)}A|^2 +
|A|^2\,|\nabla_{(2)}A|^2\big)\,\gamma^{s}\,d\mu
\,.
\end{align*}
Combining the above two estimates we find
\begin{align}
     \bigg(
     \int_M &|A|^\frac12\,|\nabla A|\,|\nabla_{(2)}A|^\frac32\,\gamma^{\frac{3s}{4}}\,d\mu
     \bigg)^{\frac43}
+ \bigg(
  \int_M |A|^\frac32\,|\nabla_{(2)}A|^\frac12\,|\nabla_{(3)}A|\,\gamma^{\frac{3s}{4}}\,d\mu
  \bigg)^{\frac43}
\notag\\&\hskip-5mm
 + \bigg(\int_M |\nabla_{(2)}A|^{\frac{12}5}\,\gamma^{\frac{4s}{5}}\,d\mu\bigg)^\frac54
 + \bigg(\int_M |\nabla A|^3\,\gamma^{\frac{s}{2}}\,d\mu\bigg)^2
\notag\\
     &\le 
     c\big(\vn{A}^\frac23_{4,[\gamma>0]} + 
       \vn{A}_{4,[\gamma>0]}^\frac43\big)
     \int_M \big(|\nabla_{(3)}A|^2 + |A|^2\,|\nabla_{(2)}A|^2 + |\nabla A|^4\big)\,\gamma^s\,d\mu
\notag\\&\quad
 + c\vn{A}_{4,[\gamma>0]}^\frac{20}{3}
    \bigg(\int_M |\nabla A|^3\,\gamma^\frac{s}{2}\,d\mu\bigg)^2
\notag\\&\quad
 + c(c_\gamma)^4[(c_\gamma)^4\mu_\gamma(f)]^\frac{1}{2}\Big(
   \vn{A}_{4,[\gamma>0]}^\frac{20}{3}
 + \vn{A}_{4,[\gamma>0]}^\frac{11}{3}
 + [(c_\gamma)^4\mu_\gamma(f)]^\frac{3}{10}
   \vn{A}_{4,[\gamma>0]}^\frac{46}{15}\Big)
   \,.
\label{EQplace4}
\end{align}
This estimates the second term in \eqref{EQpreest}.
For the third term, we estimate:
\begin{align}
   \bigg(
   s(c_\gamma)\int_M &|A|^\frac32\,|\nabla_{(2)}A|^\frac32\,\gamma^{\frac{3s}{4}-1}\,d\mu
   \bigg)^{\frac43}
\notag\\&\le c(c_\gamma)^\frac43
         \bigg(
           \int_M |A|^6\,\gamma^\frac{s}2\,d\mu
         \bigg)^\frac13
           \int_M |\nabla_{(2)}A|^2\,\gamma^\frac{5s-8}6\,d\mu
\notag\\&\le c(c_\gamma)^\frac43
         \bigg(
           \int_{[\gamma>0]} |A|^4\,d\mu
         \bigg)^\frac16
         \bigg(
           \int_M |A|^8\,\gamma^s\,d\mu
         \bigg)^\frac16
           \int_M |\nabla_{(2)}A|^2\,\gamma^\frac{5s-8}6\,d\mu
\notag\\&\le c\vn{A}_{4,[\gamma>0]}^\frac23\int_M |A|^8\,\gamma^s\,d\mu
        + c(c_\gamma)^\frac85\vn{A}_{4,[\gamma>0]}^\frac23
         \bigg(
           \int_M |\nabla_{(2)}A|^2\,\gamma^\frac{5s-8}6\,d\mu
         \bigg)^\frac65\,.
\label{EQplace8}
\end{align}
In order to estimate the last term, we first calculate
\begin{align*}
(c_\gamma)^2
           \int_M |\nabla A|^2\,\gamma^\frac{5s-20}6\,d\mu
&\le 
   c(c_\gamma)^2
           \int_M |A|\,|\nabla_{(2)} A|\,\gamma^\frac{5s-20}6\,d\mu
 + c(c_\gamma)^3
           \int_M |A|\,|\nabla A|\,\gamma^\frac{5s-26}6\,d\mu
\\
&\le 
\frac12(c_\gamma)^2
           \int_M |\nabla A|^2\,\gamma^\frac{5s-20}6\,d\mu
 + c(c_\gamma)^2
           \int_M |A|\,|\nabla_{(2)} A|\,\gamma^\frac{5s-20}6\,d\mu
\\&\qquad
 + c(c_\gamma)^4
           \int_M |A|^2\,\gamma^\frac{5s-32}6\,d\mu
\,.
\end{align*}
Absorbing, estimating, and using $s\ge\frac{32}5$, we find
\begin{align}
c(c_\gamma)^2
           \int_M |\nabla A|^2\,\gamma^\frac{5s-20}6\,d\mu
&\le 
   c(c_\gamma)^2
           \int_M |A|\,|\nabla_{(2)} A|\,\gamma^\frac{5s-20}6\,d\mu
 + c(c_\gamma)^4
           \int_M |A|^2\,\gamma^\frac{5s-32}6\,d\mu
\notag\\
&\le 
   \frac12
           \int_M |\nabla_{(2)} A|^2\,\gamma^\frac{5s-8}6\,d\mu
 + c(c_\gamma)^4\mu_\gamma(f)^\frac12
           \vn{A}_{4,[\gamma>0]}^2
\,.
\label{EQplace5}
\end{align}
Returning now to the last term of \eqref{EQplace8}, we estimate
\begin{align*}
         \bigg(
           \int_M |\nabla_{(2)}A|^2\,\gamma^\frac{5s-8}6\,d\mu
         \bigg)^\frac65
&\le
        c\bigg(
           \int_M |\nabla A|\,|\nabla_{(3)}A|\,\gamma^\frac{5s-8}6\,d\mu
         \bigg)^\frac65
      + c(c_\gamma)^\frac65
         \bigg(
           \int_M |\nabla A|\,|\nabla_{(2)}A|\,\gamma^\frac{5s-14}6\,d\mu
         \bigg)^\frac65
\\
&\le
         \frac12\bigg(
           \int_M |\nabla_{(2)}A|^2\,\gamma^\frac{5s-8}6\,d\mu
         \bigg)^\frac65
      + c\bigg(
           \int_M |\nabla A|\,|\nabla_{(3)}A|\,\gamma^\frac{5s-8}6\,d\mu
         \bigg)^\frac65
\\&\qquad
      + c(c_\gamma)^\frac{12}5
         \bigg(
           \int_M |\nabla A|^2\,\gamma^\frac{5s-14}6\,d\mu
         \bigg)^\frac65
\,.
\end{align*}
Absorbing and using \eqref{EQplace5}, we find
\begin{align}
         \bigg(
           \int_M |\nabla_{(2)}A|^2\,\gamma^\frac{5s-8}6\,d\mu
         \bigg)^\frac65
&\le
        c\bigg(
           \int_M |\nabla A|\,|\nabla_{(3)}A|\,\gamma^\frac{5s-8}6\,d\mu
         \bigg)^\frac65
 + c(c_\gamma)^\frac{24}{5}\mu_\gamma(f)^\frac{3}{5}
           \vn{A}_{4,[\gamma>0]}^\frac{12}{5}
\,.
\label{EQplace6}
\end{align}
Note that \eqref{EQplace5} implies, using also $s\ge\frac{20}{3}$,
\begin{align*}
 c(c_\gamma)^{\frac{12}5}
         \bigg(
           \int_M |\nabla A|^2\,\gamma^\frac{5s-8}3\,d\mu
         \bigg)^\frac32
&\le 
 c(c_\gamma)^{\frac{12}5}
         \bigg(
           \int_M |A|\,|\nabla_{(2)} A|\,\gamma^\frac{5s-20}6\,d\mu
         \bigg)^\frac32
 + c(c_\gamma)^{\frac{27}{5}}
         \bigg(
           \int_M |A|^2\,\gamma^\frac{5s-32}6\,d\mu
         \bigg)^\frac32
\\&\le
 c(c_\gamma)^{\frac{12}5}
         \bigg(
           \int_M |A|^\frac{12}{7}\,\gamma^\frac{6s-40}7\,d\mu
         \bigg)^\frac78
         \bigg(
           \int_M |\nabla_{(2)} A|^\frac{12}{5}\,\gamma^\frac{4s}{5}\,d\mu
         \bigg)^\frac58
\\&\qquad
 + c(c_\gamma)^{\frac{27}{5}}
                \mu_\gamma(f)^\frac34
           \vn{A}_{4,[\gamma>0]}^3
\\&\le
   (c_\gamma)^{-\frac85}
         \bigg(
           \int_M |\nabla_{(2)} A|^\frac{12}{5}\,\gamma^\frac{4s}{5}\,d\mu
         \bigg)^\frac54
 + (c_\gamma)^{\frac{32}5}
         \bigg(
           \int_M |A|^\frac{12}{7}\,\gamma^\frac{6s-40}7\,d\mu
         \bigg)^\frac74
\\&\qquad
 + c(c_\gamma)^{\frac{27}{5}}
                \mu_\gamma(f)^\frac34
           \vn{A}_{4,[\gamma>0]}^3
\\&\le
   (c_\gamma)^{-\frac85}
         \bigg(
           \int_M |\nabla_{(2)} A|^\frac{12}{5}\,\gamma^\frac{4s}{5}\,d\mu
         \bigg)^\frac54
 + (c_\gamma)^{\frac{32}5}
   \mu_\gamma(f)\vn{A}_{4,[\gamma>0]}^3
\\&\qquad
 + c(c_\gamma)^{\frac{27}{5}}
                \mu_\gamma(f)^\frac34
           \vn{A}_{4,[\gamma>0]}^3
\,.
\end{align*}
This yields the estimate
\begin{align*}
        c\bigg(
           \int_M |\nabla A|\,|\nabla_{(3)}A|\,\gamma^\frac{5s-8}6\,d\mu
         \bigg)^\frac65
&\le
        c\bigg(
           \int_M |\nabla A|^2\,\gamma^\frac{5s-8}3\,d\mu
         \bigg)^\frac35
         \bigg(
           \int_M |\nabla_{(3)}A|^2\,\gamma^{s}\,d\mu
         \bigg)^\frac35
\\
&\le
 c(c_\gamma)^{\frac{12}5}
         \bigg(
           \int_M |\nabla A|^2\,\gamma^\frac{5s-8}3\,d\mu
         \bigg)^\frac32
+ (c_\gamma)^{-\frac85}
           \int_M |\nabla_{(3)}A|^2\,\gamma^{s}\,d\mu
\\&\le
  (c_\gamma)^{-\frac85}
           \int_M |\nabla_{(3)}A|^2\,\gamma^{s}\,d\mu
+ 
  (c_\gamma)^{-\frac85}
         \bigg(
           \int_M |\nabla_{(2)} A|^\frac{12}{5}\,\gamma^\frac{4s}{5}\,d\mu
         \bigg)^\frac54
\\&\qquad
 + c(c_\gamma)^{\frac{12}{5}}
    [(c_\gamma)^4\mu_\gamma(f)]^\frac34
	   \big( 1 + [(c_\gamma)^4\mu_\gamma(f)]^\frac14 \big)\vn{A}_{4,[\gamma>0]}^3
\,.
\end{align*}
which we combine with \eqref{EQplace6} and \eqref{EQplace8} to find
\begin{align}
   \bigg(
   s(c_\gamma)\int_M &|A|^\frac32\,|\nabla_{(2)}A|^\frac32\,\gamma^{\frac{3s}{4}-1}\,d\mu
   \bigg)^{\frac43}
\notag\\&\le
  c\vn{A}_{4,[\gamma>0]}^\frac23
           \int_M \big(|\nabla_{(3)}A|^2 + |A|^8\big)\,\gamma^{s}\,d\mu
+ 
  c\vn{A}_{4,[\gamma>0]}^\frac23
         \bigg(
           \int_M |\nabla_{(2)} A|^\frac{12}{5}\,\gamma^\frac{4s}{5}\,d\mu
         \bigg)^\frac54
\notag\\&\quad
 + c(c_\gamma)^4
    [(c_\gamma)^4\mu_\gamma(f)]^\frac35
	   \Big(
         [(c_\gamma)^4\mu_\gamma(f)]^\frac{3}{20}
           + [(c_\gamma)^4\mu_\gamma(f)]^\frac{2}{5}
           + \vn{A}_{4,[\gamma>0]}^\frac{1}{15}
           \Big)\vn{A}_{4,[\gamma>0]}^3
\,.
\label{EQplace7}
\end{align}
This estimates the third term in \eqref{EQpreest}.
Combining now \eqref{EQpreest}, \eqref{EQplace7} and \eqref{EQplace4} we find
\begin{align}
\int_M &|A|^2|\nabla_{(2)}A|^2\,\gamma^s\,d\mu
 + \bigg(\int_M |\nabla_{(2)}A|^{\frac{12}5}\,\gamma^{\frac{4s}{5}}\,d\mu\bigg)^\frac54
 + \bigg(\int_M |\nabla A|^3\,\gamma^{\frac{s}{2}}\,d\mu\bigg)^2
\notag\\
&\le c\bigg(
 	\int_M |A|^\frac12\,|\nabla A|\,|\nabla_{(2)}A|^\frac32\,\gamma^{\frac{3s}{4}}\,d\mu
      \bigg)^{\frac43}
 + \bigg(\int_M |\nabla_{(2)}A|^{\frac{12}5}\,\gamma^{\frac{4s}{5}}\,d\mu\bigg)^\frac54
\notag \\&\qquad
   + c\bigg(
 	\int_M |A|^\frac32\,|\nabla_{(2)}A|^\frac12\,|\nabla_{(3)}A|\,\gamma^{\frac{3s}{4}}\,d\mu
      \bigg)^{\frac43}
 + \bigg(\int_M |\nabla A|^3\,\gamma^{\frac{s}{2}}\,d\mu\bigg)^2
\notag \\&\qquad
   + c\bigg(
   s(c_\gamma)\int_M |A|^\frac32\,|\nabla_{(2)}A|^\frac32\,\gamma^{\frac{3s}{4}-1}\,d\mu
      \bigg)^{\frac43}
   + c\vn{A}_{4,[\gamma>0]}^\frac23
                         \int_M |A|^2\,|\nabla_{(2)}A|^2\,\gamma^{s}\,d\mu
\notag\\
&\le
     c\big(\vn{A}^\frac23_{4,[\gamma>0]} + 
       \vn{A}_{4,[\gamma>0]}^\frac43\big)
     \int_M \big(|\nabla_{(3)}A|^2 + |A|^2\,|\nabla_{(2)}A|^2 + |\nabla A|^4 + |A|^8\big)\,\gamma^s\,d\mu
\notag\\&\quad
 + c\vn{A}_{4,[\gamma>0]}^\frac{20}{3}
    \bigg(\int_M |\nabla A|^3\,\gamma^\frac{s}{2}\,d\mu\bigg)^2
 + c\vn{A}_{4,[\gamma>0]}^\frac23
         \bigg(
           \int_M |\nabla_{(2)} A|^\frac{12}{5}\,\gamma^\frac{4s}{5}\,d\mu
         \bigg)^\frac54
\notag\\&\quad
 + c(c_\gamma)^4[(c_\gamma)^4\mu_\gamma(f)]^\frac{1}{2}\Big(
   \vn{A}_{4,[\gamma>0]}^\frac{11}{3}
 + \vn{A}_{4,[\gamma>0]}^\frac{2}{3}
 + [(c_\gamma)^4\mu_\gamma(f)]^\frac{3}{10}
   \vn{A}_{4,[\gamma>0]}^\frac{1}{15}
\notag\\&\qquad\qquad\qquad\quad\ 
       + [(c_\gamma)^4\mu_\gamma(f)]^\frac{1}{4}
           + [(c_\gamma)^4\mu_\gamma(f)]^\frac{4}{5}
           + [(c_\gamma)^4\mu_\gamma(f)]^\frac{1}{10}
             \vn{A}_{4,[\gamma>0]}^\frac{1}{15}
           \Big)\vn{A}_{4,[\gamma>0]}^3
      \,.
\label{EQbigest}
\end{align}
We combine \eqref{EQbigest} above with our earlier estimate \eqref{EQnabla4} to find
\begin{align}
\int_M &\big(|A|^2|\nabla_{(2)}A|^2 + |\nabla A|^4\big)\,\gamma^s\,d\mu
 + \bigg(\int_M |\nabla_{(2)}A|^{\frac{12}5}\,\gamma^{\frac{4s}{5}}\,d\mu\bigg)^\frac54
 + \bigg(\int_M |\nabla A|^3\,\gamma^{\frac{s}{2}}\,d\mu\bigg)^2
\notag\\
&\le
     \big(\theta +
       c\vn{A}_{4,[\gamma>0]}^\frac43\big)
     \int_M \big(|\nabla_{(3)}A|^2 + |A|^2\,|\nabla_{(2)}A|^2 + |\nabla A|^4 + |A|^8\big)\,\gamma^s\,d\mu
\notag\\&\quad
 + c\vn{A}_{4,[\gamma>0]}^\frac{20}{3}
    \bigg(\int_M |\nabla A|^3\,\gamma^\frac{s}{2}\,d\mu\bigg)^2
 + c\vn{A}_{4,[\gamma>0]}^\frac23
         \bigg(
           \int_M |\nabla_{(2)} A|^\frac{12}{5}\,\gamma^\frac{4s}{5}\,d\mu
         \bigg)^\frac54
\notag\\&\quad
 + c(c_\gamma)^4\Big(
           1
           + \vn{A}_{4,[\gamma>0]}^4
           + [(c_\gamma)^4\mu_\gamma(f)]^6
           \Big)\vn{A}_{4,[\gamma>0]}^3
      \,.
\label{EQbiggerest}
\end{align}
Note that we have interpolated terms inside the parentheses of the coefficient of the first and last terms.

It remains only to estimate the term $\int_M |A|^8\,\gamma^s\,d\mu$, which we
do so now with the Michael-Simon Sobolev inequality:
\begin{align}
\int_M |A|^8\,\gamma^s\,d\mu
&\le
     c\bigg(
        \int_M |A|^5\,|\nabla A|\,\gamma^\frac{3s}{4}\,d\mu
      \bigg)^\frac43
   + c\bigg(
        \int_M |A|^7\,\gamma^\frac{3s}{4}\,d\mu
      \bigg)^\frac43
   + c\bigg(
        (c_\gamma)\int_M |A|^6\,\gamma^\frac{3s-4}{4}\,d\mu
      \bigg)^\frac43
\notag\\
&\le
     c\bigg(
        \int_M |A|^6\,\gamma^\frac{s}{2}\,d\mu
      \bigg)^\frac23
      \bigg(
        \int_M |A|^4\,|\nabla A|^2\,\gamma^{s}\,d\mu
      \bigg)^\frac23
   + c\vn{A}_{4,[\gamma>0]}^\frac43
      \int_M |A|^8\,\gamma^s\,d\mu
\notag\\&\qquad
   + c(c_\gamma)^\frac43
      \vn{A}_{4,[\gamma>0]}^\frac83
      \bigg(
        \int_M |A|^8\,\gamma^{s}\,d\mu
      \bigg)^\frac23
\notag\\
&\le
     c\vn{A}_{4,[\gamma>0]}^\frac43
      \bigg(
        \int_M |A|^8\,\gamma^{s}\,d\mu
      \bigg)^\frac23
      \bigg(
        \int_M |\nabla A|^4\,\gamma^s\,d\mu
      \bigg)^\frac13
   + c\vn{A}_{4,[\gamma>0]}^\frac43
      \int_M |A|^8\,\gamma^s\,d\mu
\notag\\&\qquad
   + c(c_\gamma)^4
      \vn{A}_{4,[\gamma>0]}^\frac{16}3
\notag\\
&\le
     c\vn{A}_{4,[\gamma>0]}^\frac43
      \int_M \big(|\nabla A|^4 + |A|^8\big)\,\gamma^s\,d\mu
   + c(c_\gamma)^4
      \vn{A}_{4,[\gamma>0]}^\frac{16}3
\,.
\label{EQinter}
\end{align}
Combining this estimate with \eqref{EQbiggerest} and also the interpolation
$2|A|^4|\nabla A|^2 \le |A|^8 + |\nabla A|^4$ we conclude
\begin{align}
\int_M &\big(|A|^2|\nabla_{(2)}A|^2 + |A|^4|\nabla A|^2 + |\nabla A|^4 + |A|^8\big)\,\gamma^s\,d\mu
\notag\\&\hskip-3mm
 + \bigg(\int_M |\nabla_{(2)}A|^{\frac{12}5}\,\gamma^{\frac{4s}{5}}\,d\mu\bigg)^\frac54
 + \bigg(\int_M |\nabla A|^3\,\gamma^{\frac{s}{2}}\,d\mu\bigg)^2
\notag\\
&\le
     \big(\theta +
       c\vn{A}_{4,[\gamma>0]}^\frac43\big)
     \int_M \big(|\nabla_{(3)}A|^2 + |A|^2\,|\nabla_{(2)}A|^2 + |\nabla A|^4 + |A|^8\big)\,\gamma^s\,d\mu
\notag\\&\quad
 + c\vn{A}_{4,[\gamma>0]}^\frac{20}{3}
    \bigg(\int_M |\nabla A|^3\,\gamma^\frac{s}{2}\,d\mu\bigg)^2
 + c\vn{A}_{4,[\gamma>0]}^\frac23
         \bigg(
           \int_M |\nabla_{(2)} A|^\frac{12}{5}\,\gamma^\frac{4s}{5}\,d\mu
         \bigg)^\frac54
\notag\\&\quad
 + c(c_\gamma)^4\Big(
           1
           + \vn{A}_{4,[\gamma>0]}^4
           + [(c_\gamma)^4\mu_\gamma(f)]^6
           \Big)\vn{A}_{4,[\gamma>0]}^3
      \,,
\notag
\end{align}
as required.

Next we consider (iv).
We begin by estimating with the Michael-Simon Sobolev inequality
\begin{align*}
\int_M |A|^{10}\gamma^s\,d\mu
 &\le c\bigg(\int_M |A|^\frac{13}{2}|\nabla A|\,\gamma^\frac{3s}{4}\,d\mu\bigg)^\frac43
    + c\bigg(\int_M |A|^\frac{17}{2}\,\gamma^\frac{3s}{4}\,d\mu\bigg)^\frac43
    + c(\cg)^\frac43\bigg(\int_M |A|^\frac{15}{2}\,\gamma^\frac{3s-4}{4}\,d\mu\bigg)^\frac43
\\
 &\le c\bigg(\int_M |A|^6|\nabla A|^2\,\gamma^s\,d\mu\bigg)^\frac23
       \bigg(\int_M |A|^7\,\gamma^\frac{s}{2}\,d\mu\bigg)^\frac23
    + c\vn{A}_{4,[\gamma>0]}^\frac43\int_M |A|^{10}\,\gamma^s\,d\mu
\\&\quad
    + c(\cg)^\frac43\bigg(\int_M |A|^5\,\gamma^\frac{s-4}{2}\,d\mu\bigg)^\frac23
                    \bigg(\int_M |A|^{10}\,\gamma^s\,d\mu\bigg)^\frac23
\\
 &\le \theta\int_M |A|^6|\nabla A|^2\,\gamma^s\,d\mu
    + c\big(\theta + \vn{A}_{4,[\gamma>0]}^\frac43\big)
       \int_M |A|^{10}\,\gamma^s\,d\mu
\\&\quad
    + c\bigg(\int_M |A|^7\,\gamma^\frac{s}{2}\,d\mu\bigg)^2
    + c(\cg)^4\bigg(\int_M |A|^5\,\gamma^\frac{s-4}{2}\,d\mu\bigg)^2
\\
 &\le \theta\int_M |A|^6|\nabla A|^2\,\gamma^s\,d\mu
    + c\big(\theta + \vn{A}_{4,[\gamma>0]}^4 + \vn{A}_{4,[\gamma>0]}^\frac43\big)
       \int_M |A|^{10}\,\gamma^s\,d\mu
\\&\quad
    + c(\cg)^4\vn{A}_{4,[\gamma>0]}^4\int_M |A|^6\,\gamma^{s-4}\,d\mu
\\
 &\le \theta\int_M |A|^6|\nabla A|^2\,\gamma^s\,d\mu
    + c\big(\theta + \vn{A}_{4,[\gamma>0]}^4\big)
       \int_M |A|^{10}\,\gamma^s\,d\mu
    + c(\cg)^8\mu_\gamma(f_t)^\frac12\vn{A}_{4,[\gamma>0]}^{2}
\,.
\end{align*}
Note that in the last step we used $s\ge8$.

We shall move gradually higher in order.
Next we estimate a first order term:
\begin{align*}
\int_M |\nabla A|^2|A|^6\gamma^s\,d\mu
 &\le c\bigg(\int_M |A|^\frac{9}{2}|\nabla_{(2)}A||\nabla A|^\frac12\,\gamma^\frac{3s}{4}\,d\mu\bigg)^\frac43
    + c\bigg(\int_M |A|^\frac{7}{2}|\nabla A|^\frac52\,\gamma^\frac{3s}{4}\,d\mu\bigg)^\frac43
\\&\quad
    + c\bigg(\int_M |A|^\frac{11}{2}|\nabla A|^\frac32\,\gamma^\frac{3s}{4}\,d\mu\bigg)^\frac43
    + c(\cg)^\frac43\bigg(\int_M |A|^\frac{9}{2}|\nabla A|^\frac32\,\gamma^\frac{3s-4}{4}\,d\mu\bigg)^\frac43
\\
 &\le c\bigg(\int_M |\nabla_{(2)}A|^2|A|^4\,\gamma^s\,d\mu\bigg)^\frac23
       \bigg(\int_M |\nabla A||A|^5\,\gamma^\frac{s}{2}\,d\mu\bigg)^\frac23
\\&\quad
    + c\bigg(\int_M |\nabla A|^5\,\gamma^s\,d\mu\bigg)^\frac23
       \bigg(\int_M |A|^7\,\gamma^\frac{s}{2}\,d\mu\bigg)^\frac23
\\&\quad
    + c(\cg)^\frac43\bigg(\int_M |\nabla A|^2|A|^4\,\gamma^{s-2}\,d\mu\bigg)^\frac23
                    \bigg(\int_M |\nabla A||A|^5\,\gamma^\frac{s}{2}\,d\mu\bigg)^\frac23
\\&\quad
    + c\vn{A}_{4,[\gamma>0]}^4\int_M |\nabla A|^2|A|^6\gamma^s\,d\mu
\\
 &\le \theta\int_M \big(|\nabla_{(2)}A|^2|A|^4 + |\nabla A|^5\big)\,\gamma^s\,d\mu
     + c\bigg(\int_M |\nabla A||A|^5\,\gamma^\frac{s}{2}\,d\mu\bigg)^2
\\&\quad
    + c\vn{A}_{4,[\gamma>0]}^4
       \int_M |A|^{10}\,\gamma^s\,d\mu
    + c(\cg)^2\int_M |\nabla A|^2|A|^4\,\gamma^{s-2}\,d\mu
\\&\quad
    + c\vn{A}_{4,[\gamma>0]}^4\int_M |\nabla A|^2|A|^6\gamma^s\,d\mu
\\
 &\le \theta\int_M \big(|\nabla_{(2)}A|^2|A|^4 + |\nabla A|^5\big)\,\gamma^s\,d\mu
    + c\vn{A}_{4,[\gamma>0]}^4\int_M |\nabla A|^2|A|^6\gamma^s\,d\mu
\\&\quad
    + c\vn{A}_{4,[\gamma>0]}^4
       \int_M |A|^{10}\,\gamma^s\,d\mu
    + c(\cg)^2\int_M |\nabla A|^2|A|^4\,\gamma^{s-2}\,d\mu
\,.
\end{align*}
Combining this with the estimate (recall $3s\ge20$)
\begin{align*}
    c(\cg)^2\int_M |\nabla A|^2|A|^4\,\gamma^{s-2}\,d\mu
&\le 
    \theta\int_M |\nabla A|^5\,\gamma^s\,d\mu
  + c(\cg)^{\frac{10}{3}}\int_M |A|^\frac{20}{3}\,\gamma^{\frac{3s-10}{3}}\,d\mu
\\
&\le 
    \theta\int_M |\nabla A|^5\,\gamma^s\,d\mu
  + c(\cg)^{\frac{10}{3}}\bigg(\int_M |A|^{10}\,\gamma^{s}\,d\mu\bigg)^\frac12
                         \bigg(\int_M |A|^\frac{10}{3}\,\gamma^{\frac{3s-20}{6}}\,d\mu\bigg)^\frac12
\\
&\le 
    \theta\int_M |\nabla A|^5\,\gamma^s\,d\mu
  + c(\cg)^{\frac{10}{3}}\bigg(\int_M |A|^{10}\,\gamma^{s}\,d\mu\bigg)^\frac12
                         \bigg(\mu_\gamma(f)^\frac16\vn{A}_{4,[\gamma>0]}^\frac{10}{3}\bigg)^\frac12
\\
&\le 
    \theta\int_M \big(|\nabla A|^5+|A|^{10}\big)\,\gamma^s\,d\mu
  + c(\cg)^6[(\cg)^4\mu_\gamma(f)]^\frac16\vn{A}_{4,[\gamma>0]}^\frac{10}{3}
\end{align*}
yields
\begin{align*}
\int_M |\nabla A|^2|A|^6\gamma^s\,d\mu
 &\le \theta\int_M \big(|\nabla_{(2)}A|^2|A|^4 + |\nabla A|^5 + |A|^{10}\big)\,\gamma^s\,d\mu
    + c\vn{A}_{4,[\gamma>0]}^4\int_M |\nabla A|^2|A|^6\gamma^s\,d\mu
\\&\quad
    + c\vn{A}_{4,[\gamma>0]}^4
       \int_M |A|^{10}\,\gamma^s\,d\mu
  + c(\cg)^6[(\cg)^4\mu_\gamma(f)]^\frac16\vn{A}_{4,[\gamma>0]}^\frac{10}{3}
\,.
\end{align*}
These estimates combine to yield
\begin{align}
\int_M \big(|\nabla A|^2|A|^6 + |A|^{10}\big)\gamma^s\,d\mu
 &\le \theta\int_M \big(|\nabla_{(2)}A|^2|A|^4 + |\nabla A|^5 + |A|^{10}\big)\,\gamma^s\,d\mu
\notag\\&\quad
\label{EQcombined1}
    + c\vn{A}_{4,[\gamma>0]}^4
       \int_M \big(|\nabla A|^2|A|^6 + |A|^{10}\big)\gamma^s\,d\mu
\\&\quad
  + c(\cg)^6\big([(\cg)^4\mu_\gamma(f)]^\frac12 + [(\cg)^4\mu_\gamma(f)]^\frac16\big)
            \big(\vn{A}_{4,[\gamma>0]}^\frac{10}{3} + \vn{A}_{4,[\gamma>0]}^{2}\big)
\notag\,.
\end{align}
Now we estimate
\begin{align*}
\int_M |\nabla A|^5\gamma^s\,d\mu
 &\le 
 c\int_M |\nabla_{(2)}A|\,|\nabla A|^3\,|A|\,\gamma^s\,d\mu
 + c(\cg)\int_M |\nabla A|^4\,|A|\,\gamma^{s-1}\,d\mu
\\
 &\le 
 c\int_M |\nabla_{(2)}A|^2|\nabla A|^2\,\gamma^s\,d\mu
 + \frac{\delta}{10}\int_M |\nabla A|^4\,|A|^2\,\gamma^s\,d\mu
\\&\quad
 + \frac12\int_M |\nabla A|^5\,\gamma^{s}\,d\mu
 + c(\cg)\int_M |\nabla A|^4\,|A|\,\gamma^{s-1}\,d\mu
\\
 &\le 
 c\int_M |\nabla_{(2)}A|^2|\nabla A|^2\,\gamma^s\,d\mu
 + \delta\int_M |\nabla A|^5\,\gamma^s\,d\mu
 + \delta\int_M |A|^{10}\,\gamma^s\,d\mu
\\&\quad
 + \frac34\int_M |\nabla A|^5\,\gamma^{s}\,d\mu
 + c(\cg)^{5}\int_M |A|^5\,\gamma^{s-5}\,d\mu
\end{align*}
so that absorbing for $\delta$ sufficiently small yields
\begin{align*}
\int_M |\nabla A|^5\gamma^s\,d\mu
 &\le 
 c\int_M |\nabla_{(2)}A|^2|\nabla A|^2\,\gamma^s\,d\mu
 + \theta\int_M |A|^{10}\,\gamma^s\,d\mu
\\&\quad
 + c(\cg)^{5}\int_M |A|^5\,\gamma^{s-5}\,d\mu
\,.
\end{align*}
Now (recall $4s\ge25$)
\begin{align*}
 c(\cg)^{5}\int_M |A|^5\,\gamma^{s-5}\,d\mu
 &\le \theta\int_M |A|^{10}\,\gamma^{s}\,d\mu
 + c(\cg)^\frac{25}{4}\int_M |A|^{\frac{15}{4}}\,\gamma^{\frac{4s-25}{4}}\,d\mu
 \\
 &\le \theta\int_M |A|^{10}\,\gamma^{s}\,d\mu
 + c(\cg)^6[(\cg)^4\mu_\gamma(f)]^\frac{1}{16}\vn{A}_{4,[\gamma>0]}^{\frac{15}{4}}
\,.
\end{align*}
Combining this with the previous estimate we find
\begin{align*}
\int_M |\nabla A|^5\gamma^s\,d\mu
 &\le 
 c\int_M |\nabla_{(2)}A|^2|\nabla A|^2\,\gamma^s\,d\mu
 + \theta\int_M |A|^{10}\,\gamma^s\,d\mu
\\&\quad
 + c(\cg)^6[(\cg)^4\mu_\gamma(f)]^\frac{1}{16}\vn{A}_{4,[\gamma>0]}^{\frac{15}{4}}
\,.
\end{align*}
Using this we estimate the RHS of \eqref{EQcombined1} and absorb to find
\begin{align}
\int_M \big(|\nabla A|^2|A|^6 + |A|^{10}\big)\gamma^s\,d\mu
&\le \theta\int_M \big(|\nabla_{(2)}A|^2|A|^4 + |\nabla_{(2)}A|^2|\nabla A|^2\big)\,\gamma^s\,d\mu
\notag\\&\quad
    + c\vn{A}_{4,[\gamma>0]}^4
       \int_M \big(|\nabla A|^2|A|^6 + |A|^{10}\big)\gamma^s\,d\mu
\notag\\&\quad
  + c(\cg)^6\vn{A}_{4,[\gamma>0]}^2
            \big(1+[(\cg)^4\mu_\gamma(f)]^\frac12
            \big)
            \big(1 + \vn{A}_{4,[\gamma>0]}^{\frac{3}{4}}
	    \big)
\label{EQcombinedx}\,.
\end{align}
Note that we interpolated some terms in the last product on the right hand side.

Now we move on to terms involving $\nabla_{(2)}A$.
We begin with
\begin{align}
\int_M &|\nabla_{(2)}A|^2|A|^4\,\gamma^s\,d\mu
\notag\\
 &\le c\bigg(
          \int_M |\nabla_{(3)}A||\nabla_{(2)}A|^\frac12|A|^3
                 \gamma^{\frac{3s}{4}}\,d\mu
        + \int_M |\nabla_{(2)}A|^\frac32|\nabla A||A|^2
                 \gamma^{\frac{3s}{4}}\,d\mu
\notag\\&\qquad
        + \int_M |\nabla_{(2)}A|^\frac32|A|^4
                 \gamma^{\frac{3s}{4}}\,d\mu
        + (\cg)\int_M |\nabla_{(2)}A|^\frac32|A|^3
                 \gamma^{\frac{3s-4}{4}}\,d\mu
       \bigg)^\frac43
\notag\\
 &\le 
          \bigg(
          \int_M |\nabla_{(3)}A|^2|A|^2
                 \gamma^{s}\,d\mu
          \bigg)^\frac23
          \bigg(
          \int_M |\nabla_{(2)}A||A|^4
                 \gamma^{\frac{s}{2}}\,d\mu
          \bigg)^\frac23
\notag\\&\quad
        + c
       \bigg(
        \int_M |\nabla_{(2)}A|^2|\nabla A|^2
                 \gamma^s\,d\mu
       \bigg)^\frac23
       \bigg(
        \int_M |\nabla_{(2)}A||A|^4
                 \gamma^{\frac{s}{2}}\,d\mu
       \bigg)^\frac23
\notag\\&\quad
        + c
       \bigg(
        \int_M |\nabla_{(2)}A|^2|A|^4
                 \gamma^s\,d\mu
       \bigg)^\frac23
       \bigg(
        \int_M |\nabla_{(2)}A||A|^4
                 \gamma^{\frac{s}{2}}\,d\mu
       \bigg)^\frac23
\notag\\&\quad
       + c(\cg)^\frac43
       \bigg(
         \int_M |\nabla_{(2)}A|^2|A|^2
                 \gamma^{s-4}\,d\mu
       \bigg)^\frac23
       \bigg(
        \int_M |\nabla_{(2)}A||A|^4
                 \gamma^{\frac{s}{2}}\,d\mu
       \bigg)^\frac23
\notag\\
 &\le 
          c\int_M |\nabla_{(3)}A|^2|A|^2
                 \gamma^{s}\,d\mu
         + \theta\int_M \big(|\nabla_{(2)}A|^2|\nabla A|^2 + |\nabla_{(2)}A|^2|A|^4\big)
                 \gamma^s\,d\mu
\notag\\&\quad
           + c\bigg(\int_M |\nabla_{(2)}A||A|^4\gamma^{\frac{s}{2}}\,d\mu\bigg)^2
       + c(\cg)^2
         \int_M |\nabla_{(2)}A|^2|A|^2
                 \gamma^{s-4}\,d\mu
\notag\\
 &\le 
          c\int_M |\nabla_{(3)}A|^2|A|^2
                 \gamma^{s}\,d\mu
         + (\theta + c\vn{A}_{4,[\gamma>0]}^4)\int_M
                 \big(
             |\nabla_{(2)}A|^2|\nabla A|^2
           + |\nabla_{(2)}A|^2|A|^4
                 \big)
                 \gamma^s\,d\mu
\notag\\&\quad
       + c(\cg)^4
         \int_M |\nabla_{(2)}A|^2
                 \gamma^{s-8}\,d\mu
\,.
\label{EQcombined4}
\end{align}
Since
\begin{align*}
(\cg)^2\int_M |\nabla_{(k)}A|^2\,\gamma^{s-4}\,d\mu
&\le \frac12(\cg)^2\int_M |\nabla_{(k)}A|^2\,\gamma^{s-4}\,d\mu
 + \theta\int_M |\nabla_{(k+1)}A|^2\,\gamma^{s}\,d\mu
\\&\qquad
 + c(\cg)^4\int_M |\nabla_{(k-1)}A|^2\,\gamma^{s-8}\,d\mu
\end{align*}
implies
\begin{equation}
(\cg)^2\int_M |\nabla_{(k)}A|^2\,\gamma^{s-4}\,d\mu
 \le \theta\int_M |\nabla_{(k+1)}A|^2\,\gamma^{s}\,d\mu
 + c(\cg)^4\int_M |\nabla_{(k-1)}A|^2\,\gamma^{s-8}\,d\mu
\,,
\label{EQinducto}
\end{equation}
we have
\begin{align*}
c(\cg)^4\int_M |\nabla_{(2)}A|^2\gamma^{s-8}\,d\mu
&\le
       \theta\int_M |\nabla_{(4)}A|^2\gamma^{s}\,d\mu
     + c(\cg)^{8}\int_M |A|^2\gamma^{s-16}\,d\mu
\\
&\le
       \theta\int_M |\nabla_{(4)}A|^2\gamma^{s}\,d\mu
     + c(\cg)^{6}[(\cg)^4\mu_\gamma(f)]^\frac12\vn{A}_{4,[\gamma>0]}^2
\,.
\end{align*}
Combining this with the estimate \eqref{EQcombined4} yields
\begin{align}
\int_M |\nabla_{(2)}A|^2|A|^4\,\gamma^s\,d\mu
 &\le 
           (\theta + c\vn{A}_{4,[\gamma>0]}^4)\int_M
                 \big(
             |\nabla_{(4)}A|^2
           + |\nabla_{(2)}A|^2|\nabla A|^2
           + |\nabla_{(2)}A|^2|A|^4
                 \big)
                 \gamma^s\,d\mu
\notag\\&\quad
     + c\int_M |\nabla_{(3)}A|^2|A|^2
                 \gamma^{s}\,d\mu
     + c(\cg)^{6}[(\cg)^4\mu_\gamma(f)]^\frac12\vn{A}_{4,[\gamma>0]}^2
\,.
\label{EQcombined41}
\end{align}
In order to estimate the remaining term involving $\nabla_{(2)}A$ we first note the following equality:
\begin{align*}
	- \int_M (\nabla_{ij}&A_{kl}\nabla^jA^{kl}) (\nabla^i|\nabla A|^2)
         \,\gamma^s\,d\mu\\
	 &= 
- 2\int_M (\nabla_{ij}A_{kl}\nabla^jA^{kl}) (\nabla^i\nabla_pA_{qr}\nabla^pA^{qr})
         \,\gamma^s\,d\mu
 = -\frac12\int_M \big|\nabla|\nabla A|^2\big|^2\,\gamma^s\,d\mu
\,.
\end{align*}
In particular, this term has a sign.
We use this to estimate
\begin{align}
\int_M |\nabla_{(2)}A|^2|\nabla A|^2\,\gamma^s\,d\mu
&\le c\int_M |\nabla_{(3)}A||\nabla A|^3\,\gamma^s\,d\mu
   + c(\cg)\int_M |\nabla_{(2)}A||\nabla A|^3\,\gamma^{s-1}\,d\mu
\notag\\
&\le 
\frac12\int_M |\nabla_{(2)}A|^2|\nabla A|^2\,\gamma^s\,d\mu
 + c\int_M |\nabla_{(3)}A||\nabla A|^3\,\gamma^s\,d\mu
 \notag\\&\qquad + c(\cg)^2\int_M |\nabla A|^4\,\gamma^{s-2}\,d\mu
\,.
\label{EQcombined5}
\end{align}
In order to control the last two terms on the right, we need two auxilliary estimates.
The first is obtained by estimating
\begin{align*}
  (\cg)^2&\int_M |\nabla A|^4\,\gamma^{s-2}\,d\mu
\le c(\cg)^2\int_M |\nabla_{(2)}A||\nabla A|^2|A|\,\gamma^{s-2}\,d\mu
  + c(\cg)^3\int_M |\nabla A|^3|A|\,\gamma^{s-3}\,d\mu
\\&\le \frac12
  (\cg)^2\int_M |\nabla A|^4\,\gamma^{s-2}\,d\mu
  + c(\cg)^2\int_M |\nabla_{(2)}A|^2|A|^2\,\gamma^{s-2}\,d\mu
  + c(\cg)^6\vn{A}_{4,[\gamma>0]}^4
\,.
\end{align*}
Absorbing gives
\begin{align*}
  (\cg)^2\int_M |\nabla A|^4\,\gamma^{s-2}\,d\mu
&\le 
    c(\cg)^2\int_M |\nabla_{(2)}A|^2|A|^2\,\gamma^{s-2}\,d\mu
  + c(\cg)^6\vn{A}_{4,[\gamma>0]}^4
\,.
\end{align*}
Estimating the first term on the right as in \eqref{EQcombined4} (the only
difference here is that we have $s-2$ instead of $s-4$), using also
\eqref{EQinducto}, we find
\begin{align}
  (\cg)^2\int_M |\nabla A|^4\,\gamma^{s-2}\,d\mu
&\le 
    \theta \int_M \big(|\nabla_{(4)}A|^2 + |\nabla_{(2)}A|^2|A|^4\big)\,\gamma^s\,d\mu
\notag\\&\qquad
  + c(\cg)^6(1 + [(\cg)^4\mu_\gamma(f)]^\frac12)(1 + \vn{A}_{4,[\gamma>0]}^2)
    \vn{A}_{4,[\gamma>0]}^2
\,.
\label{EQcombined6}
\end{align}
The second term in \eqref{EQcombined5} is estimated as follows:
\begin{align*}
 c\int_M |\nabla_{(3)}A||\nabla A|^3\,\gamma^s\,d\mu
&\le
 c\bigg(
    \int_M |\nabla_{(3)}A|^4\,\gamma^{2s}\,d\mu
  \bigg)^\frac14
  \bigg(
    \int_M |\nabla A|^4\,\gamma^{\frac{2s}{3}}\,d\mu
  \bigg)^\frac34
\\
&\le
 \theta\bigg(
    \int_M |\nabla_{(3)}A|^4\,\gamma^{2s}\,d\mu
  \bigg)^\frac12
 + c\bigg(
    \int_M |\nabla A|^4\,\gamma^{\frac{2s}{3}}\,d\mu
    \bigg)^\frac32
\,.
\end{align*}
The first term will be estimated below, it is also useful in controlling the
highest order term involving $\nabla_{(3)}A$.
For the second, we calculate
\begin{align*}
 \bigg(
    \int_M |\nabla A|^4\,\gamma^{\frac{2s}{3}}\,d\mu
 \bigg)^\frac32
&\le \bigg(
       \int_M |\nabla_{(2)}A||\nabla A|^2|A|\,\gamma^\frac{2s}{3}\,d\mu
     + (\cg)\int_M |\nabla A|^3|A|\,\gamma^\frac{2s-3}{3}\,d\mu
     \bigg)^\frac32
\\
&\le \theta
     \bigg(
       \int_M |\nabla_{(2)}A|^2|\nabla A|^2\,\gamma^s\,d\mu
     \bigg)^\frac34
     \bigg(
       \int_M |\nabla A|^2|A|^2\,\gamma^\frac{s}{3}\,d\mu
     \bigg)^\frac34
\\&\qquad
     + (\cg)^\frac32\bigg(
      \vn{A}_{4,[\gamma>0]}
       \bigg[
        \int_M |\nabla A|^4\,\gamma^{\frac{8s-12}{9}}\,d\mu
       \bigg]^\frac34
     \bigg)^\frac32
\\
&\le 
 \frac12\bigg(
    \int_M |\nabla A|^4\,\gamma^{\frac{2s}{3}}\,d\mu
 \bigg)^\frac32
  + \theta
       \int_M |\nabla_{(2)}A|^2|\nabla A|^2\,\gamma^s\,d\mu
\\&\qquad
  + \bigg(
       \int_M |\nabla A|^2|A|^2\,\gamma^\frac{s}{3}\,d\mu
     \bigg)^3
     + (\cg)^6\vn{A}_{4,[\gamma>0]}^6
\\
&\le 
 \frac12\bigg(
    \int_M |\nabla A|^4\,\gamma^{\frac{2s}{3}}\,d\mu
 \bigg)^\frac32
  + \theta
       \int_M |\nabla_{(2)}A|^2|\nabla A|^2\,\gamma^s\,d\mu
\\&\qquad
  + \vn{A}_{4,[\gamma>0]}^6
     \bigg(
       \int_M |\nabla A|^4\,\gamma^\frac{2s}{3}\,d\mu
     \bigg)^\frac32
     + (\cg)^6\vn{A}_{4,[\gamma>0]}^6
\end{align*}
Note that in the above we used $\frac{8s-12}{9} \ge \frac{2s}{3}$.
Absorbing yields
\begin{align*}
 \bigg(
    \int_M |\nabla A|^4\,\gamma^{\frac{2s}{3}}\,d\mu
 \bigg)^\frac32
&\le 
    \theta
       \int_M |\nabla_{(2)}A|^2|\nabla A|^2\,\gamma^s\,d\mu
\\&\qquad
  + \vn{A}_{4,[\gamma>0]}^6
     \bigg(
       \int_M |\nabla A|^4\,\gamma^\frac{2s}{3}\,d\mu
     \bigg)^\frac32
     + (\cg)^6\vn{A}_{4,[\gamma>0]}^6
\,.
\end{align*}
This gives the following estimate for the second term in \eqref{EQcombined5}:
\begin{align*}
 c\int_M &|\nabla_{(3)}A||\nabla A|^3\,\gamma^s\,d\mu
  + \bigg(
    \int_M |\nabla A|^4\,\gamma^{\frac{2s}{3}}\,d\mu
 \bigg)^\frac32
\\
&\le
 \theta\bigg(
    \int_M |\nabla_{(3)}A|^4\,\gamma^{2s}\,d\mu
  \bigg)^\frac12
  +  \theta
       \int_M |\nabla_{(2)}A|^2|\nabla A|^2\,\gamma^s\,d\mu
\\&\qquad
  + \vn{A}_{4,[\gamma>0]}^6
     \bigg(
       \int_M |\nabla A|^4\,\gamma^\frac{2s}{3}\,d\mu
     \bigg)^\frac32
     + (\cg)^6\vn{A}_{4,[\gamma>0]}^6
\,.
\end{align*}

Combining the second order estimates \eqref{EQcombined41}, \eqref{EQcombined5},
\eqref{EQcombined6} together, and absorbing, we have the following partial estimate:
\begin{align}
\int_M &\big(|\nabla_{(2)}A|^2|A|^4 + |\nabla_{(2)}A|^2|\nabla A|^2\big)\,\gamma^s\,d\mu
  + \bigg(
    \int_M |\nabla A|^4\,\gamma^{\frac{2s}{3}}\,d\mu
 \bigg)^\frac32
\notag\\
 &\le 
           (\theta + c\vn{A}_{4,[\gamma>0]}^4)\int_M
                 \big(
             |\nabla_{(4)}A|^2
           + |\nabla_{(2)}A|^2|\nabla A|^2
           + |\nabla_{(2)}A|^2|A|^4
                 \big)
                 \gamma^s\,d\mu
\notag\\&\quad
  + \theta\bigg(
    \int_M |\nabla_{(3)}A|^4\,\gamma^{2s}\,d\mu
     \bigg)^\frac12
  + \vn{A}_{4,[\gamma>0]}^6
     \bigg(
       \int_M |\nabla A|^4\,\gamma^\frac{2s}{3}\,d\mu
     \bigg)^\frac32
\notag\\&\quad
     + c(\cg)^{6}(1 + [(\cg)^4\mu_\gamma(f)]^\frac12)(1 +  \vn{A}_{4,[\gamma>0]}^2)\vn{A}_{4,[\gamma>0]}^2
\,.
\label{EQcombined7}
\end{align}
Let us now turn to controlling the highest order term.
We first show the following estimate, which is also needed for the terms involving $\nabla_{(2)}A$ above:
\begin{align*}
\int_M |\nabla_{(3)}A|^4\,\gamma^{2s}\,d\mu
&\le c\bigg(
	  \int_M |\nabla_{(4)}A|\,|\nabla_{(3)}A|^2\,\gamma^{\frac{3s}{2}}\,d\mu
	+ \int_M |\nabla_{(3)}A|^3|A|\,\gamma^{\frac{3s}{2}}\,d\mu
	+ (\cg)\int_M |\nabla_{(3)}A|^3\,\gamma^{\frac{3s-2}{2}}\,d\mu
      \bigg)^\frac{4}{3}
\\
&\le c\bigg(
	  \bigg[
		\int_M |\nabla_{(4)}A|^2\,\gamma^{s}\,d\mu
	  \bigg]^\frac12
	  \bigg[
		\int_M |\nabla_{(3)}A|^4\,\gamma^{2s}\,d\mu
	  \bigg]^\frac12
\\&\qquad
	+ \vn{A}_{4,[\gamma>0]}\bigg[\int_M |\nabla_{(3)}A|^4\,\gamma^{2s}\,d\mu\bigg]^\frac34
	+ (\cg)\int_M |\nabla_{(3)}A|^3\,\gamma^{\frac{3s-2}{2}}\,d\mu
      \bigg)^\frac{4}{3}
\\
&\le c
	  \bigg[
		\int_M |\nabla_{(4)}A|^2\,\gamma^{s}\,d\mu
	  \bigg]^\frac23
	  \bigg[
		\int_M |\nabla_{(3)}A|^4\,\gamma^{2s}\,d\mu
	  \bigg]^\frac23
\\&\qquad
	+ c\vn{A}_{4,[\gamma>0]}^\frac43\int_M |\nabla_{(3)}A|^4\,\gamma^{2s}\,d\mu
	+ c(\cg)^\frac43\bigg(\int_M |\nabla_{(3)}A|^3\,\gamma^{\frac{3s-2}{2}}\,d\mu\bigg)^\frac{4}{3}
\\
&\le c
	  \bigg[
		\int_M |\nabla_{(4)}A|^2\,\gamma^{s}\,d\mu
	  \bigg]^2
	+ (\theta + c\vn{A}_{4,[\gamma>0]}^\frac43)\int_M |\nabla_{(3)}A|^4\,\gamma^{2s}\,d\mu
\\&\qquad
	+ c(\cg)^\frac43
	  \bigg[
		\int_M |\nabla_{(3)}A|^4\,\gamma^{2s}\,d\mu
	  \bigg]^\frac23
	  \bigg[
		\int_M |\nabla_{(3)}A|^2\,\gamma^{s-4}\,d\mu
	  \bigg]^\frac23
\\
&\le c
	  \bigg[
		\int_M |\nabla_{(4)}A|^2\,\gamma^{s}\,d\mu
	  \bigg]^2
	+ (\theta + c\vn{A}_{4,[\gamma>0]}^\frac43)\int_M |\nabla_{(3)}A|^4\,\gamma^{2s}\,d\mu
\\&\qquad
	+ c(\cg)^4
	  \bigg[
		\int_M |\nabla_{(3)}A|^2\,\gamma^{s-4}\,d\mu
	  \bigg]^2
\,.
\end{align*}
Absorbing yields the estimate
\begin{align}
\int_M |\nabla_{(3)}A|^4\,\gamma^{2s}\,d\mu
&\le c
	  \bigg[
		\int_M |\nabla_{(4)}A|^2\,\gamma^{s}\,d\mu
	  \bigg]^2
	+ c\vn{A}_{4,[\gamma>0]}^\frac43\int_M |\nabla_{(3)}A|^4\,\gamma^{2s}\,d\mu
\notag\\&\qquad
	+ c(\cg)^4
	  \bigg[
		\int_M |\nabla_{(3)}A|^2\,\gamma^{s-4}\,d\mu
	  \bigg]^2
\,.
\label{EQcombined2}
\end{align}
Estimate \eqref{EQinducto} implies
\begin{align*}
(\cg)^2\int_M |\nabla_{(3)}A|^2\,\gamma^{s-4}\,d\mu
 &\le c\int_M |\nabla_{(4)}A|^2\,\gamma^{s}\,d\mu
 + c(\cg)^8\int_M |A|^2\,\gamma^{s-16}\,d\mu
\\
 &\le c\int_M |\nabla_{(4)}A|^2\,\gamma^{s}\,d\mu
 + c(\cg)^8\mu_\gamma(f)^\frac12\vn{A}_{4,[\gamma>0]}^2
\,.
\end{align*}
Combining this with \eqref{EQcombined2} we find
\begin{align}
\int_M |\nabla_{(3)}A|^4\,\gamma^{2s}\,d\mu
&\le c
	  \bigg[
		\int_M |\nabla_{(4)}A|^2\,\gamma^{s}\,d\mu
	  \bigg]^2
	+ c\vn{A}_{4,[\gamma>0]}^\frac43\int_M |\nabla_{(3)}A|^4\,\gamma^{2s}\,d\mu
\notag\\&\qquad\qquad
 + c(\cg)^{12}[(\cg)^4\mu_\gamma(f)]\vn{A}_{4,[\gamma>0]}^4
\notag\\
&\le c
	  \bigg[
		\int_M |\nabla_{(4)}A|^2\,\gamma^{s}\,d\mu
	+ c\vn{A}_{4,[\gamma>0]}^\frac23\bigg(\int_M |\nabla_{(3)}A|^4\,\gamma^{2s}\,d\mu\bigg)^\frac12
\notag\\&\qquad\qquad
 + c(\cg)^{6}[(\cg)^4\mu_\gamma(f)]^\frac12\vn{A}_{4,[\gamma>0]}^2
\bigg]^2
\,.
\label{EQcombined3}
\end{align}
We apply the auxilliary estimate \eqref{EQcombined3} to control the following
\begin{align*}
\int_M &|\nabla_{(3)}A|^2|A|^2\,\gamma^s\,d\mu
	+ \vn{A}_{4,[\gamma>0]}^2\int_M |\nabla_{(3)}A|^4\,\gamma^{2s}\,d\mu
 \le 2\vn{A}_{4,[\gamma>0]}^2
	\bigg(\int_M |\nabla_{(3)}A|^4\,\gamma^{2s}\,d\mu\bigg)^\frac12
\\
&\le c\vn{A}_{4,[\gamma>0]}^2
		\int_M |\nabla_{(4)}A|^2\,\gamma^{s}\,d\mu
	+ c\vn{A}_{4,[\gamma>0]}^\frac83\int_M |\nabla_{(3)}A|^4\,\gamma^{2s}\,d\mu
 + c(\cg)^{6}[(\cg)^4\mu_\gamma(f)]^\frac12\vn{A}_{4,[\gamma>0]}^4
\,.
\end{align*}
Combining the above with \eqref{EQcombined3}, and the lower order estimates
\eqref{EQcombinedx}, \eqref{EQcombined7}, and interpolating some terms, we finally conclude
\begin{align*}
\int_M &\big(
          |\nabla_{(3)}A|^2|A|^2 + |\nabla_{(2)}A|^2|A|^4 + |\nabla_{(2)}A|^2|\nabla A|^2
        + |\nabla A|^2|A|^6 + |A|^{10}\big)\gamma^s\,d\mu
\\&
  + \bigg(
    \int_M |\nabla A|^4\,\gamma^{\frac{2s}{3}}\,d\mu
 \bigg)^\frac32
	+ \vn{A}_{4,[\gamma>0]}^2\int_M |\nabla_{(3)}A|^4\,\gamma^{2s}\,d\mu
\\&\le
           (\theta + c\vn{A}_{4,[\gamma>0]}^4)\int_M
                 \big(
             |\nabla_{(4)}A|^2
           + |\nabla_{(2)}A|^2|\nabla A|^2
           + |\nabla_{(2)}A|^2|A|^4
	   + |\nabla A|^2|A|^6 + |A|^{10}
                 \big)
                 \gamma^s\,d\mu
\notag\\&\quad
  + \vn{A}_{4,[\gamma>0]}^6
     \bigg(
       \int_M |\nabla A|^4\,\gamma^\frac{2s}{3}\,d\mu
     \bigg)^\frac32
	+ c(\theta + \vn{A}_{4,[\gamma>0]}^\frac23)\vn{A}_{4,[\gamma>0]}^2\int_M |\nabla_{(3)}A|^4\,\gamma^{2s}\,d\mu
\notag\\&\quad
  + c(\cg)^6\vn{A}_{4,[\gamma>0]}^2
            \big(1+[(\cg)^4\mu_\gamma(f)]^\frac12
            \big)
            \big(1 + \vn{A}_{4,[\gamma>0]}^2
	    \big)
\,.
\end{align*}

Finally let us consider (v).
The estimate \eqref{EQnew1} has already been proved, it is the intermediate estimate \eqref{EQinter}.
For \eqref{EQnew2}, we note first that the equality \eqref{EQibpuse} implies the estimate
\begin{align*}
\int_M |\nabla A|^4\,\gamma^s\,d\mu
&\le
c\int_M |\nabla_{(2)}A|\,|\nabla A|^2\,|A|\,\gamma^s\,d\mu
+ c(\cg)\int_M 
        |\nabla_{(2)}A|\,|\nabla A|\,|A|^2
                            \,\gamma^{s-1}\,d\mu
			    \\
&\le
\frac14\int_M |\nabla A|^4\,\gamma^s\,d\mu
+ c\int_M |\nabla_{(2)}A|^2|A|^2\,\gamma^s\,d\mu
+ c(\cg)^2\int_M 
        |\nabla A|^2\,|A|^2
                            \,\gamma^{s-2}\,d\mu
			    \\
&\le
\frac12\int_M |\nabla A|^4\,\gamma^s\,d\mu
+ c\int_M |\nabla_{(2)}A|^2|A|^2\,\gamma^s\,d\mu
+ c(\cg)^4\int_M |A|^4 \,\gamma^{s-4}\,d\mu
\,.
\end{align*}
The final estimate \eqref{EQnew2} follows by absorption.
\end{proof}%}}}

\section*{Acknowledgements}

During the completion of this work the third author was supported by the ARC
Discovery Project grant DP120100097 and DP150100375 at the University of
Wollongong. Part of this work was completed during IMIA sponsored visits of the
first author at the University of Wollongong. Part of this work was completed
during visits of supported by the School of Mathematical Sciences of the second
and third authors.
We are grateful for this support.

We would also like to thank the anonymous referee for reading the manuscript
carefully and giving valuable feedback that has resulted in improvements made
to the article.

\bibliographystyle{plain}
\bibliography{mbib}

\end{document}